\documentclass[12pt]{article}

\usepackage{amsfonts,amsmath,amsthm,amssymb,amscd}
\usepackage[dvips]{graphicx}

\theoremstyle{definition}
\newtheorem{theorem}{Theorem}
\newtheorem{lemma}[theorem]{Lemma}
\newtheorem{proposition}[theorem]{Proposition}
\newtheorem{corollary}[theorem]{Corollary}

\numberwithin{equation}{section}
\numberwithin{theorem}{section}

\newcommand{\C}{\mathbb{C}}
\newcommand{\Z}{\mathbb{Z}}

\newcommand{\Res}{\textrm{Res}}
\begin{document}

\begin{center}
{\bf{\Large Rational solutions of \\
the Sasano systems of types $B_4^{(1)},$ $D_4^{(1)}$ and $D_5^{(2)}$}}
\end{center}

\begin{center}
By Kazuhide Matsuda
\end{center}

\begin{center}
Department of Engineering Science, Niihama National College of Technology,\\
7-1 Yagumo-chou, Niihama, Ehime, 792-8580, Japan. 
\end{center}

{\bf Abstract}
In this paper, we completely classify the rational solutions of the Sasano system of types $B_4^{(1)},$ $D_4^{(1)}$ and $D_5^{(2)},$ 
which are all given by coupled $P_{\rm III}$ systems and have the affine Weyl group symmetries of types $B_4^{(1)},$ $D_4^{(1)}$ and $D_5^{(2)}.$ 
The rational solutions are classified as one type by the B\"acklund transformation group. 

\section*{Introduction}
Paul Painlev\'e and his colleagues 
\cite{Painleve, Gambier} 
intended to find new transcendental functions defined by second order nonlinear differential equations. 
For the purpose, 
they investigated which second order ordinary differential equations of the form 
$$
y^{\prime\prime}=F(t,y,y^{\prime})
$$
where ${}^{\prime}=d/dt$ and $F$ is rational in $y,y^{\prime}$ and analytic in $t,$ 
have the property that 
the singularities other than poles of any of the solutions are independent of the particular solution and so are dependent only upon 
the equation. 
This property is called the Painlev\'e property. 
As a result,  they discovered the six non-linear differential equations, 
$P_{\rm I},$ $P_{\rm II},$ $P_{\rm III},$ $P_{\rm IV},$ $P_{\rm V},$ and $ P_{\rm VI}, $ 
which are called the Painlev\'e equations. 
\par
While generic solutions of the Painlev\'e equations are 
``new transcendental functions,'' 
there are special solutions which are expressible 
in terms of rational, algebraic or classical special functions. 
Especially, 
the rational solutions of $P_J \,\,({\rm J}={\rm II,III,IV,V,VI})$ 
were classified by 
Yablonski and Vorobev \cite{Yab:59,Vorob}, 
Gromak \cite{Gr:83,Gro}, 
Murata \cite{Mura1, Mura2}, 
Kitaev, Law and McLeod \cite{Kit-Law-McL}, 
Mazzoco \cite{Mazzo}, and 
Yuang and Li \cite{YuangLi}. 
Especially, 
Murata \cite{Mura1} 
classified the rational solutions of 
the second and fourth Painlev\'e equations 
by using the B\"acklund transformations, 
which transform a solution into another solution of the same equation 
with different parameters. 
\par
$P_J \,\,(\rm J=II,III,IV,V,VI)$ have the B\"acklund transformation group. 
It was shown by Okamoto \cite{oka1, oka2, oka3, oka4} that 
the B\"acklund transformation groups of the Painlev\'e equations except for $P_{\rm I}$ are isomorphic 
to the extended affine Weyl groups. 
For $ P_{\rm II},$ $P_{\rm III},$ $P_{\rm IV},$ $P_{\rm V},$ and $ P_{\rm VI}$, 
the B\"acklund transformation groups correspond to 
$A^{ ( 1 ) }_1,$ 
$A^{ ( 1 ) }_1 \bigoplus A^{ ( 1 ) }_1,$ 
$A^{ ( 1 ) }_2,$ 
$A^{ ( 1 ) }_3,$ 
and 
$D^{ ( 1 ) }_4$, 
respectively.
\par
Many fourth order Painlev\'e type equations have now been found. 
The examples are the Noumi and Yamada systems, Sasano systems, Fuji and Suzuki systems, etc. 
Our aim is to classify the rational solutions of all the fourth order Painlev\'e type equations. 
For this purpose, 
we will mainly use the residue calculus of their solutions and Hamiltonians. 
\par
Noumi and Yamada \cite{NoumiYamada-B} discovered the equations of type $A^{(1)}_l \,(l\geq 2),$ 
whose B\"acklund 
transformation group is isomorphic to the extended affine Weyl group $\tilde{W}(A^{(1)}_l)$. 
The Noumi and Yamada systems of types $A_2^{(1)}$ and $A_3^{(1)}$ correspond to the fourth and fifth 
Painlev\'e equations, respectively. 
Furthermore, 
we \cite{Matsuda1, Matsuda2} classified the rational solutions of the Noumi and Yamada systems of types $A_4^{(1)}$ and $A_5^{(1)},$ 
which are fourth order versions of the fourth and fifth Painlev\'e equations, respectively. 
\par
Sasano \cite{Sasano-1, Sasano-6}
obtained the coupled Painlev\'e III, V and VI systems from a higher dimensional
generalization of Okamoto's space of initial conditions, 
which have 
the affine Weyl group symmetries of types $B_4^{(1)},$ $D^{(1)}_5$ and $D_6^{(1)},$ 
and are called the Sasano systems of types $B_4^{(1)},$ $D^{(1)}_5$ and $D_6^{(1)},$ 
respectively. 
Moreover, he \cite{Sasano-2, Sasano-3, Sasano-4, Sasano-5}
obtained the equations of many different affine
Weyl group symmetries,  
which are all called the Sasano systems. 
Especially, for the Sasano systems of types $B_4^{(1)},$ $D_4^{(1)}$ and $D_5^{(2)},$ see \cite{Sasano-2, Sasano-6}. 
We \cite{Matsuda3, Matsuda4, Matsuda5, Matsuda6, Matsuda7} classified the rational solutions of the Sasano systems of types 
$A_5^{(2)},$ $A_4^{(2)},$ $A_1^{(1)},$ $D_3^{(2)},$ and $D_5^{(1)}.$ 
\par
Fuji and Suzuki \cite{Fuji-Suzuki-1, Fuji-Suzuki-2} obtained the equation of type $A_5^{(1)}$ from a similarity
reduction of the Drinfel'd-Sokolov hierarchy, 
which is called the Fuji and Suzuki system of type $A_5^{(1)}.$ 
Moreover, 
the Noumi and Yamada system of type $A_5^{(1)}$ is expected to be obtained from the degeneration of the 
Fuji and Suzuki system of type $A_5^{(1)}.$ 
\par
Following Oshima's \cite{Oshima} classification of the irreducible Fuchsian equation
with four accessory parameters, 
Sakai \cite{Sakai} obtained four ``source'' equations of all the 
fourth order Painlev\'e type equations, that is, 
the well-known Garnier system with two variables, 
the Fuji and Suzuki system of type $A_5^{(1)},$ 
the Sasano system of type $D_6^{(1)},$ 
and a new one. 
\par
Therefore, 
the Sasano system of type $D_6^{(1)}$ is obtained from the isomonodromic deformations of the Fuchsian type, 
which implies that from Miwa's theorem \cite{Miwa}, the system has the Painlev\'e property. 
Therefore, any other Sasano system, including the Sasano systems of types $B_4^{(1)},$ $D_4^{(1)}$ and $D_5^{(2)},$ 
is expected to be obtained by the degeneration from the Sasano system of type $D_6^{(1)}$ and have the Painleve property. 
\par
In this paper, 
we first classify the rational solutions of 
the Sasano system of type $B_4^{(1)},$ 
which is defined by 
\begin{equation*}
B_4^{(1)}(\alpha_j)_{0\leq j \leq 4}
\begin{cases}
\displaystyle 
t x^{\prime}=
2x^2y-x^2+(1-2\alpha_2-2\alpha_3-2\alpha_4)x+2\alpha_3z+2z^2w+t, \\
\displaystyle 
t y^{\prime}=
-2xy^2+2xy-(1-2\alpha_2-2\alpha_3-2\alpha_4)y+\alpha_1, \\
\displaystyle 
t z^{\prime}=
2z^2w-z^2+(1-2\alpha_4)z+2yz^2+t, \\
\displaystyle 
t w^{\prime}=
-2zw^2+2zw-(1-2\alpha_4)w-2\alpha_3y-4yzw+\alpha_3, \\
\alpha_0+\alpha_1+2\alpha_2+2\alpha_3+2\alpha_4=1.
\end{cases}
\end{equation*}
This system of ordinary differential equations is also expressed by the Hamiltonian system: 
\begin{equation*}
t \frac{dx}{dt}=\frac{\partial H_{B_4^{(1)}}}{\partial y}, \quad
t \frac{dy}{dt}=-\frac{\partial H_{B_4^{(1)}}}{\partial x},   \quad
t \frac{dz}{dt}=\frac{\partial H_{B_4^{(1)}}}{\partial w},  \quad 
t \frac{dw}{dt}=-\frac{\partial H_{B_4^{(1)}}}{\partial z},
\end{equation*}
where 
the Hamiltonian $H$ is given by 
\begin{alignat*}{5}
H_{B_4^{(1)}}=&
x^2y(y-1)& &+
x
\{
(1-2\alpha_2-2\alpha_3-2\alpha_4)y-\alpha_1
\} 
& 
&+ty \\
+&z^2w(w-1)& &+
z
\{
(1-2\alpha_4)w-\alpha_3
\} & 
&+tw
+
2yz(zw+\alpha_3).
\end{alignat*}

\par
Our main theorem is as follows.
\begin{theorem}
\label{them:B4}
{\it
For a rational solution of $B_4^{(1)}(\alpha_j)_{0\leq j \leq 4},$ 
by some B\"acklund transformations, 
the solution and parameters can be transformed so that 
\newline
\hspace{15mm}
$
x\equiv 0, \,y\equiv 1/2, \,
z=1/\{2\alpha_4\}\cdot t, \,
w\equiv 0
$
and 
$\alpha_0-\alpha_1=0, \,\alpha_3+\alpha_4=0, \alpha_4\neq 0,$ 
respectively. 
\par
Moreover, 
for $B_4^{(1)}(\alpha_j)_{0\leq j \leq 4},$ 
there exists a rational solution 
if and only if 
the parameters satisfy 
one of the following conditions:
\begin{align*}
&{\rm (1)} &  &\alpha_0-\alpha_1\in\Z, & &2\alpha_3+2\alpha_4\in \Z, & &\alpha_0-\alpha_1\equiv 2\alpha_3+2\alpha_4  & &\mathrm{ mod} \,\,2,   \\
%\newline
&{\rm (2)} & &\alpha_0-\alpha_1\in\Z, & &2\alpha_4\in \Z,               &  &\alpha_0-\alpha_1\equiv 2\alpha_4 & &\mathrm{ mod} \,\,2,  \\
%\newline
&{\rm (3)} & &\alpha_0+\alpha_1\in\Z, & &2\alpha_3+2\alpha_4\in \Z,  &  &\alpha_0+\alpha_1\equiv 2\alpha_3+2\alpha_4 & &\mathrm{ mod} \,\,2,  \\
%\newline
&{\rm (4)} & &\alpha_0+\alpha_1\in\Z, & &2\alpha_4\in \Z,               &   &\alpha_0+\alpha_1\equiv 2\alpha_4 & &\mathrm{ mod} \,\,2,  \\
&{\rm (5)} & &\alpha_0-\alpha_1\in\Z, & &\alpha_0+\alpha_1 \in \Z,   &    &\alpha_0-\alpha_1\not\equiv \alpha_0+\alpha_1 & &\mathrm{ mod} \,\,2,  \\
%\newline
&{\rm (6)}  &  &2\alpha_3\in\Z, & &2\alpha_4\in \Z,                      &     &2\alpha_3\equiv 1 & &\mathrm{ mod} \,\,2.
\end{align*}
}
\end{theorem}
In order to explain the method to prove our main theorem, 
let us define the coefficients of the Laurent series 
of $(x,y,z,w)$ at $t=\infty, \,\,0, \,\,c\in\C^{*}$ by 
\begin{equation*}
\begin{cases}
a_{\infty,k}, \,\,a_{0,k}, \,\,a_{c,k} \,\,(k\in\Z), \,\,
b_{\infty,k}, \,\,b_{0,k}, \,\,b_{c,k} \,\,(k\in\Z), \\
c_{\infty,k}, \,\,c_{0,k}, \,\,c_{c,k} \,\,(k\in\Z), \,\,
d_{\infty,k}, \,\,d_{0,k}, \,\,d_{c,k} \,\,(k\in\Z).
\end{cases}
\end{equation*}
For example, if $x,y,z,w$ all have a pole at $t=\infty,$ 
we define 
\begin{equation*}
\begin{cases}
x=a_{\infty, n_0}t^{n_0}+a_{\infty,n_0-1}t^{n_0-1}+\cdots+a_{\infty,0}+a_{\infty,-1}t^{-1}+\cdots,   \\
y=b_{\infty, n_1}t^{n_1}+b_{\infty,n_1-1}t^{n_1-1}+\cdots+b_{\infty,0}+b_{\infty,-1}t^{-1}+\cdots,   \\
z=c_{\infty, n_2}t^{n_2}+c_{\infty,n_2-1}t^{n_2-1}+\cdots+c_{\infty,0}+c_{\infty,-1}t^{-1}+\cdots, \\
w=d_{\infty, n_3}t^{n_3}+d_{\infty,n_3-1}t^{n_3-1}+\cdots+d_{\infty,0}+d_{\infty,-1}t^{-1}+\cdots,   
\end{cases}
\end{equation*}
where $n_0, \,n_1, \,n_2, \,n_3$ are all positive integers and $a_{\infty,n_0}b_{\infty,n_1}c_{\infty,n_2}d_{\infty,n_3}\neq 0.$ 
Moreover, 
we define 
the coefficients of the Laurent series of $H_{B_4^{(1)}}$ at $t=\infty, 0, c\in\mathbb{C}^{*}$ 
by 
$h_{\infty,k}, h_{0,k},h_{c,k} \,\,(k\in\mathbb{Z}),$ 
respectively.  
\par
This paper is organized as follows. 
In Section 1, 
for $B_4^{(1)}(\alpha_j)_{0\leq j \leq 4},$ 
we determine the meromorphic solutions near $t=\infty$ 
and prove that 
for a  meromorphic solution near $t=\infty,$ 
$z$ has a pole of order $n \,(n\geq 1)$ at $t=\infty$ 
and 
$x,y,w$ are all 
holomorphic at $t=\infty.$ 
Furthermore, 
we find that $a_{\infty,0}=\alpha_0-\alpha_1.$
\par
In Section 2, 
for $B_4^{(1)}(\alpha_j)_{0\leq j \leq 4},$ 
we deal with the meromorphic solutions near $t=0$ 
and find that $a_{0,0}=0, \alpha_0-\alpha_1.$
\par
In Section 3, 
for $B_4^{(1)}(\alpha_j)_{0\leq j \leq 4},$ 
we treat the meromorphic solutions near $t=c\in\mathbb{C}^{*}$ 
and 
show that for a rational solution of $B_4^{(1)}(\alpha_j)_{0\leq j \leq 4},$ 
$a_{\infty,0}-a_{0,0} \in\mathbb{Z}.$  
\par
In Section 4, 
for a meromorphic solution of $B_4^{(1)}(\alpha_j)_{0\leq j \leq 4}$ near $t=\infty, 0, c\in \mathbb{C}^{*},$ 
we study the Hamiltonian $H_{B_4^{(1)}}$ and 
compute $h_{\infty,0}, h_{0,0}, h_{c,-1},$ 
where $h_{\infty,0}, h_{0,0}$ are both expressed by the parameters and 
$h_{c,-1}$ is $nc \,(n\in\mathbb{Z}).$ Therefore, 
we find that for a rational solution of $B_4^{(1)}(\alpha_j)_{0\leq j \leq 4},$ 
the parameters satisfy $h_{\infty,0}-h_{0,0}\in\mathbb{Z}.$ 
\par
In Section 5, 
we define the B\"acklund transformations, 
$s_0,$ $s_1,$ $s_2,$ $s_3,$ $s_4,$ $\pi_1,$ $\pi_2$ and investigate their properties. 
We note that the B\"acklund transformation group 
$\langle s_0, s_1, s_2, s_3, s_4, \pi_1, \pi_2 \rangle$ is isomorphic to the affine Weyl group of type 
$B_4^{(1)},$ $\tilde{W}(B_4^{(1)}).$ 
\par
In Section 6, 
we treat the infinite solutions, 
that is, 
solutions such that some of $x,y,z,w$ are identically equal to $\infty.$
\par
In Section 7, 
we treat a rational solution such that 
$z$ has a pole of order one at $t=\infty$ 
and 
obtain the necessary conditions for $B_4^{(1)}(\alpha_j)_{0\leq j \leq 4}$ 
to have such rational solutions. 
For this purpose, we use the formulas, $a_{\infty,0}-a_{0,0}\in\mathbb{Z}$ and 
$h_{\infty,0}-h_{0,0}\in\mathbb{Z}.$ 
\par
In Section 8, 
we deal with a rational solution such that 
$z$ has a pole of order $n \,(n\geq 2)$ at $t=\infty$ 
and 
obtain necessary conditions for $B_4^{(1)}(\alpha_j)_{0\leq j \leq 4}$ 
to have such rational solutions. 
For this purpose, 
we mainly use the formula, $a_{\infty,0}-a_{0,0}\in\mathbb{Z}.$
\par
In Section 9, 
we summarize the necessary conditions in Sections 7 and 8 
and 
transform the parameters so that 
either of the following occurs:
$$
\mathrm{I}: \,\,\alpha_0-\alpha_1=0, \,\alpha_3+\alpha_4=0, \,\alpha_4\neq 0,\,\,\,\mathrm{II}: \,\,\alpha_0-\alpha_1=0,\,\alpha_3+\alpha_4=1/2.
$$
In this paper, cases I and II are called 
the standard forms I and II, respectively. 
For the standard form I, 
we determine the rational solution of Corollary \ref{coro:solution-1} in Section 1.
\par
In Section 10, we treat the standard form II. 
We then transform the parameters so that either of the following occurs:
$$
\mathrm{(1)}\,\,\alpha_0-\alpha_1=0, \,\alpha_3+\alpha_4=0, \,\alpha_4\neq 0,\,\,\,\mathrm{(2)}\,\,\alpha_0-\alpha_1=\alpha_2=0,\,\alpha_3+\alpha_4=1/2.
$$ 
\par
In Section 11, we deal with the case where $\alpha_0-\alpha_1=\alpha_2=0,\,\alpha_3+\alpha_4=1/2.$ 
We then transform the parameters so that either of the following occurs:
$$
\mathrm{(1)}\,\,\alpha_0-\alpha_1=0, \,\alpha_3+\alpha_4=0, \,\alpha_4\neq 0,\,\,\,\mathrm{(2)}\,\,\alpha_0=\alpha_1=\alpha_2=\,\alpha_3=0,\,\alpha_4=1/2.
$$ 
\par
In Section 12, 
we prove that for $B_4^{(1)}(0,0,0,0,1/2),$ there exists no rational solution. 
In Section 13, 
we prove our main theorem.
\par
In the appendix, 
following Sasano \cite{Sasano-2}, 
we introduce the Sasano systems of types $D_4^{(1)}$ and $D_5^{(2)}$ 
and show that 
the Sasano systems of types $B_4^{(1)},$ $D_4^{(1)}$ and $D_5^{(2)}$ are all 
equivalent by birational transformations. 
Using Theorem \ref{them:B4}, 
we classify the rational solutions of the Sasano systems of types $D_4^{(1)}$ and $D_5^{(2)}.$ 

\subsubsection*{Remark}
(1)\quad
We \cite{Matsuda3} classified the rational solutions of the Sasano system of type $A_5^{(2)},$ 
which is also given by the coupled $P_{\rm III}$ system. 
For this purpose, we did not have to consider the infinite solutions. 
For $B_4^{(1)}(\alpha_j)_{0\leq j \leq 4},$ we first treat the infinite solutions. 
\newline
(2)\quad
We \cite{Matsuda3} computed the meromorphic solutions and proved that they are deteimined by the differential equations. 
However, 
Proposition \ref{prop:inf-summary} shows that for $B_4^{(1)}(\alpha_j)_{0\leq j \leq 4},$ in some cases, 
the meromorphic solutions at $t=\infty$ are not unique.  
\newline
(3)\quad
Noted is that for $B_4^{(1)}(\alpha_j)_{0\leq j \leq 4},$
we use not only $a_{\infty,0}-a_{0,0}\in\mathbb{Z}$ but also $h_{\infty,0}-h_{0,0}\in\mathbb{Z}$ 
in order to obtain necessary conditions for $B_4^{(1)}(\alpha_j)_{0\leq j \leq 4}$ to have rational solutions. 
For the Sasano system of type $D_3^{(2)},$  
we \cite{Matsuda3} only used the residue calculus of the rational solutions for the necessary conditions.

\subsubsection*{Acknowledgments.}
The author thanks 
Professor Yousuke Ohyama 
for the careful guidance.

\section{Meromorphic solutions at $t=\infty$}

In this section, 
we determine the meromorphic solution at $t=\infty.$

\subsection{The case where $x,y,z,w$ are all holomorphic at $t=\infty$}

In this subsection, 
we deal with the case where 
$x,y,z,w$ are all holomorphic at $t=\infty.$

\begin{proposition}
{\it
For $B_4^{(1)}(\alpha_j)_{0\leq j \leq 4},$ 
there exists no solution such that 
$x,y,z,w$ are all holomorphic at $t=\infty.$ 
}
\end{proposition}

\begin{proof}
Comparing the coefficients of the term $t$ in 
\begin{equation*}
t x^{\prime}=
2x^2y-x^2+(1-2\alpha_2-2\alpha_3-2\alpha_4)x+2\alpha_3z+2z^2w+t, 
\end{equation*}
we can prove the proposition.
\end{proof}

\subsection{The case where one of $(x,y,z,w)$ has a pole at $t=\infty$}
In this subsection, 
we deal with the case in which 
one of $(x,y,z,w)$ has a pole at $t=\infty$ 
and 
consider the following four cases:
\newline
{\rm (1)}\quad $x$ has a pole at $t=\infty$ and $y,z,w$ are all holomorphic at $t=\infty,$
\newline
{\rm (2)}\quad $y$ has a pole at $t=\infty$ and $x,z,w$ are all holomorphic at $t=\infty,$
\newline
{\rm (3)}\quad $z$ has a pole at $t=\infty$ and $x,y,w$ are all holomorphic at $t=\infty,$
\newline
(4)\quad $w$ has a pole at $t=\infty$ and $x,y,z$ are all holomorphic at $t=\infty.$

\subsubsection{The case where $x$ has a pole at $t=\infty$}

\begin{proposition}
{\it
For $B_4^{(1)}(\alpha_j)_{0\leq j \leq 4},$ 
there exists no solution such that 
$x$ has a pole at $t=\infty$ and $y,z,w$ are all holomorphic at $t=\infty.$ 
}
\end{proposition}

\begin{proof}
Suppose that $B_4^{(1)}(\alpha_j)_{0\leq j \leq 4}$ has such s solution. 
We then note that $n_0\geq 1,$ $n_1, n_2, n_3\leq 0$ and $a_{\infty, n_0}\neq 0.$ 
Comparing the coefficients of the term $t^{2n_0}$ in 
$$
t x^{\prime}=
2x^2y-x^2+(1-2\alpha_2-2\alpha_3-2\alpha_4)x+2\alpha_3z+2z^2w+t, 
$$
we have $0=2a_{\infty,n_0}^2b_{\infty, 0}-a_{\infty, n_0}^2,$ which implies that $b_{\infty,0}=1/2.$ 
\par
On the other hand, 
by comparing the coefficients of the term $t^{n_0}$ in 
$$
t y^{\prime}=
-2xy^2+2xy-(1-2\alpha_2-2\alpha_3-2\alpha_4)y+\alpha_1, 
$$
we obtain $0=-2a_{\infty,n_0}b_{\infty, 0}^2+2a_{\infty,n_0}b_{\infty,0},$ 
which implies that $b_{\infty,0}=0,1.$ 
This is impossible. 
\end{proof}

\subsubsection{The case where $y$ has a pole at $t=\infty$}

\begin{lemma}
\label{leminf(y,w)-1}
{\it
Suppose that for $B_4^{(1)}(\alpha_j)_{0\leq j \leq 4},$ 
there exists a meromorphic solution near $t=\infty$ such that 
$y$ has a pole of order $n \,(n\geq 1)$ at $t=\infty$ 
and 
$x$ is holomorphic at $t=\infty.$ 
Then, 
$$
a_{\infty,0}=a_{\infty,-1}=\cdots=a_{\infty,-(n-1)}=0, \,\,
a_{\infty,-n}=\frac{-n-\alpha_0-\alpha_1}{2b_{\infty,n}}, \,\,b_{\infty,n}\neq 0,
$$
which implies that 
$$
2\alpha_3z+2z^2w+t
=
t x^{\prime}
-\left\{   
2x^2y-x^2+(1-2\alpha_2-2\alpha_3-2\alpha_4)x
\right\}=O(t^{-1}).
$$
}
\end{lemma}

\begin{proof}
Substituting the Laurent series of $x,y$ at $t=\infty$ in 
$$
t y^{\prime}=
-2xy^2+2xy-(1-2\alpha_2-2\alpha_3-2\alpha_4)y+\alpha_1, 
$$
we can obtain the lemma.

\end{proof}

By Lemma \ref{leminf(y,w)-1}, 
we can easily prove the following proposition:

\begin{proposition}
\label{prop:y-inf-sol}
{\it
For $B_4^{(1)}(\alpha_j)_{0\leq j \leq 4},$ 
there exists no solution such that 
$y$ has a pole at $t=\infty$ and $x,z,w$ are all holomorphic at $t=\infty.$
}
\end{proposition}

\subsubsection{The case where $z$ has a pole at $t=\infty$}

\begin{proposition}
\label{prop:z-inf}
{\it
Suppose that 
for $B_4^{(1)}(\alpha_j)_{0\leq j \leq 4},$ 
there exists a solution 
such that 
$z$ has a pole of order $n \,\,(n\geq 1)$ at $t=\infty$ 
and 
$x,y,w$ are all holomorphic at $t=\infty.$ 
\newline
{\rm (1)}\quad If $n\geq 2,$ then 
\begin{equation*}
\begin{cases}
x=\displaystyle(\alpha_0-\alpha_1)-\frac{\{(n-1)+2\alpha_2+2\alpha_3+2\alpha_4\}\{(n-1)+2\alpha_3+2\alpha_4\}}{c_{\infty,n}} t^{-n}+\cdots, \\
y=\displaystyle \frac12+\frac{(n-1)+2\alpha_3+2\alpha_4}{2c_{\infty,n}} t^{-n}+\cdots, \\
z=c_{\infty,n}t^n+c_{\infty,n-1}t^{n-1}+\cdots, \\
w=\displaystyle -\frac{\alpha_3}{c_{\infty,n}}t^{-n}+\cdots.
\end{cases}
\end{equation*}
{\rm (2)}\quad 
 If $n=1,$ then $\alpha_4\neq 0$ and 
\begin{equation*}
\begin{cases}
x=\displaystyle(\alpha_0-\alpha_1)-2\alpha_4(2\alpha_2+2\alpha_3+2\alpha_4)(2\alpha_3+2\alpha_4)t^{-1}\cdots, \\
y=\displaystyle \frac12 +2\alpha_4(\alpha_3+\alpha_4) t^{-1}+\cdots, \\
z=\displaystyle \frac{1}{2\alpha_4}t+\cdots, \\
w=\displaystyle 
-2\alpha_4(\alpha_3+\alpha_4)
t^{-1}+\cdots.
\end{cases}
\end{equation*}
}
\end{proposition}

\begin{proof}
It can be proved by direct calculation.
\end{proof}

From Proposition \ref{prop:z-inf}, 
let us consider the relationship between case (1) and 
the B\"acklund transformation, $s_3.$

\begin{corollary}
\label{coro:z(n)-s_3}
{\it
Suppose that 
for $B_4^{(1)}(\alpha_j)_{0\leq j \leq 4},$ 
there exists a solution 
such that 
$z$ has a pole of order $n \,\,(n\geq 2)$ at $t=\infty$ 
and 
$x,y,w$ are all holomorphic at $t=\infty.$ 
Moreover, 
assume that $\alpha_3\neq 0.$ 
$s_3(x,y,z,w)$ 
is then a solution of 
$B_4^{(1)}(\alpha_0,\alpha_1,\alpha_2+\alpha_3,-\alpha_3,\alpha_4+\alpha_3)$ 
such that 
$z$ has a pole of order one at $t=\infty$ 
and 
$x,y,w$ are all holomorphic at $t=\infty.$
}
\end{corollary}

\begin{proof}
By direct calculation, 
we find that 
$s_3(x,y,z,w)$ 
is a solution of 
$B_4^{(1)}(\alpha_0,\alpha_1,\alpha_2+\alpha_3,-\alpha_3,\alpha_4+\alpha_3)$ 
such that 
$s_3(z)$ has a pole of order $m \,\,(1\leq m \leq n-1)$ at $t=\infty$ 
and all of $s_3(x,y,w)$ are holomorphic at $t=\infty.$ 
\par
We assume that $s_3(z)$ has a pole of order $m \,\,(2 \leq m \leq n-1)$ at $t=\infty$ 
and show a contradiction. 
By the definition of $s_3,$ 
we see that 
$s_3(w)=-\alpha_3/c_{\infty,n} t^{-n}+\cdots.$ 
On the other hand, 
we observe that 
$$
s_3(z)=c^{\prime}_{\infty, m}t^{m}+\cdots, \,\,s_3(w)=-(-\alpha_3)/c^{\prime}_{\infty,m}t^{-m}+\cdots.
$$ 
It then follows that 
$$s_3(w)=-\alpha_3/c_{\infty,n} t^{-n}+\cdots=-(-\alpha_3)/c^{\prime}_{\infty,m}t^{-m}+\cdots,
$$ 
which is a contradiction.

\end{proof}

By Corollary \ref{coro:z(n)-s_3}, 
we can obtain the necessary conditions for $B_4^{(1)}(\alpha_j)_{0\leq j \leq 4}$ 
to have a solution such that 
$z$ has a pole of order $n \,(n\geq 2)$ at $t=\infty.$

\begin{corollary}
\label{coro:order-n}
{\it
Suppose that 
for $B_4^{(1)}(\alpha_j)_{0\leq j \leq 4},$ 
there exists a solution 
such that 
$z$ has a pole of order $n \,\,(n\geq 2)$ at $t=\infty$ 
and 
$x,y,w$ are all holomorphic at $t=\infty.$ 
One of the following then occurs: 
%\newline
{\rm (1)}\quad $\alpha_4=0,$
\quad
{\rm (2)}\quad $\alpha_0=\alpha_1=\alpha_2=\alpha_3=0, \,\,\alpha_4=1/2.$
}
\end{corollary}

\begin{proof}
Let us first prove that 
$\alpha_4=0$ if $\alpha_3\neq 0.$ 
For this purpose, 
we note that 
$s_3(x,y,z,w)$ is a solution of $B_4^{(1)}(\alpha_0,\alpha_1,\alpha_2+\alpha_3,-\alpha_3,\alpha_4+\alpha_3)$ 
such that 
$s_3(z)$ has a pole of order one at $t=\infty$ 
and all of $s_3(x,y,w)$ are holomorphic at $t=\infty.$ 
By direct calculation and Proposition \ref{prop:z-inf}, 
we then see that $\alpha_4+\alpha_3\neq0$ and 
$$
-\alpha_3/c_{\infty,n} \cdot t^{-n}+\cdots=s_3(w)=-2(\alpha_4+\alpha_3)\alpha_4t^{-1}+\cdots,
$$ 
which implies that $\alpha_4=0.$
\par
Let us show that $\alpha_4=0$ if $\alpha_3=0$ and $\alpha_2\neq 0.$ 
For this purpose, 
we note that 
$s_2(x,y,z,w)$ is a solution of $B_4^{(1)}(\alpha_0+\alpha_2,\alpha_1+\alpha_2,-\alpha_2,\alpha_2,\alpha_4)$ 
such that 
$s_2(z)$ has a pole of order $n$ at $t=\infty$ 
and all of $s_2(x,y,w)$ are holomorphic at $t=\infty.$ 
Based on the above discussion, it then follows that $\alpha_4=0$
\par
If $\alpha_2=\alpha_3=0$ and $\alpha_0\neq 0,$ or 
if $\alpha_2=\alpha_3=0$ and $\alpha_1\neq 0,$ 
we can show that $\alpha_4=0$ by using $s_0$ or $s_1.$ 
\par
The remaining case is that $\alpha_0=\alpha_1=\alpha_2=\alpha_3=0$ and $\alpha_4=1/2,$ 
which proves the corollary.

\end{proof}

Let us treat the case where 
$z$ has a pole of order one at $t=\infty.$

\begin{proposition}
\label{prop:uniqueness}
{\it
Suppose that 
for $B_4^{(1)}(\alpha_j)_{0\leq j \leq 4},$ 
there exists a solution 
such that 
$z$ has a pole of order one at $t=\infty$ 
and 
$x,y,w$ are holomorphic at $t=\infty.$ 
It is then unique.
}
\end{proposition}

\begin{proof}
Let us prove that 
the coefficients $a_{\infty,-2}, b_{\infty,-2}, c_{\infty,0}, d_{\infty,-2}$ are uniquely determined. 
The coefficients, $a_{\infty,-k}, b_{\infty,-k}, c_{\infty,-(k-2)}, d_{\infty,-k} \,\,(k=3,4,\ldots)$ 
can be computed in the same way. 
\par 
Comparing the constant terms in 
\begin{equation*}
t x^{\prime}=
2x^2y-x^2+(1-2\alpha_2-2\alpha_3-2\alpha_4)x+2\alpha_3z+2z^2w+t, 
\end{equation*}
we have 
\begin{equation}
\label{eqn:1-coeff}
d_{\infty,-2}=1/\{2c_{\infty,1}^2\}\cdot \{(2\alpha_3+4\alpha_4)c_{\infty,0}-(\alpha_0+\alpha_1)a_{\infty,0}\},
\end{equation} 
where $a_{\infty,0}, c_{\infty,1}$ both have been determined.  
Comparing the coefficients of the term $t^{-2}$ in 
\begin{equation*}
t y^{\prime}=
-2xy^2+2xy-(1-2\alpha_2-2\alpha_3-2\alpha_4)y+\alpha_1, 
\end{equation*}
we obtain 
\begin{equation}
\label{eqn:2-coeff}
a_{\infty,-2}=2(\alpha_0+\alpha_1-2)b_{\infty,-2}+4a_{\infty,0} b_{\infty,-1}^2,
\end{equation} 
where $b_{\infty,-1}$ has been determined.
Comparing the constant terms in 
\begin{equation*}
t z^{\prime}=
2z^2w-z^2+(1-2\alpha_4)z+2yz^2+t, 
\end{equation*}
we have 
$
b_{\infty,-2}+d_{\infty,-2}=1/\{2c_{\infty,1}^2\} \cdot \{(2\alpha_4-1)c_{\infty,0} \},$ 
which implies that 
\begin{equation}
\label{eqn:3-coeff}
b_{\infty,-2}=1/\{2c_{\infty,1}^2\} \cdot \{(-2\alpha_3-2\alpha_4-1)c_{\infty,0}+(\alpha_0+\alpha_1)a_{\infty,0} \}.
\end{equation}
Moreover, 
comparing  the coefficients of the term $t^{-2}$ in 
\begin{equation*}
t w^{\prime}=
-2zw^2+2zw-(1-2\alpha_4)w-2\alpha_3y-4yzw+\alpha_3,
\end{equation*}
we obtain 
\begin{equation}
\label{eqn:4-coeff}
2c_{\infty,0}d_{\infty,-1}^2 +(1+2\alpha_4)d_{\infty,-2}+(2\alpha_3+4\alpha_4)b_{\infty,-2}. 
\end{equation}
From (\ref{eqn:1-coeff}),  (\ref{eqn:3-coeff}) and (\ref{eqn:4-coeff}), 
we have 
\begin{equation*}
c_{\infty,0}=(\alpha_0+\alpha_1)(\alpha_0-\alpha_1)(\alpha_0+\alpha_1+2\alpha_2)/4\alpha_4^2,
\end{equation*}
which determines $a_{\infty,-2}, \,\,b_{\infty,-2}, \,\,d_{\infty,-2}.$ 
\end{proof}

%\begin{proposition}
%{\it
%Suppose that 
%for $B_4^{(1)}(\alpha_j)_{0\leq j \leq 4},$ 
%there exists a solution 
%such that 
%$z$ has a pole of order one at $t=\infty$ 
%and 
%$x,y,w$ are all holomorphic at $t=\infty.$ 
%Moreover, assume that $\alpha_0-\alpha_1=0.$ 
%Then, one of the following occurs:
%\newline
%(1)\quad $\alpha_0=\alpha_1=0,$
%\newline
%(2)\quad $\alpha_0=\alpha_1=1/2, $
%\newline
%(3)\quad $\alpha_0-\alpha_1=0, \,\,\alpha_3+\alpha_4=0.$
%}
%\end{proposition}

\subsubsection{The case where $w$ has a pole at $t=\infty$}

\begin{proposition}
{\it
For $B_4^{(1)}(\alpha_j)_{0\leq j \leq 4},$ 
there exists no solution such that 
$w$ has a pole at $t=\infty$ and $x,y,z$ are holomorphic at $t=\infty.$
}
\end{proposition}

\begin{proof}
It can be proved in the same way as Proposition \ref{prop:y-inf-sol}. 

\end{proof}

\subsection{The case where two of $(x,y,z,w)$ have a pole at $t=\infty$}

In this subsection, 
we consider the following four cases:
\newline
(1)\quad $x,y$ have a pole at $t=\infty$ and $z,w$ are holomorphic at $t=\infty,$
\newline
(2)\quad $x,z$ have a pole at $t=\infty$ and $y,w$ are holomorphic at $t=\infty,$
\newline
(3)\quad $x,w$ have a pole at $t=\infty$ and $y,z$ are holomorphic at $t=\infty,$
\newline
(4)\quad $y,z$ have a pole at $t=\infty$ and $x,w$ are holomorphic at $t=\infty,$ 
\newline
(5)\quad $y,w$ have a pole at $t=\infty$ and $x,z$ are holomorphic at $t=\infty,$ 
\newline
(6)\quad $z,w$ have a pole at $t=\infty$ and $x,y$ are holomorphic at $t=\infty.$

\subsubsection{The case where $x,y$ have a pole at $t=\infty$}

\begin{proposition}
{\it
For $B_4^{(1)}(\alpha_j)_{0\leq j \leq 4},$ 
there exists no solution such that 
$x,y$ have a pole at $t=\infty$ and $z,w$ are holomorphic at $t=\infty.$
}
\end{proposition}

\begin{proof}
Suppose that $B_4^{(1)}(\alpha_j)_{0\leq j \leq 4}$ has such a solution. We then note that $n_0, n_1\geq 1,$ $n_1, n_2\leq 0$ and $a_{\infty, n_0}, b_{\infty,n_1}\neq 0.$ 
\par
On the other hand, 
comparing the coefficients of the term $t^{n_0+2n_1}$ in 
$$
t y^{\prime}=
-2xy^2+2xy-(1-2\alpha_2-2\alpha_3-2\alpha_4)y+\alpha_1, 
$$
we have $0=-2a_{\infty,n_0}b_{\infty, n_1}^2,$ which is impossible. 
\end{proof}

\subsubsection{The case where $x,z$ have a pole at $t=\infty$}

\begin{proposition}
{\it
For $B_4^{(1)}(\alpha_j)_{0\leq j \leq 4},$ 
there exists no solution such that 
$x,z$ have a pole at $t=\infty$ and $y,w$ are holomorphic at $t=\infty.$
}
\end{proposition}

\begin{proof}
Suppose that $B_4^{(1)}(\alpha_j)_{0\leq j \leq 4}$ has such a solution. We then note that $n_0, n_2\geq 1,$ $n_1, n_3\leq 0,$ and 
$a_{\infty,n_0}, c_{\infty, n_2}\neq 0.$ 
\par
Comparing the coefficients of the term $t^{2n_2}$ and $t^{n_2}$ in 
\begin{equation*}
\begin{cases}
t z^{\prime}=
2z^2w-z^2+(1-2\alpha_4)z+2yz^2+t, \\
t w^{\prime}=
-2zw^2+2zw-(1-2\alpha_4)w-2\alpha_3y-4yzw+\alpha_3,
\end{cases}
\end{equation*}
we have $b_{\infty,0}+d_{\infty,0}=1/2$ and $-2d_{\infty,0}^2+2d_{\infty,0}-4b_{\infty,0}d_{\infty,0}=0,$ which implies that 
$b_{\infty,0}=1/2, \,d_{\infty,0}=0.$ 
\par
On the other hand, 
comparing the coefficients of the term $t^{n_0}$ in 
$$
t y^{\prime}=
-2xy^2+2xy-(1-2\alpha_2-2\alpha_3-2\alpha_4)y+\alpha_1, 
$$
we have $0=-2a_{\infty, n_0}b_{\infty, 0}^2+2a_{\infty,n_0}b_{\infty, 0},$ which implies that 
$b_{\infty,0}=0,1.$ This is impossible. 
\end{proof}

\subsubsection{The case where $x,w$ have a pole at $t=\infty$}

\begin{proposition}
{\it
For $B_4^{(1)}(\alpha_j)_{0\leq j \leq 4},$ 
there exists no solution such that 
$x,w$ have a pole at $t=\infty$ and $y,z$ are holomorphic at $t=\infty.$
}
\end{proposition}

\begin{proof}
It can be proved in the same way as Proposition \ref{prop:y-inf-sol}. 

\end{proof}

\subsubsection{The case where $y,z$ have a pole at $t=\infty$}

\begin{proposition}
{\it
For $B_4^{(1)}(\alpha_j)_{0\leq j \leq 4},$ 
there exists no solution such that 
$y,z$ have a pole at $t=\infty$ and $x,w$ are holomorphic at $t=\infty.$
}
\end{proposition}

\begin{proof}
Suppose that $B_4^{(1)}(\alpha_j)_{0\leq j \leq 4}$ has such a solution. 
We then note that $n_1, n_2\geq 1,$ $n_0, n_2\leq 0,$ and $b_{\infty, n_1}, c_{\infty, n_2}\neq 0.$ 
\par
Comparing the coefficients of the term $t^{n_1+2n_2}$ in
$$
t z^{\prime}=
2z^2w-z^2+(1-2\alpha_4)z+2yz^2+t,
$$
we have $0=2b_{\infty, n_1}c_{\infty, n_2}^2,$ which is impossible. 
\end{proof}

\subsubsection{The case where $y,w$ have a pole at $t=\infty$}

By Lemma \ref{leminf(y,w)-1}, 
we obtain the following lemma:

\begin{lemma}
\label{leminf(y,w)-2}
{\it
Suppose that 
for $B_4^{(1)}(\alpha_j)_{0\leq j \leq 4},$ 
there exists a solution such that 
$y,w$ both have a pole at $t=\infty$ and $x,z$ are both holomorphic at $t=\infty.$ 
$w$ then has a pole of order $n_3$ at $t=\infty,$ 
where $n_3$ is an odd number and $n_3\geq 3.$
}
\end{lemma}

By Lemma \ref{leminf(y,w)-2}, 
we can prove the following proposition:

\begin{proposition}
{\it
For $B_4^{(1)}(\alpha_j)_{0\leq j \leq 4},$ 
there exists no solution such that 
$y,w$ both have a pole at $t=\infty$ and $x,z$ are both holomorphic at $t=\infty.$
}
\end{proposition}

\begin{proof}
Let us assume that $w$ has a pole of order $n_3=3$ at $t=\infty.$ 
If $n_3=5,7,9,\ldots,$ 
the proposition can be proved in the same way. 
\par
By Lemma \ref{leminf(y,w)-1}, 
we find that $2z^2w+t=O(t^{-1})$ and 
$c_{\infty,0}=0, \,2c_{\infty,-1}^2d_{\infty,3}+1=0,$ 
which implies that $c_{\infty,-1}\neq 0.$ 
By considering that 
$$
t z^{\prime}=
2z^2w-z^2+(1-2\alpha_4)z+2yz^2+t, 
$$
we then observe that 
$y$ has a pole of order $n_1=1,2$ at $t=\infty.$ 
Thus, 
comparing the coefficients of the term $t^5$ in 
$$
t w^{\prime}=
-2zw^2+2zw-(1-2\alpha_4)w-2\alpha_3y-4yzw+\alpha_3, 
$$
we see that 
$-2c_{\infty,-1}d_{\infty,3}^2=0,$ 
which is impossible.
\end{proof}

\subsubsection{The case where $z,w$ have a pole at $t=\infty$}

\begin{proposition}
{\it
For $B_4^{(1)}(\alpha_j)_{0\leq j \leq 4},$ 
there exists no solution such that 
$z,w$ both have a pole at $t=\infty$ and $x,y$ are both holomorphic at $t=\infty.$
}
\end{proposition}

\begin{proof}
$B_4^{(1)}(\alpha_j)_{0\leq j \leq 4}$ has such a solution. We then note that $n_2, n_3\geq 1,$ $n_0, n_1\leq 0,$ 
and $c_{\infty, n_0}, d_{\infty, n_3}\neq 0.$ 
\par
Comparing the coefficients of the term $t^{2n_2+n_3}$ in 
$$
t z^{\prime}=
2z^2w-z^2+(1-2\alpha_4)z+2yz^2+t, 
$$ 
we have $0=2c_{\infty, n_2}^2d_{\infty, n_3},$ which is impossible. 
\end{proof}

\subsection{The case where three of $(x,y,z,w)$ have a pole at $t=\infty$}

In this subsection, 
we consider the following four cases:
\newline
(1)\quad $x,y,z$ all have a pole at $t=\infty$ and $w$ is holomorphic at $t=\infty,$
\newline
(2)\quad $x,y,w$ all have a pole at $t=\infty$ and $z$ is holomorphic at $t=\infty,$
\newline
(3)\quad $x,z,w$ all have a pole at $t=\infty$ and $y$ is holomorphic at $t=\infty,$
\newline
(4)\quad $y,z,w$ all have a pole at $t=\infty$ and $x$ is holomorphic at $t=\infty.$

\subsubsection{The case where $x,y,z$ have a pole at $t=\infty$}

\begin{proposition}
{\it
For $B_4^{(1)}(\alpha_j)_{0\leq j \leq 4},$ 
there exists no solution such that 
$x,y,z$ all have a pole at $t=\infty$ and $w$ is holomorphic at $t=\infty.$
}
\end{proposition}

\begin{proof}
It can be easily checked.

\end{proof}

\subsubsection{The case where $x,y,w$ have a pole at $t=\infty$}

\begin{proposition}
{\it
For $B_4^{(1)}(\alpha_j)_{0\leq j \leq 4},$ 
there exists no solution such that 
$x,y,w$ all have a pole at $t=\infty$ and $z$ is holomorphic at $t=\infty.$
}
\end{proposition}

\begin{proof}
It can be easily checked.

\end{proof}

\subsubsection{The case where $x,z,w$ have a pole at $t=\infty$}

\begin{proposition}
{\it
For $B_4^{(1)}(\alpha_j)_{0\leq j \leq 4},$ 
there exists no solution such that 
$x,z,w$ all have a pole at $t=\infty$ and $y$ is holomorphic at $t=\infty.$
}
\end{proposition}

\begin{proof}
It can be easily checked.

\end{proof}

\subsubsection{The case where $y,z,w$ have a pole at $t=\infty$}

\begin{proposition}
{\it
For $B_4^{(1)}(\alpha_j)_{0\leq j \leq 4},$ 
there exists no solution such that 
$y,z,w$ all have a pole at $t=\infty$ and $x$ is holomorphic at $t=\infty.$
}
\end{proposition}

\begin{proof}
It can be easily checked.

\end{proof}

\subsection{The case where all of $(x,y,z,w)$ have a pole at $t=\infty$}

\begin{proposition}
{\it
For $B_4^{(1)}(\alpha_j)_{0\leq j \leq 4},$ 
there exists no solution such that 
all of $(x,y,z,w)$ have a pole at $t=\infty.$
}
\end{proposition}

\begin{proof}
It can be easily checked.

\end{proof}

\subsection{Summary}

Let us summarize the results in this section.

\begin{proposition}
\label{prop:inf-summary}
{\it
Suppose that 
for $B_4^{(1)}(\alpha_j)_{0\leq j \leq 4},$ 
there exists a meromorphic solution at $t=\infty.$ 
$z$ then has a pole of order $n \,\,(n\geq 1)$ at $t=\infty$ 
and 
all of $x,y,w$ are holomorphic at $t=\infty.$ 
\newline
{\rm (1)}\quad If $n\geq 2,$ then 
\begin{equation*}
\begin{cases}
x=\displaystyle(\alpha_0-\alpha_1)-\frac{\{(n-1)+2\alpha_2+2\alpha_3+2\alpha_4\}\{(n-1)+2\alpha_3+2\alpha_4\}}{c_{\infty,n}} t^{-n}+\cdots, \\
y=\displaystyle \frac12+\frac{(n-1)+2\alpha_3+2\alpha_4}{2c_{\infty,n}} t^{-n}+\cdots, \\
z=c_{\infty,n}t^n+c_{\infty,n-1}t^{n-1}+\cdots, \\
w=\displaystyle -\frac{\alpha_3}{c_{\infty,n}}t^{-n}+\cdots,
\end{cases}
\end{equation*}
and 
either of the following occurs: 
{\rm (i)}\quad $\alpha_4=0,$ \,\,{\rm (ii)}\quad $\alpha_0=\alpha_1=\alpha_2=\alpha_3=0, \,\alpha_4=1/2.$ 
\newline
{\rm (2)}\quad 
 If $n=1,$ then $\alpha_4\neq 0$ and $x,y,z,w$ are uniquely expanded as follows: 
\begin{equation*}
\begin{cases}
x=\displaystyle(\alpha_0-\alpha_1)-2\alpha_4(2\alpha_2+2\alpha_3+2\alpha_4)(2\alpha_3+2\alpha_4)t^{-1}\cdots, \\
y=\displaystyle \frac12 +2\alpha_4(\alpha_3+\alpha_4)t^{-1}+\cdots, \\
z=\displaystyle \frac{1}{2\alpha_4} t+\cdots, \\
w=\displaystyle 
-2\alpha_4(\alpha_3+\alpha_4)
t^{-1}+\cdots.
\end{cases}
\end{equation*}
}
\end{proposition}

By the uniqueness in case (2), 
we can determine the rational solutions of $B_4^{(1)}(\alpha\j)_{0\leq j \leq 4}$ if $\alpha_0-\alpha_1=0, \,\,\alpha_3+\alpha_4=0$ and $\alpha_4\neq 0.$

\begin{corollary}
\label{coro:solution-1}
{\it
If $\alpha_0-\alpha_1=0, \,\,\alpha_3+\alpha_4=0$ and $\alpha_4\neq 0,$ 
then 
for $B_4^{(1)}(\alpha_j)_{0\leq j \leq 4},$ 
there exists a rational solution such that 
$$
x\equiv 0, \,\,y\equiv 1/2, \,\,z=1/(2\alpha_4)\cdot t, \,\,w\equiv 0,
$$
and it is unique. 
}
\end{corollary}

\begin{proof}
The corollary follows from the direct calculation and Proposition \ref{prop:inf-summary}.

\end{proof}

\subsubsection*{Remark}
Proposition \ref{prop:inf-summary} shows that 
if $z$ has a pole of order $n \,(n\geq 2)$ at $t=\infty,$ 
the meromorphic solutions at $t=\infty$ cannot be uniquely determined.

\section{Meromorphic solutions at $t=0$}

In this section, 
we treat the meromorphic solutions near $t=0.$ 
We then obtain the following proposition 
in the same way as Proposition \ref{prop:inf-summary}.

\begin{proposition}
\label{prop:t=0behavior}
{\it
Suppose that 
for $B_4^{(1)}(\alpha_j)_{0\leq j \leq 4},$ 
there exists a meromorphic solution at $t=0.$ 
One of the following then occurs:
\newline
{\rm (1)}\quad 
$x,y,z,w$ are all holomorphic at $t=0,$
\newline
{\rm (2)}\quad 
$z$ has a pole of order $n \,\,(n\geq 1)$ at $t=0$ 
and $x,y,w$ are all holomorphic at $t=0,$ 
\newline
{\rm (3)}\quad 
$y,w$ both have a pole at $t=0$ and 
$x,z$ are both holomorphic at $t=0.$
}
\end{proposition}

\subsection{The case where $x,y,z,w$ are all holomorphic at $t=0$}

\begin{proposition}
\label{prop:a0-determine}
{\it
Suppose that 
for $B_4^{(1)}(\alpha_j)_{0\leq j \leq 4},$ 
there exists a solution 
such that 
$x,y,z,w$ are all holomorphic at $t=0.$ 
Then, $a_{0,0}=0,\alpha_0-\alpha_1.$

}
\end{proposition}

\begin{proof}
Comparing the constant terms in 
\begin{equation*}
\begin{cases}
\displaystyle 
t x^{\prime}=
2x^2y-x^2+(1-2\alpha_2-2\alpha_3-2\alpha_4)x+2\alpha_3z+2z^2w+t, \\
\displaystyle 
t y^{\prime}=
-2xy^2+2xy-(1-2\alpha_2-2\alpha_3-2\alpha_4)y+\alpha_1, \\
\displaystyle 
t z^{\prime}=
2z^2w-z^2+(1-2\alpha_4)z+2yz^2+t, \\
\displaystyle 
t w^{\prime}=
-2zw^2+2zw-(1-2\alpha_4)w-2\alpha_3y-4yzw+\alpha_3, 
\end{cases}
\end{equation*}
we have 
\begin{align}
%\begin{cases}
&2a_{0,0}^2b_{0,0}-a_{0,0}^2+(1-2\alpha_2-2\alpha_3-2\alpha_4)a_{0,0}+2\alpha_3c_{0,0}+2c_{0,0}^2d_{0,0}=0,      \label{eqn:relation-1}      \\
&-2a_{0,0}b_{0,0}^2+2a_{0,0}b_{0,0}-(1-2\alpha_2-2\alpha_3-2\alpha_4)b_{0,0}+\alpha_1=0,    \label{eqn:relation-2}                  \\
&2c_{0,0}^2d_{0,0}-c_{0,0}^2+(1-2\alpha_4)c_{0,0}+2b_{0,0}c_{0,0}^2=0,               \label{eqn:relation-3}      \\
&-2c_{0,0}d_{0,0}^2+2c_{0,0}d_{0,0}-(1-2\alpha_4)d_{0,0}-2\alpha_3b_{0,0}-4b_{0,0}c_{0,0}d_{0,0}+\alpha_3=0,    \label{eqn:relation-4}        
%\end{cases}
\end{align}
respectively.
From (\ref{eqn:relation-1}) and (\ref{eqn:relation-2}), we see that 
\begin{equation}
\label{eqn:relation-5}  
a_{0,0}^2b_{0,0}+2\alpha_3b_{0,0}c_{0,0}+2b_{0,0}c_{0,0}^2d_{0,0}+\alpha_1a_{0,0}.     
\end{equation} 
From (\ref{eqn:relation-3}) and (\ref{eqn:relation-4}), we find that 
\begin{equation}
\label{eqn:relation-6}  
c_{0,0}^2d_{0,0}-2\alpha_3b_{0,0}c_{0,0}-2b_{0,0}c_{0,0}^2d_{0,0}+\alpha_3c_{0,0}=0.
\end{equation} 
From (\ref{eqn:relation-5}) and (\ref{eqn:relation-6}), 
we then have 
\begin{equation}
\label{eqn:relation-7}  
a_{0,0}^2b_{0,0}+\alpha_1a_{0,0}+c_{0,0}^2d_{0,0}+\alpha_3c_{0,0}=0.
\end{equation} 
By  (\ref{eqn:relation-1}) and (\ref{eqn:relation-7}), 
we obtain 
\begin{equation}
a_{0,0}\left\{a_{0,0}-(\alpha_0-\alpha_1)\right\}=0,
\end{equation} 
which implies that $a_{0,0}=0, \alpha_0-\alpha_1.$

\end{proof}

\subsubsection{The case where $a_{0,0}=0$}

Let us assume that $a_{0,0}=0$

\begin{proposition}
\label{prop:t=0holo-a_{0,0}=0}
{\it
Suppose that 
for $B_4^{(1)}(\alpha_j)_{0\leq j \leq 4},$ 
there exists a solution 
such that 
$x,y,z,w$ are all holomorphic at $t=0.$ 
Moreover, assume that $a_{0,0}=0.$ 
One of the following then occurs:
\newline
{\rm (1)}\quad $\alpha_0=\alpha_1=0$ and 
$a_{0,0}=c_{0,0}=0, \,\, 
2\alpha_3b_{0,0}+(2\alpha_2+2\alpha_3)d_{0,0}=\alpha_3,
$
\newline
{\rm (2)}\quad $\alpha_0=\alpha_1=0$ and 
$a_{0,0}=0, \,\,c_{0,0}d_{0,0}=-\alpha_3, \,\, 
2\alpha_3b_{0,0}-2\alpha_2d_{0,0}=\alpha_3,
$
\newline
{\rm (3)}\quad 
$\alpha_0=\alpha_1\neq 0, \,\,\alpha_4= 1/2$ and 
$
a_{0,0}=0, \,\,b_{0,0}=1/2, \,\,c_{0,0}=0,
$ 
\newline
{\rm (4)}\quad $\alpha_0+\alpha_1\neq 0, \,\,\alpha_3=0, \,\,\alpha_4=1/2$ 
and $a_{0,0}=0, \,\,b_{0,0}=\alpha_1/(\alpha_0+\alpha_1), \,\,c_{0,0}=0,$
\newline
{\rm (5)}\quad $\alpha_0+\alpha_1\neq 0, \,\,\alpha_4\neq 1/2$ 
and 
$$
a_{0,0}=0, \,\,b_{0,0}=\alpha_1/(\alpha_0+\alpha_1), \,\,c_{0,0}=0, \,\,d_{0,0}=-\alpha_3(\alpha_0-\alpha_1)/\{(2\alpha_4-1)(\alpha_0+\alpha_1)\},
$$
{\rm (6)}\quad 
$\alpha_0=\alpha_1\neq 0, \,\,\alpha_3+\alpha_4=1/2$ and 
$a_{0,0}=0, \,\,b_{0,0}=1/2, \,\,c_{0,0}d_{0,0}=-\alpha_3,$
\newline
{\rm (7)}\quad 
$\alpha_0+\alpha_1\neq 0, \,\,\alpha_0-\alpha_1\neq 0, \,\,$ and 
$$
a_{0,0}=0, \,\,
b_{0,0}=\frac{\alpha_1}{(\alpha_0+\alpha_1)}, \,\,
c_{0,0}=\frac{(\alpha_0+\alpha_1)(1-2\alpha_3-2\alpha_4)}{(\alpha_0-\alpha_1)}\neq 0, \,\,
d_{0,0}=-\frac{\alpha_3(\alpha_0-\alpha_1)}{ (\alpha_0+\alpha_1)(1-2\alpha_3-2\alpha_4) }.
$$
}
\end{proposition}

\begin{proof}
It can be easily checked.

\end{proof}

\subsubsection{The case where $a_{0,0}=\alpha_0-\alpha_1\neq0$}

Let us treat the case in which $a_{0,0}=\alpha_0-\alpha_1\neq0.$ 
We first obtain the following proposition:

\begin{proposition}
\label{prop:t=0holo-a_{0,0}nonzero-determine}
{\it
Suppose that 
for $B_4^{(1)}(\alpha_j)_{0\leq j \leq 4},$ 
there exists a solution 
such that 
all of $(x,y,z,w)$ are holomorphic at $t=0.$ 
Moreover, assume that $a_{0,0}=\alpha_0-\alpha_1\neq 0.$ 
Then, 
$b_{0,0}=1/2, \,\,-\alpha_1/(\alpha_0-\alpha_1).$ 
}
\end{proposition}

\begin{proof}
Substituting $a_{0,0}=\alpha_0-\alpha_1$ in (\ref{eqn:relation-2}), 
we can obtain $b_{0,0}=1/2, \,\,-\alpha_1/(\alpha_0-\alpha_1).$

\end{proof}

Let us deal with the case where $a_{0,0}=\alpha_0-\alpha_1\neq 0, \,b_{0,0}=1/2.$ 

\begin{proposition}
\label{prop:t=0holo-a_{0,0}nonzero-half}
{\it
Suppose that 
for $B_4^{(1)}(\alpha_j)_{0\leq j \leq 4},$ 
there exists a solution 
such that 
all of $(x,y,z,w)$ are holomorphic at $t=0.$ 
Moreover, assume that 
$a_{0,0}=\alpha_0-\alpha_1\neq 0$ and $b_{0,0}=1/2.$ 
One of the following then occurs:
\newline
{\rm (1)}\quad 
$\alpha_0+\alpha_1=0, \,\,\alpha_4=1/2$ and 
$a_{0,0}=\alpha_0-\alpha_1\neq 0, \,\,b_{0,0}=1/2,\,\,c_{0,0}=0,$
\newline
{\rm (2)}\quad 
$\alpha_0+\alpha_1=0$ and 
$a_{0,0}=\alpha_0-\alpha_1\neq 0, \,\,b_{0,0}=1/2, \,\,c_{0,0}=d_{0,0}=0,$
\newline
{\rm (3)}\quad 
$\alpha_0+\alpha_1=\alpha_2=0, \,\,\alpha_3+\alpha_4=1/2$ and 
$$
a_{0,0}=\alpha_0-\alpha_1\neq 0, \,\,b_{0,0}=1/2, \,\,
c_{0,0}\neq0, \,\,d_{0,0}=(2\alpha_4-1)/\{2c_{0,0}\},
$$
{\rm (4)}\quad 
$\alpha_3+\alpha_4\neq 1/2$ and 
$$
a_{0,0}=\alpha_0-\alpha_1, \,\,b_{0,0}=\frac12, \,\,
c_{0,0}=\frac{(\alpha_0+\alpha_1)(\alpha_0-\alpha_1)}{1-2\alpha_3-2\alpha_4}\neq 0, \,\,
d_{0,0}=-\frac{(1-2\alpha_4)(1-2\alpha_3-2\alpha_4)}{2(\alpha_0+\alpha_1)(\alpha_0-\alpha_1)}.
$$
}
\end{proposition}

\begin{proof}
It can be easily checked.

\end{proof}

Let us treat the case where $a_{0,0}=\alpha_0-\alpha_1\neq 0, \,b_{0,0}\neq1/2.$

\begin{proposition}
\label{prop:t=0holo-a_{0,0}nonzero-nonhalf}
{\it
Suppose that 
for $B_4^{(1)}(\alpha_j)_{0\leq j \leq 4},$ 
there exists a solution 
such that 
all of $(x,y,z,w)$ are holomorphic at $t=0.$ 
Moreover, assume that 
$a_{0,0}=\alpha_0-\alpha_1\neq 0$ and $b_{0,0}=-\alpha_1/(\alpha_0-\alpha_1)\neq 1/2,$ 
which implies that $\alpha_0+\alpha_1\neq 0.$  
One of the following then occurs:
\newline
{\rm (1)}\quad 
$\alpha_3=0, \,\,\alpha_4=1/2$ and 
$a_{0,0}=\alpha_0-\alpha_1\neq 0, \,\,b_{0,0}=-\alpha_1/(\alpha_0-\alpha_1), \,\,c_{0,0}=0,$
\newline
{\rm (2)}\quad 
$\alpha_4\neq 1/2$ and 
$$
a_{0,0}=\alpha_0-\alpha_1, \,\,b_{0,0}=-\alpha_1/(\alpha_0-\alpha_1), \,\,c_{0,0}=0, \,\,d_{0,0}=\alpha_3(\alpha_0+\alpha_1)/\{(1-2\alpha_4)(\alpha_0-\alpha_1)\}, 
$$
%\newline
{\rm (3)}\quad 
$\alpha_3=0$ and 
$a_{0,0}=\alpha_0-\alpha_1, \,\,b_{0,0}=-\alpha_1/(\alpha_0-\alpha_1), \,\,c_{0,0}\neq 0, \,\,d_{0,0}=0,$
\newline
{\rm (4)}\quad 
$\alpha_3\neq0$ and 
$$
a_{0,0}=\alpha_0-\alpha_1, \,\,
b_{0,0}=\frac{-\alpha_1}{\alpha_0-\alpha_1}, \,\,
c_{0,0}=\frac{(\alpha_0-\alpha_1)(1-2\alpha_3-2\alpha_4)}{\alpha_0+\alpha_1}\neq 0, \,\,
d_{0,0}=
\frac{-\alpha_3(\alpha_0+\alpha_1)}{(\alpha_0-\alpha_1)(1-2\alpha_3-2\alpha_4)}.
$$
}
\end{proposition}

\begin{proof}
It can be easily checked.

\end{proof}

\subsection{The case where $z$ has a pole at $t=0$}

\begin{proposition}
\label{prop:t=0-z}
{\it
Suppose that 
for $B_4^{(1)}(\alpha_j)_{0\leq j \leq 4},$ 
there exists a solution 
such that 
$z$ has a pole of order $n \,\,(n\geq 1)$ at $t=0$ 
and 
$x,y,w$ are all holomorphic at $t=0.$ 
Then, 
\begin{equation*}
\begin{cases}
\displaystyle x=(\alpha_0-\alpha_1)-\frac{\{(n+1)-2\alpha_3-2\alpha_4\}\{(n+1)-2\alpha_2-2\alpha_3-2\alpha_4\}}{c_{0,-n}} t^n+\cdots, \\
\displaystyle y=\frac12-\frac{(n+1)-2\alpha_3-2\alpha_4}{2c_{0,-n}} t^n +\cdots, \\
\displaystyle z=c_{0,-n}t^{-n}+c_{0,-(n-1)}t^{-(n-1)}+\cdots, \\
\displaystyle w=-\frac{\alpha_3}{c_{0,-n}} t^n+\cdots.
\end{cases}
\end{equation*}
}
\end{proposition}

\begin{proof}
It can be easily checked.

\end{proof}

Let us study the relationship between the solution in Proposition \ref{prop:t=0-z} and the B\"acklund transformation, $s_3.$

\begin{corollary}
\label{coro:t=0-z(n)-s_3}
{\it
Suppose that 
for $B_4^{(1)}(\alpha_j)_{0\leq j \leq 4},$ 
there exists a solution 
such that 
$z$ has a pole of order $n \,\,(n\geq 1)$ at $t=0$ 
and 
$x,y,w$ are all holomorphic at $t=0.$ 
Moreover, assume that $\alpha_3\neq 0.$ 
$s_3(x,y,z,w)$ is then a solution of $B_4^{(1)}(\alpha_0,\alpha_1,\alpha_2+\alpha_3,-\alpha_3,\alpha_4+\alpha_3)$ 
such that all of $s_3(x,y,z,w)$ are holomorphic at $t=0.$ 
}
\end{corollary}

\begin{proof}
By direct calculation, 
we see that 
$s_3(x,y,z,w)$ is a solution of $B_4^{(1)}(\alpha_0,\alpha_1,\alpha_2+\alpha_3,-\alpha_3,\alpha_4+\alpha3)$ 
such that $s_3(z)$ has a pole of order $m \,\,(0\leq m \leq n-1)$ at $t=0$ and all of $s_3(x,y,w)$ are holomorphic at $t=0.$ 
\par
We assume that $s_3(z)$ has a pole of order $m \,\,(1\leq m \leq n-1)$ at $t=0,$ and show contradiction. 
From Proposition \ref{prop:t=0-z}, it follows that 
$$
s_3(z)=c_{0,-m}^{\prime}t^{-m}+\cdots, \,\,   s_3(w)=-(-\alpha_3)/c^{\prime}_{0,-m}t^m+\cdots. 
$$ 
On the other hand, 
by the definition of $s_3,$ we find that $s_3(w)=-\alpha_3/c_{0,-n}t^n+\cdots,$ 
which is a contradiction.

\end{proof}

\subsection{The case where $y,w$ have a pole at $t=0$}

Let us prove the following four lemmas:

\begin{lemma}
\label{lem:t=0-(y,w)-1}
{\it
Suppose that 
for $B_4^{(1)}(\alpha_j)_{0\leq j \leq 4},$ 
there exists a solution 
such that 
$y$ has a pole of order $n \,\,(n\geq 1)$ at $t=0$ 
and $x$ is holomorphic at $t=0.$ 
Then, 
\begin{equation*}
a_{0,0}=a_{0,1}=\cdots=a_{0,(n-1)}=0, \,\,
a_{0,n}=\frac{(n-1)+2\alpha_2+2\alpha_3+2\alpha_4}{2b_{0,-n}},
\end{equation*}
which implies that 
\begin{equation*}
2\alpha_3z+2z^2w+t
=
t x^{\prime}-
\left[
2x^2y-x^2+(1-2\alpha_2-2\alpha_3-2\alpha_4)x
\right]
=O(t^2).
\end{equation*}
}
\end{lemma}

\begin{proof}
Substituting the Laurent series of $x,y$ at $t=0$ in 
\begin{equation*}
t y^{\prime}=
-2xy^2+2xy-(1-2\alpha_2-2\alpha_3-2\alpha_4)y+\alpha_1, 
\end{equation*}
we can prove the proposition.
\end{proof}

\begin{lemma}
\label{lem:t=0-(y,w)-2}
{\it
Suppose that 
for $B_4^{(1)}(\alpha_j)_{0\leq j \leq 4},$ 
there exists a solution 
such that 
$y,w$ both have a pole of order $n_1, n_3 \,\,(n_1, n_3\geq 1)$ at $t=0,$ 
and $x, z$ are both holomorphic at $t=0.$ 
$n_3$ is then an odd number.
}
\end{lemma}

\begin{proof}
We suppose that $n_3=2$ and show a contradiction. 
If $n_3=4, 6, \ldots,$ 
we can prove the contradiction in the same way. 
\par
Comparing the coefficients of the terms $t^{-2}, t^{0}$ in 
\begin{equation*}
t x^{\prime}=
2x^2y-x^2+(1-2\alpha_2-2\alpha_3-2\alpha_4)x+2\alpha_3z+2z^2w+t, 
\end{equation*}
we find that $c_{0,0}=c_{0,1}=0.$ 
Furthermore, 
by comparing the coefficients of the term $t$ in 
\begin{equation*}
t x^{\prime}=
2x^2y-x^2+(1-2\alpha_2-2\alpha_3-2\alpha_4)x+2\alpha_3z+2z^2w+t, 
\end{equation*}
we see that $0=1,$ which is impossible. 
\end{proof}

\begin{lemma}
\label{lem:t=0-(y,w)-3}
{\it
Suppose that 
for $B_4^{(1)}(\alpha_j)_{0\leq j \leq 4},$ 
there exists a solution 
such that 
$y,w$ both have a pole of order $n_1, n_3 \,\,(n_1, n_3\geq 1)$ at $t=0,$ 
and $x, z$ are both holomorphic at $t=0.$ 
Moreover, 
assume that $n_3=3,5,7,\ldots.$ 
Then, 
$n_1 \leq n_3.$
}
\end{lemma}

\begin{proof}
We treat the case where $n_3=3.$ The other cases can be proved in the same way. 
\par
Comparing the coefficients of the terms $t^{-3}, t^{-1}$ and $ t$ in 
\begin{equation*}
t x^{\prime}=
2x^2y-x^2+(1-2\alpha_2-2\alpha_3-2\alpha_4)x+2\alpha_3z+2z^2w+t, 
\end{equation*}
we observe that $c_{0,0}=c_{0,1}=0$ and $2c_{0,2}^2d_{0,-3}+1=0,$ 
which implies that $c_{0,2}\neq 0.$ 
\par
Comparing the lowest terms in
\begin{equation*}
t z^{\prime}=
2z^2w-z^2+(1-2\alpha_4)z+2yz^2+t, 
\end{equation*}
we find that $n_1=1,2,3.$ 
\end{proof}

\begin{lemma}
\label{lem:t=0-(y,w)-4}
{\it
Suppose that 
for $B_4^{(1)}(\alpha_j)_{0\leq j \leq 4},$ 
there exists a solution 
such that 
$y,w$ both have a pole at $t=0$ 
and 
$x,z$ are both holomorphic at $t=0.$ 
$y,w$ then both have a pole of order one at $t=0.$
}
\end{lemma}

\begin{proof}
We assume that $n_3=3$ and show a contradiction. 
If $n_3=5,7, \ldots,$ 
we can prove the contradiction in the same way. 
\par
By Lemma \ref{lem:t=0-(y,w)-3} and its proof, 
we see that $c_{0,0}=c_{0,1}=0, \,\,2c_{0,2}^2d_{0,-3}+1=0$ and $n_1=1,2,3.$ 
We note that $c_{0,2}\neq 0.$
\par
Let us first suppose that $n_1=1$ and $n_3=3.$ 
By comparing the coefficients of the terms $t$ in 
\begin{equation*}
t z^{\prime}=
2z^2w-z^2+(1-2\alpha_4)z+2yz^2+t, 
\end{equation*}
we then find that $2c_{0,2}^2d_{0,-3}+2b_{0,-3}c_{0,2}^2+1=0,$ 
which implies that $2b_{0,-3}c_{0,2}^2=0.$ 
This is impossible. 
\par
Let us suppose that $n_1=1,2$ and $n_3.$ 
By comparing the coefficients of the terms $t^{-4}$ in 
\begin{equation*}
t w^{\prime}=
-2zw^2+2zw-(1-2\alpha_4)w-2\alpha_3y-4yzw+\alpha_3, 
\end{equation*}
we then observe that $-2c_{0,2}d_{0,-3}^2=0,$ 
which is impossible.
\end{proof}

Let us then prove the following proposition:

\begin{proposition}
\label{prop:t=0-(y,w)}
{\it
Suppose that 
for $B_4^{(1)}(\alpha_j)_{0\leq j \leq 4},$ 
there exists a solution 
such that 
$y,w$ both have a pole at $t=0$ 
and 
$x,z$ are both holomorphic at $t=0.$ 
Then, $\alpha_4(\alpha_3+\alpha_4)\neq 0$ and 
\begin{equation*}
\begin{cases}
\displaystyle x=\frac{\alpha_2+\alpha_3+\alpha_4}{2\alpha_4(\alpha_3+\alpha_4)} t+\cdots, \\
\displaystyle y=2\alpha_4(\alpha_3+\alpha_4) t^{-1}+\cdots, \\
\displaystyle z=\frac{1}{2\alpha_4} t+\cdots, \\
\displaystyle w=-2\alpha_4(\alpha_3+\alpha_4)t^{-1}+\cdots.
\end{cases}
\end{equation*}

}
\end{proposition}

\begin{proof}
Let us first note that $y,w$ have a pole of order one at $t=0.$ 
From Lemma \ref{lem:t=0-(y,w)-1}, 
it then follows that $c_{0,0}=0.$ 
\par
Comparing the terms $t, t, t^{-1}$ in 
\begin{equation*}
\begin{cases}
\displaystyle 
t x^{\prime}=
2x^2y-x^2+(1-2\alpha_2-2\alpha_3-2\alpha_4)x+2\alpha_3z+2z^2w+t, \\
\displaystyle 
t z^{\prime}=
2z^2w-z^2+(1-2\alpha_4)z+2yz^2+t, \\
\displaystyle 
t w^{\prime}=
-2zw^2+2zw-(1-2\alpha_4)w-2\alpha_3y-4yzw+\alpha_3, \\
\end{cases}
\end{equation*}
we find that 
\begin{equation}
\label{eqn:relation -(y,w)}
\begin{cases}
2\alpha_3c_{0,1}+2c_{0,1}^2d_{0,-1}+1=0, \\
2c_{0,1}^2-2\alpha_4c_{0,1}+2b_{0,-1}c_{0,1}^2+1=0, \\
-2c_{0,1}d_{0,-1}^2+2\alpha_4d_{0,-1}-2\alpha_3b_{0,-1}-4b_{0,-1}c_{0,1}d_{0,-1}=0,
\end{cases}
\end{equation}
respectively. 
Based on the first equation of (\ref{eqn:relation -(y,w)}), 
we find that $c_{0,1}\neq 0.$ 
The first and second equations of (\ref{eqn:relation -(y,w)}) shows that 
$2(\alpha_3+\alpha_4)c_{0,1}-2b_{0,-1}c_{0,1}^2=0,$ 
which implies that 
$$
c_{0,1}=(\alpha_3+\alpha_4)/b_{0,-1}\neq 0.
$$  
Thus, from the first equation of (\ref{eqn:relation -(y,w)}), 
we find that 
$$
d_{0,-1}=-b_{0,-1}^2/\{2(\alpha_3+\alpha_4)^2\}-\alpha_3b_{0,-1}/(\alpha_3+\alpha_4).
$$
From the third equation of (\ref{eqn:relation -(y,w)}), 
we then see that 
$b_{0,-1}=2\alpha_4(\alpha_3+\alpha_4)\neq 0.$ 
\par
By Lemma \ref{lem:t=0-(y,w)-1}, 
we observe that
\begin{equation*} 
\begin{cases}
a_{0,0}=0, \,\,a_{0,1}=(\alpha_2+\alpha_3+\alpha_4)/\{2\alpha_4(\alpha_3+\alpha_4)\}, \\
b_{0,-1}=2\alpha_4(\alpha_3+\alpha_4), \\
c_{0,0}=0, \,\,c_{0,1}=1/2 \cdot \alpha_4, \\
d_{0,-1}=-2\alpha_4(\alpha_3+\alpha_4).
\end{cases}
\end{equation*}

\end{proof}

\subsection{Summary}

Let us summarize the results in this section.

\begin{proposition}
{\it
Suppose that 
for $B_4^{(1)}(\alpha_j)_{0\leq j \leq 4},$ 
there exists a meromorphic solution at $t=0.$ 
Then, 
$a_{0,0}=0, \,\,\alpha_0-\alpha_1$ 
and 
$y+w$ is holomorphic at $t=0.$

}
\end{proposition}

\section{Meromorphic solution at $t=c\in\C^{*}$}

In this section, 
we treat meromorphic solutions near $t=c\in\mathbb{C}^{*}.$ 
We then obtain the following proposition 
in the same way as Proposition \ref{prop:inf-summary}. 

\begin{proposition}
{\it
Suppose that 
for $B_4^{(1)}(\alpha_j)_{0\leq j \leq 4},$ 
there exists a meromorphic solution such that 
some of $(x,y,z,w)$ have a pole at $t=c\in\mathbb{C}^{*}.$ 
One of the following then occurs:
\newline
{\rm (1)}\quad $x$ has a pole of order one at $t=c$ and $y,z,w$ are all holomorphic at $t=c,$
\newline
{\rm (2)}\quad $y$ has a pole of order two at $t=c$ and $x,z,w$ are all holomorphic at $t=c,$
\newline
{\rm (3)}\quad $z$ has a pole of order $n \,(n\geq 1)$ at $t=c$ and $x,y,w$ are all holomorphic at $t=c,$
\newline
{\rm (4)}\quad $w$ has a pole of order two at $t=c$ and $y,z,w$ are all holomorphic at $t=c,$
\newline
{\rm (5)}\quad $x,z$ both have a pole of order one at $t=c$ and $y,w$ are both holomorphic at $t=c,$
\newline
{\rm (6)}\quad $x$ has a pole of order one at $t=c$ and $w$ has a pole of order two at $t=c$ and 
$y,z$ 
\newline
\hspace{4mm}
both 
are holomorphic at $t=c,$
\newline
{\rm (7)}\quad $y,w$ both have a pole at $t=c$ and $x,z$ are both holomorphic at $t=c.$
}
\end{proposition}

\subsection{The case where $x$ has a pole at $t=c\in\C^{*}$}

\begin{proposition}
\label{prop:t=c-x}
{\it
Suppose that 
for $B_4^{(1)}(\alpha_j)_{0\leq j \leq 4},$ 
there exists a solution such that 
$x$ has a pole of order one at $t=c\in\C^{*}$ 
and $y,z,w$ are all holomorphic at $t=c.$ 
Then, 
$b_{c,0}=0,1.$ 
\newline
{\rm (1)}\quad If $b_{c,0}=0,$ 
\begin{equation*}
\begin{cases}
x=c(t-c)^{-1}+\cdots, \\
y=\displaystyle -\frac{\alpha_1}{c}(t-c)+\cdots.
\end{cases}
\end{equation*}
{\rm (2)}\quad If $b_{c,0}=1,$ 
\begin{equation*}
\begin{cases}
x=-c(t-c)^{-1}+\cdots,   \\
y=\displaystyle 1+\frac{\alpha_0}{c}(t-c)+\cdots.
\end{cases}
\end{equation*}
}
\end{proposition}

\begin{proof}
It can be easily checked.

\end{proof}

\subsection{The case where $y$ has a pole at $t=c\in\C^{*}$}

\begin{proposition}
\label{prop:t=c-y}
{\it
Suppose that 
for $B_4^{(1)}(\alpha_j)_{0\leq j \leq 4},$ 
there exists a solution such that 
$y$ has a pole of order two at $t=c\in\C^{*}$ 
and $x,z,w$ are all holomorphic at $t=c.$ 
Then, $\alpha_3=0$ and one of the following occurs:
\begin{equation*}
\mathrm{(1)}
\begin{cases}
x=\displaystyle -(t-c)+\frac{(\alpha_0+\alpha_1)-2}{2c}(t-c)^2+\cdots, \\
y=-c(t-c)^{-2}+b_{c,0}+\cdots, \\
z=\displaystyle \frac12(t-c)+\cdots, \\
w=d_{c,2}(t-c)^2+\cdots,
\end{cases}
\end{equation*}
%%%%%%%%%%%%%%%%%%%
\begin{equation*}
\mathrm{(2)}
\begin{cases}
x=\displaystyle -(t-c)+\frac{(\alpha_0+\alpha_1)-2}{2c}(t-c)^2\cdots, \\
y=-c(t-c)^{-2}+b_{c,0}+\cdots, \\
z=-(t-c)+\cdots, \\
w=d_{c,2}(t-c)^2+\cdots.
\end{cases}
\end{equation*}
}
\end{proposition}

\begin{proof}
It can be easily checked.

\end{proof}

\subsection{The case where $z$ has a pole at $t=c\in\C^{*}$}

\begin{proposition}
\label{prop:t=c-z}
{\it
Suppose that 
for $B_4^{(1)}(\alpha_j)_{0\leq j \leq 4},$ 
there exists a solution such that 
$z$ has a pole of order $n \,\,(n\geq 1)$ at $t=c\in\C^{*}$ 
and $x,y,w$ are all holomorphic at $t=c.$ 
\newline
{\rm (1)}\quad If $n\geq 2,$
\begin{equation*}
\begin{cases}
\displaystyle y=\frac12+O((t-c)^{n-1})\cdots, \\
\displaystyle z=c_{0,-n}(t-c)^{-n}+\cdots, \\
\displaystyle w=-\frac{\alpha_3}{c_{0,-n}}(t-c)^n+\cdots.
\end{cases}
\end{equation*}
{\rm (2)}\quad If $n=1,$ then
\begin{equation*}
\begin{cases}
\displaystyle y=\left(\frac12-\frac{c}{2c_{c,-1}}\right)+\cdots, \\
\displaystyle z=c_{c,-1}(t-c)^{-1}+\cdots, \\
\displaystyle w=-\frac{\alpha_3}{c_{c,-1}}(t-c)+\cdots.
\end{cases}
\end{equation*}
}
\end{proposition}

\begin{proof}
It can be easily checked.

\end{proof}

\subsection{The case where $w$ has a pole at $t=c\in\C^{*}$}

\begin{proposition}
\label{prop:t=c-w}
{\it
Suppose that 
for $B_4^{(1)}(\alpha_j)_{0\leq j \leq 4},$ 
there exists a solution such that 
$w$ has a pole of order two at $t=c\in\C^{*}$ 
and $x,y,z$ are all holomorphic at $t=c.$ 
Then, 
\begin{equation*}
\begin{cases}
y=b_{c,0}+\cdots, \\
z=\displaystyle -(t-c)-\frac{2\alpha_4+1}{2c}(t-c)^2+c_{c,3}(t-c)^3+\cdots, \\
w=-c(t-c)^{-2}+d_{c,0}+\cdots,
\end{cases}
\end{equation*}
where 
$b_{c,0}+d_{c,0}=(1-c_{c,3}c)/2.$
}
\end{proposition}

\begin{proof}
It can be easily checked.

\end{proof}

\subsection{The case where $x,z$ have a pole at $t=c\in\C^{*}$}

\begin{proposition}
\label{prop:t=c-(x,z)}
{\it
Suppose that 
for $B_4^{(1)}(\alpha_j)_{0\leq j \leq 4},$ 
there exists a solution such that 
$x,z$ have a pole of order one at $t=c\in\C^{*}$ 
and $y,w$ are all holomorphic at $t=c.$ 
Then, 
$(b_{c,0},d_{c,0})=(0,0), (0,1), (1,0), (1,-1).$ 
\newline
{\rm (1)}\quad 
If $(b_{c,0},d_{c,0})=(0,0),$ 
\begin{equation*}
\begin{cases}
x=c(t-c)^{-1}+\cdots, \\
y=\displaystyle -\frac{\alpha_1}{c}(t-c)+\cdots, \\
z=c(t-c)^{-1}+\cdots, \\
w=\displaystyle -\frac{\alpha_3}{c}(t-c)+\cdots. 
\end{cases}
\end{equation*}
{\rm (2)}\quad 
Assume that $(b_{c,0},d_{c,0})=(0,1).$ 
Then, $(a_{c,-1}, c_{c,-1})=(-c,-c), (2c,-c).$ 
\par
{\rm (i)}\quad If $(a_{c,-1}, c_{c,-1})=(-c,-c),$ 
\begin{equation*}
\begin{cases}
x=-c(t-c)^{-1}+\cdots, \\
y=\displaystyle \frac{\alpha_1}{3c}(t-c)+\cdots, \\
z=-c(t-c)^{-1}+\cdots, \\
w=\displaystyle 1+\left(1-\frac{4\alpha_1}{3}-\alpha_3-2\alpha_4 \right)\frac{1}{c} (t-c)+\cdots. 
\end{cases}
\end{equation*}
\par
{\rm (ii)}\quad If $(a_{c,-1}, c_{c,-1})=(2c,-c),$ 
\begin{equation*}
\begin{cases}
x=2c(t-c)^{-1}+\cdots, \\
y=\displaystyle -\frac{\alpha_1}{3c}(t-c)+\cdots, \\
z=-c(t-c)^{-1}+\cdots, \\
w=\displaystyle 1+\left(1+\frac{4\alpha_1}{3}-\alpha_3-2\alpha_4 \right)\frac{1}{c} (t-c)+\cdots. 
\end{cases}
\end{equation*}
{\rm (3)}\quad 
If $(b_{c,0},d_{c,0})=(1,0),$ 
\begin{equation*}
\begin{cases}
x=-c(t-c)^{-1}+\cdots, \\
y=\displaystyle 1+\frac{\alpha_0}{c}(t-c)+\cdots, \\
z=-c(t-c)^{-1}+\cdots, \\
w=\displaystyle \frac{\alpha_3}{c}(t-c)+\cdots. 
\end{cases}
\end{equation*}
{\rm (4)}\quad 
Assume that $(b_{c,0},d_{c,0})=(1,-1).$ 
Then, $(a_{c,-1}, c_{c,-1})=(c,c), (-2c,c).$ 
\par
(i)\quad If $(a_{c,-1}, c_{c,-1})=(c,c),$ 
\begin{equation*}
\begin{cases}
x=c(t-c)^{-1}+\cdots, \\
y=\displaystyle 1-\frac{\alpha_0}{3c}(t-c)+\cdots, \\
z=c(t-c)^{-1}+\cdots, \\
w=\displaystyle -1+\left(-1+\frac{4\alpha_0}{3}+\alpha_3+2\alpha_4 \right)\frac{1}{c} (t-c)+\cdots. 
\end{cases}
\end{equation*}
\par
(ii)\quad If $(a_{c,-1}, c_{c,-1})=(-2c,c),$ 
\begin{equation*}
\begin{cases}
x=-2c(t-c)^{-1}+\cdots, \\
y=\displaystyle 1+\frac{\alpha_0}{3c}(t-c)+\cdots, \\
z=c(t-c)^{-1}+\cdots, \\
w=\displaystyle -1+\left(-1-\frac{4\alpha_0}{3}+\alpha_3+2\alpha_4 \right)\frac{1}{c} (t-c)+\cdots. 
\end{cases}
\end{equation*}

}
\end{proposition}

\begin{proof}
It can be easily checked.

\end{proof}

\subsection{The case where $x,w$ have a pole at $t=c\in\C^{*}$}

\begin{proposition}
\label{prop:t=c-(x,w)}
{\it
Suppose that 
for $B_4^{(1)}(\alpha_j)_{0\leq j \leq 4},$ 
there exists a solution such that 
$x,w$ have a pole at $t=c\in\C^{*}$ 
and $y,z$ are all holomorphic at $t=c.$ 
Then, $b_{c,0}=0,1.$ 
\newline
{\rm (1)}\quad If $b_{c,0}=0,$ 
\begin{equation*}
\begin{cases}
x=\displaystyle c(t-c)^{-1}+\frac{1+\alpha_0-\alpha_1}{2}+\cdots, \\
y=\displaystyle -\frac{\alpha_1}{c}(t-c)+\cdots, \\
z=\displaystyle -(t-c)-\frac{2\alpha_4+1}{2c}(t-c)^2+c_{c,3}(t-c)^3+\cdots, \\
w=\displaystyle -c(t-c)^{-2} \,\,+\frac{1-c_{c,3} c}{2}+\cdots.
\end{cases}
\end{equation*}
%%%%%%%%%%%%%%%%%%\newline
{\rm (2)}\quad If $b_{c,0}=1,$ 
\begin{equation*}
\begin{cases}
x=\displaystyle -c(t-c)^{-1}+\frac{-1+\alpha_0-\alpha_1}{2}+\cdots, \\
y=\displaystyle 1+\frac{\alpha_0}{c}(t-c)+\cdots, \\
z=\displaystyle -(t-c)-\frac{2\alpha_4+1}{2c}(t-c)^2+c_{c,3}(t-c)^3+\cdots, \\
w=\displaystyle -c(t-c)^{-2}+\frac{-1-c_{c,3} c}{2}+\cdots. 
\end{cases}
\end{equation*}
}
\end{proposition}

\begin{proof}
It can be easily checked.

\end{proof}

\subsection{The case where $y,w$ have a pole at $t=c\in\C^{*}$}

%\begin{lemma}
%{\it
%Suppose that 
%for $B_4(\alpha_j)_{0\leq j \leq 4},$ 
%there exists a solution such that 
%$y$ has a pole order $n \,\,(n\geq 2)$ at $t=c\in\C^{*}$ 
%and 
%$x$ is holomorphic at $t=c.$
%Then, 
%\begin{equation*}
%a_{c,0}=a_{c,1}=\cdots =a_{c,n-2}=0, \,\,a_{c,n-1}=\frac{nc}{2b_{c,-n}}.
%\end{equation*}
%Thus, if $n\geq 2,$ 
%\begin{equation*}
%2\alpha_3 z+2z^2w +t
%=t x^{\prime}-
%\left[
%2x^2y-x^2+(1-2\alpha_2-2\alpha_3-2\alpha_4)x
%\right]
%=O(1).
%\end{equation*}
%Especially, 
%if $n\geq 3,$
%\begin{equation*}
%2\alpha_3 z+2z^2w +t
%=t x^{\prime}-
%\left[
%2x^2y-x^2+(1-2\alpha_2-2\alpha_3-2\alpha_4)x
%\right]
%=O(t-c).
%\end{equation*}

%}
%\end{lemma}

%\begin{lemma}
%{\it
%Suppose that 
%for $B_4(\alpha_j)_{0\leq j \leq 4},$ 
%there exists a solution such that 
%$y$ has a pole order $n_1 \,\,(n_1\geq 1)$ at $t=c\in\C^{*}$ 
%and 
%$w$ has a pole order $n_3 \,\,(n_3\geq 1)$ at $t=c\in\C^{*}$ 
%and 
%$x,z$ are holomorphic at $t=c\in\C^{*}.$ 
%Then, 
%$n_1\leq 2.$

%}
%\end{lemma}

%\begin{lemma}
%{\it
%Suppose that 
%for $B_4(\alpha_j)_{0\leq j \leq 4},$ 
%there exists a solution such that 
%$y$ has a pole order two at $t=c\in\C^{*}$ 
%and 
%$w$ has a pole order $n_3 \,\,(n_3\geq 1)$ at $t=c\in\C^{*}$ 
%and 
%$x,z$ are holomorphic at $t=c\in\C^{*}.$ 
%Then, 
%$n_3=1,2.$
%}
%\end{lemma}

%\begin{lemma}
%{\it
%Suppose that 
%for $B_4(\alpha_j)_{0\leq j \leq 4},$ 
%there exists a solution such that 
%$y$ has a pole order one at $t=c\in\C^{*}$ 
%and 
%$w$ has a pole order $n_3 \,\,(n_3\geq 1)$ at $t=c\in\C^{*}$ 
%and 
%$x,z$ are holomorphic at $t=c\in\C^{*}.$ 
%Then, 
%$n_3=1.$
%}
%\end{lemma}

In the same way as Proposition \ref{prop:t=0-(y,w)}, 
we can prove the following proposition:

\begin{proposition}
\label{prop:t=c-(y,w)(2,1)}
{\it
Suppose that 
for $B_4(\alpha_j)_{0\leq j \leq 4},$ 
there exists a solution such that 
$y$ has a pole of order two at $t=c\in\C^{*}$ 
and 
$w$ has a pole of order one at $t=c$ 
and 
$x,z$ are holomorphic at $t=c.$ 
Then, $c_{c,0}=0$ and $c_{c,1}=1/2, -1.$
\newline
{\rm (1)}\quad If $c_{c,1}=1/2,$
\begin{equation*}
\begin{cases}
x=\displaystyle -(t-c)+\frac{\alpha_0+\alpha_1-2\alpha_3-2}{2c}(t-c)^2+\cdots, \\
y=\displaystyle -c(t-c)^{-2}+\frac{2\alpha_3}{3}(t-c)^{-1}+\cdots, \\
z=\displaystyle \frac12(t-c)+\frac{-\alpha_4+1}{4c}(t-c)^2+\cdots, \\
w=\displaystyle -\frac{2\alpha_3}{3}(t-c)^{-1}+\frac{\alpha_3(4\alpha_3+3\alpha_4+3)}{9c}+\cdots.
\end{cases}
\end{equation*}
{\rm (2)}\quad If $c_{c,1}=-1,$
\begin{equation*}
\begin{cases}
x=\displaystyle -(t-c)+\frac{\alpha_0+\alpha_1+2\alpha_3-2}{2c}(t-c)^2+\cdots, \\
y=\displaystyle -c(t-c)^{-2}-\frac{2\alpha_3}{3}(t-c)^{-1}+\cdots, \\
z=\displaystyle -(t-c)+\frac{-2\alpha_4-1}{2c}(t-c)^2+\cdots, \\
w=\displaystyle \frac{2\alpha_3}{3}(t-c)^{-1}+\frac{\alpha_3(\alpha_3-3\alpha_4-3)}{9c}+\cdots.
\end{cases}
\end{equation*}

}
\end{proposition}

\begin{proposition}
\label{prop:t=c-(y,w)(2,2)}
{\it
Suppose that 
for $B_4(\alpha_j)_{0\leq j \leq 4},$ 
there exists a solution such that 
$y,w$ both have a pole of order two at $t=c\in\C^{*}$ 
and 
$x,z$ are holomorphic at $t=c.$ 
Then, 
$c_{c,0}=0$ and 
$c_{c,1}=-1/2,1.$ 
\newline
{\rm (1)}\quad If $c_{c,1}=-1/2,$
\begin{equation*}
\begin{cases}
x=\displaystyle (t-c)+\frac{1+2\alpha_2+4\alpha_3+6\alpha_4}{2c}(t-c)^2+\cdots, \\
y=\displaystyle c(t-c)^{-2}-\frac{2\alpha_3+4\alpha_4}{3}(t-c)^{-1}+\cdots, \\
z=\displaystyle -\frac12(t-c)-\frac{\alpha_4+1}{4c}(t-c)^2+\cdots, \\
w=\displaystyle -4c(t-c)^{-2}+\frac{2\alpha_3+4\alpha_4}{3}(t-c)^{-1}+\cdots.
\end{cases}
\end{equation*}
{\rm (2)}\quad If $c_{c,1}=1,$
\begin{equation*}
\begin{cases}
x=\displaystyle (t-c)+\frac{1+2\alpha_2-2\alpha_4}{2c}(t-c)^2+\cdots, \\
y=\displaystyle c(t-c)^{-2}+\frac{2\alpha_3+4\alpha_4}{3}(t-c)^{-1}+\cdots, \\
z=\displaystyle (t-c)+\frac{1-2\alpha_4}{2c}(t-c)^2+\cdots, \\
w=\displaystyle -c(t-c)^{-2}-\frac{2\alpha_3+4\alpha_4}{3}(t-c)^{-1}+\cdots.
\end{cases}
\end{equation*}

}
\end{proposition}

\begin{proposition}
\label{prop:t=c-(y,w)(1,1)}
{\it
Suppose that 
for $B_4(\alpha_j)_{0\leq j \leq 4},$ 
there exists a solution such that 
$y,w$ both have a pole of order one at $t=c\in\C^{*}$ 
and 
$x,z$ are holomorphic at $t=c\in\C^{*}.$ 
Then, 
$c_{c,0}\neq0$ and 
\begin{equation*}
\begin{cases}
x=\displaystyle \frac{c}{2b_{c,-1}}+a_{c,1}(t-c)+\cdots, \\
y=\displaystyle b_{c,-1}(t-c)^{-1}+\cdots, \\
z=\displaystyle \frac{c}{2b_{c,-1}}+c_{c,1}(t-c)+\cdots, \\
w=\displaystyle d_{c,-1}(t-c)^{-1}+\cdots, \\
\end{cases}
\end{equation*}
where $b_{c,-1}+d_{c,-1}=0$ and $a_{c,1}-c_{c,1}=2\alpha_2c_{c,0}.$

}
\end{proposition}

\subsection{Summary}

\begin{proposition}
\label{rational-necessary-1}
{\it
{\rm (1)}\quad 
Suppose that 
for $B_4^{(1)}(\alpha_j)_{0\leq j \leq 4},$ 
there exists a meromorphic solution at $t=c\in\C^{*}.$ 
{\rm (i)} and {\rm (ii)} then hold:
\newline
{\rm (i)}\quad $x$ 
has a pole of order at most one at $t=c$
and $\Res_{t=c} x=nc, n\in\Z,$
\newline
{\rm (ii)}\quad $y+w$ 
has a pole of order at most two at $t=c$
and $\Res_{t=c} (y+w)=0.$
\newline
{\rm (2)}\quad 
Suppose that 
for $B_4^{(1)}(\alpha_j)_{0\leq j \leq 4},$ 
there exists a rational solution. 
Then, 
\begin{equation*}
a_{\infty,0}-a_{0,0}\in\Z \,\,{\it and} \,\,
b_{\infty,-1}+d_{\infty,-1}=0,
\end{equation*}
because $b_{0,-1}+d_{0,-1}=b_{c,-1}+d_{c,-1}=0.$

}
\end{proposition}

\begin{proof}
Case (1) is obvious. 
Case (2) can be proved by applying the residue theorem to $t^{-1}x.$

\end{proof}

\section{Hamiltonian and its properties}

\subsection{The Laurent series of $H_{B_4^{(1)}}$ at $t=\infty$}

By Proposition \ref{prop:inf-summary}, 
we can calculate the constant term, $h_{\infty,0},$ of the Laurent series of $H_{B_4^{(1)}}$ at $t=\infty.$ 

\begin{proposition}
\label{prop:hinf}
{\it
Suppose that 
for $B_4^{(1)}(\alpha_j)_{0\leq j \leq 4},$ 
there exists a solution such that 
$z$ has a pole order $n \,\,(n\geq 1)$ at $t=\infty$ and 
$x,y,w$ are all holomorphic at $t=\infty.$ 
\newline
{\rm (1)}\quad 
If $n\geq 2,$ then 
$$
h_{\infty,0}=\frac14(\alpha_0-\alpha_1)^2+\alpha_3(\alpha_3+2\alpha_4-1).
$$
{\rm (2)}\quad 
If $n=1,$ then
$$
h_{\infty,0}=\frac14(\alpha_0-\alpha_1)^2+(\alpha_3+\alpha_4)^2-(\alpha_3+\alpha_4).
$$

}
\end{proposition}

\subsection{The Laurent series of $H_{B_4^{(1)}}$ at $t=0$}

In this subsection, 
we compute the constant term, $h_{0,0},$ of the Laurent series of $H_{B_4^{(1)}}$ at $t=0.$

\subsubsection{The case where all of $(x,y,z,w)$ are holomorphic at $t=0$}

By Propositions \ref{prop:t=0holo-a_{0,0}=0}, \ref{prop:t=0holo-a_{0,0}nonzero-half} 
and \ref{prop:t=0holo-a_{0,0}nonzero-nonhalf}, 
we obtain the following proposition:

\begin{proposition}
\label{prop:hzero-1}
{\it
Suppose that 
for $B_4^{(1)}(\alpha_j)_{0\leq j \leq 4},$ 
there exists a solution such that 
all of $(x,y,z,w)$ are holomorphic at $t=0.$ 
Then, $a_{0,0}=0, \alpha_0-\alpha_1.$ 
\newline
{\rm (1)-(i)}\quad 
If $a_{0,0}=0$ and $c_{0,0}=0,$ 
\begin{equation*}
h_{0,0}=0.
\end{equation*}
{\rm (1)-(ii)}\quad 
If $a_{0,0}=0$ and $c_{0,0}\neq0,$
\begin{equation*}
h_{0,0}=\alpha_3(\alpha_3+2\alpha_4-1).
\end{equation*}
{\rm (2)}\quad Moreover, assume that $a_{0,0}=\alpha_0-\alpha_1\neq 0.$ 
Then, $b_{0,0}=1/2, -\alpha_1/(\alpha_0-\alpha_1).$ 
\newline
{\rm (2)-(i)}\quad 
If $b_{0,0}=1/2$ and $c_{0,0}=0,$ 
\begin{equation*}
h_{0,0}=\frac14(\alpha_0-\alpha_1)^2.
\end{equation*}
{\rm (2)-(ii)}\quad 
If $b_{0,0}=1/2$ and $c_{0,0}\neq0,$ 
\begin{equation*}
h_{0,0}=\frac14(\alpha_0-\alpha_1)^2-\frac14(2\alpha_4-1)^2.
\end{equation*}
{\rm (2)-(iii)}\quad 
If $b_{0,0}=-\alpha_1/(\alpha_0-\alpha_1)\neq 1/2$ and $c_{0,0}=0,$ 
\begin{equation*}
h_{0,0}=-\alpha_0\alpha_1.
\end{equation*}
{\rm (2)-(iv)}\quad 
If $b_{0,0}=-\alpha_1/(\alpha_0-\alpha_1)\neq 1/2$ and $c_{0,0}\neq0,$ 
\begin{equation*}
h_{0,0}=-\alpha_0\alpha_1+\alpha_3(\alpha_3+2\alpha_4-1).
\end{equation*}
}
\end{proposition}

\subsubsection{The case where $z$ has a pole at $t=0$}

By Proposition \ref{prop:t=0-z}, 
we obtain the following proposition:

\begin{proposition}
\label{prop:hzero-2}
{\it
Suppose that 
for $B_4^{(1)}(\alpha_j)_{0\leq j \leq 4},$ 
there exists a solution such that 
$z$ has a pole of order $n \,\,(n\geq 1)$ at $t=0$ 
and 
$x,y,w$ are all holomorphic at $t=0.$ 
Then, 
\begin{equation*}
h_{0,0}=\frac14(\alpha_0-\alpha_1)^2+\alpha_3(\alpha_3+2\alpha_4-1).
\end{equation*}
}
\end{proposition}

\subsubsection{The case where $y,w$ have a pole at $t=0$}

By Proposition \ref{prop:t=0-(y,w)}, 
we have the following proposition:

\begin{proposition}
\label{prop:hzero-3}
{\it
Suppose that 
for $B_4^{(1)}(\alpha_j)_{0\leq j \leq 4},$ 
there exists a solution such that 
$y,w$ have a pole at $t=0$ 
and 
$x,z$ are holomorphic at $t=0.$ 
Then, 
\begin{equation*}
h_{0,0}=\alpha_2(\alpha_0+\alpha_1+\alpha_2).
\end{equation*}
}
\end{proposition}

\subsection{The Laurent series of $H_{B_4^{(1)}}$ at $t=c\in\C^{*}$}

In this subsection, 
we calculate the residue, $h_{c,-1},$ of $H_{B_4^{(1)}}$ at $t=c\in\mathbb{C}^{*}.$

\subsubsection{The case where $x$ has a pole at $t=c\in \C^{*}$}

By Proposition \ref{prop:t=c-x}, 
we can show the following proposition:

\begin{proposition}
{\it
Suppose that 
for $B_4^{(1)}(\alpha_j)_{0\leq j \leq 4},$ 
there exists a solution such that 
$x$ has a pole at $t=c\in\C^{*}$ 
and 
$y,z,w$ are all holomorphic at $t=c.$ 
$H_{B_4^{(1)}}$ is then holomorphic at $t=c.$
}
\end{proposition}

\subsubsection{The case where $y$ has a pole at $t=c\in \C^{*}$}

By Proposition \ref{prop:t=c-y}, 
we can prove the following proposition:

\begin{proposition}
{\it
Suppose that 
for $B_4^{(1)}(\alpha_j)_{0\leq j \leq 4},$ 
there exists a solution such that 
$y$ has a pole of order two at $t=c\in\C^{*}$ 
and $x,z,w$ are all holomorphic at $t=c.$ 
Then, $\alpha_3=0$ 
and $H_{B_4^{(1)}}$ has a pole of order one at $t=c$ 
and $h_{c,-1}=\mathrm{Res}_{t=c} H_{B_4^{(1)}} =c.$
}
\end{proposition}

\subsubsection{The case where $z$ has a pole at $t=c\in \C^{*}$}

By Proposition \ref{prop:t=c-z}, 
we can show the following proposition:

\begin{proposition}
{\it
Suppose that for $B_4^{(1)}(\alpha_j)_{0\leq j \leq 5},$ 
there exists a solution such that 
$z$ has a pole of order $n \,\,(n\geq 1)$ at $t=c\in\C^{*}$ 
and $x,y,w$ are all holomorphic at $t=c.$ 
$H_{B_4^{(1)}}$ is then holomorphic at $t=c.$
}
\end{proposition}

\subsubsection{The case where $w$ has a pole at $t=c\in \C^{*}$}

By Proposition \ref{prop:t=c-w}, 
we can prove the following proposition:

\begin{proposition}
{\it
Suppose that for $B_4^{(1)}(\alpha_j)_{0\leq j \leq 5},$ 
there exists a solution such that 
$w$ has a pole of order two at $t=c\in\C^{*}$ 
and $x,y,z$ are all holomorphic at $t=c.$ 
$H_{B_4^{(1)}}$ then has a pole of order one at $t=c$ and 
$h_{c,-1}=\mathrm{Res}_{t=c} H_{B_4^{(1)}} =c.$
}
\end{proposition}

\subsubsection{The case where $x,z$ have a pole at $t=c\in \C^{*}$}

By Proposition \ref{prop:t=c-(x,z)}, 
we have the following proposition:

\begin{proposition}
{\it
Suppose that for $B_4^{(1)}(\alpha_j)_{0\leq j \leq 5},$ 
there exists a solution such that 
 $x,z$ have a pole of order one at $t=c\in \C^{*}$ 
and 
$y,w$ are holomorphic at $t=c.$ 
Then, 
$(b_{c,0},d_{c,0})=(0,0), (0,1), (1,0), (1,-1).$ 
\newline
{\rm (1)}\quad 
If $(b_{c,0},d_{c,0})=(0,0), $ 
$H_{B_4^{(1)}}$ is holomorphic at $t=c.$
\newline
{\rm (2)}\quad If $(b_{c,0},d_{c,0})=(0,1),$ 
$H_{B_4^{(1)}}$ is holomorphic at $t=c.$
\newline
{\rm (3)}\quad If $(b_{c,0},d_{c,0})=(1,0),$ 
$H_{B_4^{(1)}}$ is holomorphic at $t=c.$
\newline
{\rm (4)}\quad If $(b_{c,0},d_{c,0})=(1,-1),$ 
$H_{B_4^{(1)}}$ is holomorphic at $t=c.$
}
\end{proposition}

\subsubsection{The case where $x,w$ have a pole at $t=c\in \C^{*}$}

By Proposition \ref{prop:t=c-(x,w)}, 
we have the following proposition:

\begin{proposition}
{\it
Suppose that for $B_4^{(1)}(\alpha_j)_{0\leq j \leq 5},$ 
there exists a solution such that 
 $x,w$ have a pole at $t=c\in \C^{*}$ 
and 
$y,z$ are holomorphic at $t=c.$ 
$H_{B_4^{(1)}}$ then has a pole of order one at $t=c$ 
and 
$h_{c,-1}=\mathrm{Res}_{t=c} H_{B_4^{(1)}} =c.$
}
\end{proposition}

\subsubsection{The case where $y,w$ have a pole at $t=c\in \C^{*}$}

By Proposition \ref{prop:t=c-(y,w)(2,1)}, 
we obtain the following proposition:

\begin{proposition}
{\it
Suppose that for $B_4^{(1)}(\alpha_j)_{0\leq j \leq 5},$ 
there exists a solution such that 
 $y$ has a pole order two at $t=c\in \C^{*}$ 
and 
$w$ has a pole of order one at $t=c$ 
and 
$x,z$ are holomorphic at $t=c.$ 
$H_{B_4^{(1)}}$ then has a pole of order one at $t=c$ 
and 
$h_{c,-1}=\mathrm{Res}_{t=c} H_{B_4^{(1)}} =c.$
}
\end{proposition}

By Proposition \ref{prop:t=c-(y,w)(2,2)}, 
we have the following proposition:

\begin{proposition}
{\it
Suppose that for $B_4^{(1)}(\alpha_j)_{0\leq j \leq 5},$ 
there exists a solution such that 
$y,w$ have a pole order two at $t=c\in \C^{*}$ 
and 
$x,z$ are holomorphic at $t=c.$ 
Then, 
$c_{c,1}=-1/2, 1.$ 
\newline
{\rm (1)}\quad If $c_{c,1}=-1/2, $ $H_{B_4^{(1)}}$ has a pole of order one at $t=c$ and $h_{c,-1}=\mathrm{Res}_{t=c} H_{B_4^{(1)}} =3c.$
\newline
{\rm (2)}\quad If $c_{c,1}=1, $ 
$H_{B_4^{(1)}}$ is holomorphic at $t=c.$
}
\end{proposition}

By Proposition \ref{prop:t=c-(y,w)(1,1)}, 
we have the following proposition:

\begin{proposition}
{\it
Suppose that for $B_4^{(1)}(\alpha_j)_{0\leq j \leq 5},$ 
there exists a solution such that 
$y,w$ both have a pole order one at $t=c\in \C^{*}$ 
and 
$x,z$ are holomorphic at $t=c.$ 
$H_{B_4^{(1)}}$ is then holomorphic at $t=c.$
}
\end{proposition}

\subsection{Summary}

\begin{proposition}
\label{rational-necessary-2}
{\it
{\rm (1)}\quad 
Suppose that for $B_4^{(1)}(\alpha_j)_{0\leq j \leq 5},$ 
there exists a meromorphic solution at $t=c\in\C^{*}.$ 
$H_{B_4^{(1)}}$ then has a pole of order at most one at $t=c$ and 
$h_{c,-1}=\mathrm{Res}_{t=c} H_{B_4^{(1)}} =nc \,\,(n=0,1,3).$
\newline
{\rm (2)}\quad 
Suppose that for $B_4^{(1)}(\alpha_j)_{0\leq j \leq 5},$ 
there exists a rational solution. 
Then, 
$h_{\infty,0}-h_{0,0}\in\Z.$
}
\end{proposition}

\begin{proof}
Case (1) is obvious. 
Case (2) can be proved by applying the residue theorem to $t^{-1}H_{B_4^{(1)}}$

\end{proof}

\section{B\"acklund transformations and their properties}
In this section, following Sasano \cite{Sasano-1}, 
we introduce the B\"acklund transformations $s_0, s_1, s_2, s_3, s_4, \pi_1, $ and $\pi_2$ 
and 
investigate their properties.

\subsection{Definition of the B\"acklund transformations}
$B_4^{(1)}(\alpha_j)_{0\leq j \leq 4}$ 
has 
B\"acklund transformations, 
$s_0, s_1, s_2, s_3, s_4, \pi_1, $ and $\pi_2,$ 
which are defined by 
\begin{align*}
&s_0: \,\,
(x,y,z,w,t;\alpha_0,\alpha_1,\alpha_2,\alpha_3,\alpha_4)
\rightarrow
\left(
x+\frac{\alpha_0}{y-1},y,z,w,t;
-\alpha_0,\alpha_1,\alpha_2+\alpha_0,\alpha_3,\alpha_4
\right),              \\
&s_1: \,\,
(x,y,z,w,t;\alpha_0,\alpha_1,\alpha_2,\alpha_3,\alpha_4)
\rightarrow
\left(
x+\frac{\alpha_1}{y},y,z,w,t; 
\alpha_0,-\alpha_1,\alpha_2+\alpha_1,\alpha_3,\alpha_4
\right),                \\
&s_2: \,\,
(x,y,z,w,t;\alpha_0,\alpha_1,\alpha_2,\alpha_3,\alpha_4)
\rightarrow          \\
&\hspace{50mm}
\left(
x,y-\frac{\alpha_2}{x-z},z,w+\frac{\alpha_2}{x-z},t;
\alpha_0+\alpha_2,\alpha_1+\alpha_2,-\alpha_2,\alpha_3+\alpha_2,\alpha_4
\right),   \\
&s_3: \,\,
(x,y,z,w,t;\alpha_0,\alpha_1,\alpha_2,\alpha_3,\alpha_4)
\rightarrow
\left(
x,y,z+\frac{\alpha_3}{w},w,t;\alpha_0,\alpha_1,\alpha_2+\alpha_3,-\alpha_3,\alpha_4+\alpha_3
\right),  \\
&s_4: \,\,
(x,y,z,w,t;\alpha_0,\alpha_1,\alpha_2,\alpha_3,\alpha_4)
\rightarrow
\left(
x,y,z,w-\frac{2\alpha_4}{z}+\frac{t}{z^2},-t;
\alpha_0,\alpha_1,\alpha_2,\alpha_3+2\alpha_4,-\alpha_4
\right),        \\
&\pi_1: \,\,
(x,y,z,w,t;\alpha_0,\alpha_1,\alpha_2,\alpha_3,\alpha_4)
\rightarrow
\left(
-x,1-y,-z,-w,-t;
\alpha_1,\alpha_0,\alpha_2,\alpha_3,\alpha_4
\right),  \\
&\pi_2 : \,\,
(x,y,z,w,t;\alpha_0,\alpha_1,\alpha_2,\alpha_3,\alpha_4)
\rightarrow \\
&\hspace{50mm}
\left(
\frac{t}{z}, 
-\frac{z}{t}(zw+\alpha_3), 
\frac{t}{x}, 
-\frac{x}{t}(xy+\alpha_1), 
t; 
2\alpha_4+\alpha_3, \alpha_3,\alpha_2,  \alpha_1, \frac{\alpha_0-\alpha_1}{2}
\right).
\end{align*}
Let us note that the B\"acklund transformations, except for $\pi_1,$ are not polynomial in $x,y,z,w$ and $t$ but rational in $x,y,z,w$ and $t.$
%The B\"acklund transformation group $\langle s_0, s_1, s_2, s_3, s_4, \pi_1, \pi_2 \rangle$ is isomorphic to the extended affine Weyl group $\tilde{W}(B_4^{(1)}).$
\newline
{\bf Remark}
At first, 
Sasano \cite{Sasano-1} 
defined 
$$
\pi_2: (\alpha_0,\alpha_1,\alpha_2,\alpha_3,\alpha_4)
\longrightarrow
(2\alpha_4+\alpha_3, \alpha_3,\alpha_2, (\alpha_0-\alpha_1)/2,\alpha_1).
$$
But, 
following Sasano \cite{Sasano-7}, 
we corrected and redefined it  
$$
\pi_2: (\alpha_0,\alpha_1,\alpha_2,\alpha_3,\alpha_4)
\longrightarrow
(2\alpha_4+\alpha_3, \alpha_3,\alpha_2, \alpha_1,(\alpha_0-\alpha_1)/2).
$$

\subsection{The properties of the B\"acklund transformations}

Considering $s_0, s_1, s_2, s_3,$ 
we can prove the following proposition:

\begin{proposition}
\label{prop:backlund}
{\it
{\rm (0)}\quad 
If $y\equiv 1$ for $B_4^{(1)}(\alpha_j)_{0\leq j \leq 4},$ 
then $\alpha_0=0.$
\newline
{\rm (1)}\quad 
If $y\equiv 0$ for $B_4^{(1)}(\alpha_j)_{0\leq j \leq 4},$ then $\alpha_1=0.$
\newline
{\rm (2)}\quad 
If $x\equiv z$ for $B_4^{(1)}(\alpha_j)_{0\leq j \leq 4},$ then $\alpha_2=0.$
\newline
{\rm (3)}\quad 
If $w\equiv 0$ for $B_4^{(1)}(\alpha_j)_{0\leq j \leq 4},$ then $\alpha_3=0$ or $y\equiv 1/2.$ 
\newline
{\rm (4)}\quad 
For $B_4^{(1)}(\alpha_j)_{0\leq j \leq 4},$ 
there exists no solution such that $z\equiv 0.$
}
\end{proposition}

\begin{proof}
We treat case (2). The other cases can be proved in the same way. 
By considering 
\begin{equation*}
\begin{cases}
\displaystyle 
t x^{\prime}=
2x^2y-x^2+(1-2\alpha_2-2\alpha_3-2\alpha_4)x+2\alpha_3z+2z^2w+t, \\
\displaystyle 
t z^{\prime}=
2z^2w-z^2+(1-2\alpha_4)z+2yz^2+t, 
\end{cases}
\end{equation*}
we find that $\alpha_2=0$ or $x\equiv z \equiv 0.$ 
If $x\equiv z \equiv 0,$ 
it follows that $0=t,$ because 
\begin{equation*}
t x^{\prime}=
2x^2y-x^2+(1-2\alpha_2-2\alpha_3-2\alpha_4)x+2\alpha_3z+2z^2w+t,
\end{equation*}
but this is impossible.

\end{proof}
\par
By Proposition \ref{prop:backlund}, 
we consider $s_0,$ $s_1,$ or $s_2$ as the identical transformation, 
if $y\equiv 1,$ $y\equiv 0$ or $x\equiv z$ for $B_4^{(1)}(\alpha_j)_{0\leq j \leq 4}.$
\par
Considering $s_3$ more in detail, 
we can easily check the following proposition:

\begin{proposition}
\label{prop:exam1}
{\it
Suppose that for $B_4^{(1)}(\alpha_j)_{0\leq j \leq 4},$ 
$w\equiv 0$ and $\alpha_3\neq 0.$ 
Either of the following then occurs:
\newline
{\rm (1)}\quad 
$\alpha_3=-1/2, \,\,\alpha_4=1/2$ 
and 
\begin{equation*}
x\equiv \alpha_0-\alpha_1, \,\,
y\equiv  \frac12, \,\,
z=t+(\alpha_0+\alpha_1)(\alpha_0-\alpha_1), \,\,
w\equiv 0,
\end{equation*}
{\rm (2)}\quad 
$\alpha_3+\alpha_4=0, \,\,(\alpha_0+\alpha_1)(\alpha_0-\alpha_1)=0$ 
and 
\begin{equation*}
x\equiv \alpha_0-\alpha_1, \,\,
y\equiv \frac12, \,\,
z=-\frac{1}{2\alpha_3}t, \,\,
w\equiv 0.
\end{equation*}
}
\end{proposition}

Proposition \ref{prop:exam1} shows that  
if $w\equiv 0$ and $\alpha_3\neq 0$ for $B_4^{(1)}(\alpha_j)_{0\leq j \leq 4},$
$$
s_3(x,y,z,w)=(\alpha_0-\alpha_1, 1/2, \infty, 0).
$$
Therefore, 
we have to consider the infinite solution such that $z\equiv \infty,$ 
which is treated in the next section. 
If $w\equiv 0$ and $\alpha_3=0$ for $B_4^{(1)}(\alpha_j)_{0\leq j \leq 4},$ 
we consider $s_3$ as the identical transformation.

In order to study $\pi_2,$
we assume that $x\equiv 0$ and obtain examples of the rational solutions.

\begin{proposition}
\label{prop:exam2}
{\it
Suppose that for $B_4^{(1)}(\alpha_j)_{0\leq j \leq 4},$ 
$x\equiv 0.$ 
Then, 
$\alpha_4\neq 0$ and 
either of the following occurs:
\newline
{\rm (1)}\quad 
$-\alpha_0+\alpha_1=0,\,\, \alpha_3+\alpha_4=0$ 
and 
\begin{equation*}
x\equiv 0, \,\,
y\equiv \frac12, \,\,
z=\frac{1}{2\alpha_4}t, \,\,
w\equiv 0, 
\end{equation*}
{\rm (2)}\quad $\alpha_0=\alpha_1=1/2$ and 
\begin{equation*}
x\equiv 0, \,\,
y=\frac12+\frac{2\alpha_4(\alpha_3+\alpha_4)}{t}, \,\,
z=\frac{1}{2\alpha_4}t, \,\,
w=-\frac{2\alpha_4(\alpha_3+\alpha_4)}{t}.
\end{equation*}
}
\end{proposition}

Proposition \ref{prop:exam2} implies that 
if $x\equiv 0$ for $B_4^{(1)}(\alpha_j)_{0\leq j \leq 4},$ 
$$
\pi_2(x,y,z,w)=(2\alpha_4, 1/2, \infty, 0).
$$
Therefore, 
we have to consider the infinite solution such that $z\equiv \infty,$ 
which is treated in the next section.

\section{Infinite solutions}

In this section, 
we consider an ``infinite solution," that is, a solution such that some of $(x,y,z,w)$ are identically equal to $\infty.$ 
Especially, we treat the infinite solution such that $z\equiv \infty.$ 
For this purpose, 
following Sasano \cite{Sasano-2}, 
we introduce the coordinate transformation 
which is given by 
\begin{equation*}
%&x_0=1/x, & &y_0=-\{(y-1)x+\alpha_0\}x, &  &z_0=z,  &  &w_0=w, \\
%&x_1=1/x, & &y_1=-(yx+\alpha_1)x, & &z_1=z, & &w_1=w, \\
%&x_2=-\{(x-z)y-\alpha_2\}y, & &y_2=1/y, & &z_2=z, & &w_2=w+y,  \\
m_3: x_3=x, \, y_3=y, \, z_3=1/z_, \, w_3=-(wz+\alpha_3)z, 
%&x_4=x, & &y_4=y, & &z_4=z, & &w_4=\displaystyle w-2\alpha_4/z+t/z^2.\\
\end{equation*}

By $m_3,$ we have the following proposition.

\begin{proposition}
\label{prop:zinf}
{\it
For $B_4^{(1)}(\alpha_j)_{0\leq j \leq 4},$ 
there exists a solution such that 
$z\equiv \infty.$ 
Either of the following then occurs:
\newline
{\rm (1)}\quad $(\alpha_0+\alpha_1)(\alpha_0-\alpha_1)=0, \,\,\alpha_4=0,$ and 
$$
x_3=\alpha_0-\alpha_1, \,y_3=\frac12, \,z_3=0, \,w_3=\frac{t}{2},
$$
that is, $x=\alpha_0-\alpha_1, \,y=1/2, \,z=\infty, \,w=0,$
\newline
{\rm (2)}\quad $\alpha_3=1/2, \,\,\alpha_4=0. $ 
$$
x_3=\alpha_0-\alpha_1, \,y_3=\frac12, \,z_3=0, \,w_3=\frac{t}{2}+\frac{(\alpha_0+\alpha_1)(\alpha_0-\alpha_1)}{2}, 
$$
that is, $x=\alpha_0-\alpha_1, \,y=1/2, \,z=\infty, \,w=0.$ 
}
\end{proposition}

\begin{proof}
For $B_4^{(1)}(\alpha_j)_{0\leq j \leq 4},$ 
$m_3$ transforms the system of $(x,y,z,w)$ into the system of $(x_3, y_3, z_3, w_3),$ 
which is given by 
\begin{equation*}
\begin{cases}
tx_3^{\prime}=2x_3^2y_3-x_3^2+(1-2\alpha_2-2\alpha_3-2\alpha_4)x_3-2w_3+t,  \\
ty_3^{\prime}=-2x_3y_3^2+2x_3y_3-(1-2\alpha_2-2\alpha_3-2\alpha_4)y_3+\alpha_1, \\
tz_3^{\prime}=2z_3^2w_3+1-(1-2\alpha_3-2\alpha_4)z_3-2y_3-tz_3^2, \\
tw_3^{\prime}=-2z_3w_3^2+2tz_3w_3+\alpha_3t+(1-2\alpha_3-2\alpha_4)w_3. 
\end{cases}
\end{equation*}
Substituting $z_3\equiv 0$ in 
$$
tz_3^{\prime}=2z_3^2w_3+1-(1-2\alpha_3-2\alpha_4)z_3-2y_3-tz_3^2,
$$
we have $y_3=1/2,$ which implies that $x_3=\alpha_0-\alpha_1,$ 
because 
$$
ty_3^{\prime}=-2x_3y_3^2+2x_3y_3-(1-2\alpha_2-2\alpha_3-2\alpha_4)y_3+\alpha_1.
$$ 
Substituting $x_3=\alpha_0-\alpha_1, y_3=1/2, z_3=0$ in 
$$
tx_3^{\prime}=2x_3^2y_3-x_3^2+(1-2\alpha_2-2\alpha_3-2\alpha_4)x_3-2w_3+t, 
$$
we obtain 
\begin{equation}
\label{eqn:D4-inf.sol}
w_3=\frac{t}{2}+\frac{(\alpha_0+\alpha_1)(\alpha_0-\alpha_1)}{2}.
\end{equation}
Substituting $z_3\equiv 0$ and (\ref{eqn:D4-inf.sol}) in 
$$
tw_3^{\prime}=-2z_3w_3^2+2tz_3w_3+\alpha_3t+(1-2\alpha_3-2\alpha_4)w_3,
$$
we have 
$$
0=-\alpha_4t+(1-2\alpha_3-2\alpha_4)\frac{(\alpha_0-\alpha_1)(\alpha_0+\alpha_1)}{2},
$$
which proves the proposition. 
\end{proof}

{\bf Remark}
\quad
In both cases of Proposition \ref{prop:zinf}, 
we can express the infinite solution by 
$$
x_3=\alpha_0-\alpha_1, \,y_3=\frac12, \,z_3=0, w_3=\frac{t}{2}+\frac{(\alpha_0+\alpha_1)(\alpha_0-\alpha_1)}{2}. 
$$

\subsection{B\"acklund transformations and infinite solutions}

In this subsection, 
we treat the relationship between the B\"acklund transformations and 
a solution such that $z\equiv \infty$ and $x,y,w \not\equiv \infty.$
By $m_3$ and Proposition \ref{prop:zinf}, 
we can prove the following proposition.
\begin{proposition}
{\it
Suppose that 
for $B_4^{(1)}(\alpha_j)_{0\leq j \leq 4},$ 
there exists a solution such that 
$z\equiv \infty$ and $x,y,w\not\equiv \infty.$ 
The actions of the B\"acklund transformations to the inifinite solutions are then expressed by
\begin{align*}
s_0 &: (x,y,z,w) \longrightarrow (-\alpha_0-\alpha_1, \,\,1/2, \,\,\infty, \,\,0),  \\
s_1 &: (x,y,z,w) \longrightarrow (\alpha_0+\alpha_1, \,\,1/2, \,\,\infty, \,\,0),  \\
s_2 &: (x,y,z,w) \longrightarrow (\alpha_0-\alpha_1, \,\,1/2, \,\,\infty, \,\,0), \\
s_3 &: (x,y,z,w) \longrightarrow (\alpha_0-\alpha_1, \,\,1/2, \,\,t/\{2\alpha_3\}+(\alpha_0+\alpha_1)(\alpha_0-\alpha_1)/\{2\alpha_3\}, 0 ), \\
s_4 &: (x,y,z,w) \longrightarrow (\alpha_0-\alpha_1, \,\,1/2, \,\,\infty, \,\,0), \\
\pi_1 &: (x,y,z,w) \longrightarrow (\alpha_1-\alpha_0, \,\,1/2, \,\,\infty, \,\, 0), \\
\pi_2 &: (x,y,z,w) \longrightarrow (0, \,\,  1/2+(\alpha_0+\alpha_1)(\alpha_0-\alpha_1)/\{2t\},  \,\,t/(\alpha_0-\alpha_1), \,\, -(\alpha_0+\alpha_1)(\alpha_0-\alpha_1)/\{2t\} ).
\end{align*}
Furthermore, if $\alpha_0-\alpha_1=0,$ $\pi_2$ is given by 
$
 (x,y,z,w) \longrightarrow  (0, 1/2, \,\,\infty, \,\,0).
$
}
\end{proposition}

\begin{proof}
We treat $\pi_2.$ The other cases can be proved in the same way. 
From the definitions of $\pi_2$ and $m_3,$ 
it follows that 
\begin{align*}
&\pi_2(x)=\frac{t}{z}=tz_3=0, \\
&\pi_2(y)=-\frac{z}{t}(zw+\alpha_3)
=\frac{w_3}{t}=\frac12+\frac{(\alpha_0+\alpha_1)(\alpha_0-\alpha_1)}{2t}, \\
&\pi_2(z)=\frac{t}{x}=\frac{t}{x_3}=\frac{t}{\alpha_0-\alpha_1},   \\
&\pi_2(w)=-\frac{x}{t}(xy+\alpha_1)=-\frac{x_3}{t}(x_3y_3+\alpha_1)=-\frac{(\alpha_0-\alpha_1)(\alpha_0+\alpha_1)}{2t}. 
\end{align*}
If $\alpha_0-\alpha_1=0,$ 
by $m_3\circ \pi_2\circ m_3^{-1},$ 
we obtain 
\begin{align*}
&m_3(\pi_2(m_3^{-1}(x_3)))=tz_3=0, \\
&m_3(\pi_2(m_3^{-1}(y_3)))=\frac12+\frac{(\alpha_0+\alpha_1)(\alpha_0-\alpha_1)}{2t}=\frac12, \\
&m_3(\pi_2(m_3^{-1}(z_3)))=\frac{\alpha_0-\alpha_1}{t}=0, \\
&m_3(\pi_2(m_3^{-1}(w_3)))=-(\pi_2(w)\pi_2(z)+\pi_2(\alpha_3))\pi_2(z)=ty_3=\frac{t}{2},
\end{align*}
which means that 
$$
\pi_2(x,y,z,w)=(0, 1/2, \,\,\infty, \,\,0).
$$

\end{proof}

\section{Necessary conditions$\cdots$the case where $z$ has a pole of order one at $t=\infty$}

In this section, 
we assume that 
$z$ has a pole of order one at $t=\infty$ 
and 
obtain 
the necessary conditions 
for $B_4^{(1)}(\alpha_j)_{0\leq j \leq 4}$ 
to have such a rational solution.

\subsection{The case where $x,y,z,w$ are all holomorphic at $t=0$ }

In this subsection, 
we assume that 
$z$ has a pole of order one at $t=\infty$ 
and 
all of $(x,y,z,w)$ are holomorphic at $t=0.$

\subsubsection{The case where $a_{0,0}=0$}

\begin{proposition}
\label{prop:nec-order1-holozero-zero-nonzero}
{\it
Suppose that for $B_4^{(1)}(\alpha_j)_{0\leq j \leq 4},$  
there exists a rational solution 
such that 
$z$ has a pole of order one at $t=\infty$ 
and 
$x,y,w$ are all holomorphic at $t=\infty.$ 
Moreover, 
assume that 
$x,y,z,w$ are all holomorphic at $t=0$ 
and $a_{0,0}=0, \,\,c_{0,0}\neq0.$ 
Then, 
$\alpha_0-\alpha_1\in\Z, \,\,2\alpha_4\in\Z.$
}
\end{proposition}

\begin{proof}
From Proposition \ref{rational-necessary-1}, 
it follows that $\alpha_0-\alpha_1\in\Z.$ 
Furthermore, 
from Proposition \ref{rational-necessary-2}, 
it follows that 
\begin{equation}
\label{eqn:h-relation-1}
h_{\infty,0}-h_{0,0}=1/4\cdot(\alpha_0-\alpha_1)^2+\alpha_4^2-\alpha_4\in\Z.
\end{equation}
\par
$s_4(x,y,z,w)$ is a solution of $B_4^{(1)}(\alpha_0,\alpha_1,\alpha_2,\alpha_3+2\alpha_4,-\alpha_4)$ 
such that $z$ has a pole of order one at $t=\infty$ 
and 
$x,y,w$ are all holomorphic at $t=\infty.$ 
Moreover, 
for $s_4(x,y,z,w)$ 
$x,y,z,w$ are all holomorphic at $t=0$ 
and $a_{0,0}=0, \,\,c_{0,0}\neq0.$ 
From Proposition \ref{rational-necessary-2}, 
it then follows that 
\begin{equation}
\label{eqn:h-relation-2}
h_{\infty,0}-h_{0,0}=1/4\cdot(\alpha_0-\alpha_1)^2+\alpha_4^2+\alpha_4\in\Z.
\end{equation}
From (\ref{eqn:h-relation-1}) and (\ref{eqn:h-relation-2}), 
it follows that $2\alpha_4\in\Z.$

\end{proof}

In order to treat the case where $a_{0,0}=c_{0,0}=0,$ 
let us prove the following two lemmas:

\begin{lemma}
\label{lem:holo-1}
{\it
Suppose that for $B_4^{(1)}(\alpha_j)_{0\leq j \leq 4},$  
there exists a solution 
such that $x,y,z,w$ are all holomorphic at $t=0$ 
and $c_{0,0}=0.$ 
Then, $\alpha_4\neq 0$ and 
$c_{0,1}=1/\{2\alpha_4\}.$
}
\end{lemma}

\begin{proof}
From Proposition \ref{prop:backlund}, 
let us first note that $z\not\equiv 0.$ 
Suppose that $\alpha_4=0.$ 
$s_4(x,y,z,w)$ is then a solution of $B_4^{(1)}(\alpha_0,\alpha_1,\alpha_2,\alpha_3,0)$ 
such that only $s_4(w)$ has a pole at $t=0,$ 
which is impossible from Proposition \ref{prop:t=0behavior}. 
\par
Thus, it follows that $\alpha_4\neq 0$ 
and 
$s_4(x,y,z,w)$ is a solution of $B_4^{(1)}(\alpha_0,\alpha_1,\alpha_2,\alpha_3+2\alpha_4,-\alpha_4)$ 
such that all of $s_4(x,y,z,w)$ are holomorphic at $t=0,$ which implies that 
$c_{0,1}=1/\{2\alpha_4\}.$ 

\end{proof}

\begin{lemma}
\label{lem:holo-2}
{\it
Suppose that for $B_4^{(1)}(\alpha_j)_{0\leq j \leq 4},$  
there exists a solution 
such that $x,y,z,w$ are all holomorphic at $t=0$ 
and $a_{0,0}=c_{0,0}=0.$ 
Moreover, assume that $\alpha_2\neq 0.$ 
Then, $\alpha_4(\alpha_2+\alpha_3+\alpha_4)\neq 0$ 
and 
$a_{0,1}=(\alpha_3+\alpha_4)/\{2\alpha_4(\alpha_2+\alpha_3+\alpha_4)\}.$
}
\end{lemma}

\begin{proof}
From Proposition \ref{prop:backlund}, 
let us first note that $x\not\equiv z.$ 
$s_2(x,y,z,w)$ is then a solution of $B_4^{(1)}(\alpha_0+\alpha_2,\alpha_1+\alpha_2,-\alpha_2,\alpha_3+\alpha_2,\alpha_4)$ 
such that both of  $s_2(y,w)$ have a pole at $t=0$ and 
both of $s_2(x,z)$ are holomorphic at $t=0.$ 
Thus, it follows from Proposition \ref{prop:t=0-(y,w)} that 
$\alpha_2+\alpha_3+\alpha_4\neq 0$ and 
$$
-\frac{\alpha_2}{a_{0,1}-1/\{2\alpha_4\}}=\mathrm{Res}_{t=0} s_2(y)=2\alpha_4(\alpha_2+\alpha_3+\alpha_4),
$$
which implies that $a_{0,1}=(\alpha_3+\alpha_4)/2\alpha_4(\alpha_2+\alpha_3+\alpha_4).$

\end{proof}

By Lemmas \ref{lem:holo-1} and \ref{lem:holo-2}, 
we can prove the following proposition:

\begin{proposition}
\label{prop:nec-order1-holozero-zero-zero}
{\it
Suppose that for $B_4^{(1)}(\alpha_j)_{0\leq j \leq 4},$  
there exists a rational solution 
such that 
$z$ has a pole of order one at $t=\infty$ 
and 
$x,y,w$ are all holomorphic at $t=\infty.$ 
Moreover, 
assume that 
$x,y,z,w$ are all holomorphic at $t=0$ 
and $a_{0,0}=0, c_{0,0}=0.$ 
One of the following then occurs: 
{\rm (1)}\quad $\alpha_0=\alpha_1=0,$
\,\,
{\rm (2)}\quad $\alpha_0-\alpha_1\in\Z, \,\,2\alpha_4=1,$
\,\,
{\rm (3)}\quad $\alpha_0-\alpha_1\in\Z, \,\,2\alpha_3+2\alpha_4\in\Z.$
}
\end{proposition}

\begin{proof}
From Proposition \ref{rational-necessary-1}, 
it follows that $\alpha_0-\alpha_1\in\Z.$ 
By Proposition \ref{prop:t=0holo-a_{0,0}=0} and \ref{prop:exam2}, 
we can assume that $x\not\equiv 0$ and 
$$
a_{0,0}=0, \,\,
b_{0,0}=\alpha_1/(\alpha_0+\alpha_1), \,\,
c_{0,0}=0, \,\,
d_{0,0}=-\alpha_3(\alpha_0-\alpha_1)/\{(2\alpha_4-1)(\alpha_0+\alpha_1)\},
$$
where $\alpha_0+\alpha_1\neq 0$ and $2\alpha_4\neq 1.$
\par
If $\alpha_0-\alpha_1\neq 0$ and $\alpha_3\neq 0,$ 
then 
$s_3(x,y,z,w)$ is a solution of $B_4^{(1)}(\alpha_0,\alpha_1,\alpha_2+\alpha_3,-\alpha_3,\alpha_4+\alpha_3)$ 
such that $s_3(z)$ has a pole of order one at $t=\infty$ 
and 
all of $s_3(x,y,w)$ are holomorphic at $t=\infty.$
Moreover, 
all of $s_3(x,y,z,w)$ are holomorphic at $t=0$ 
and for $s_3(x,y,z,w),$ $a_{0,0}=0, c_{0,0}\neq 0.$ 
Thus, by Proposition \ref{prop:nec-order1-holozero-zero-nonzero}, 
we find that $\alpha_0-\alpha_1\in\Z, \,\,2\alpha_3+2\alpha_4\in\Z.$
\par
If $\alpha_0-\alpha_1\neq 0$ and $\alpha_3=0,$ 
then 
$s_4(x,y,z,w)$ 
is a solution of $B_4^{(1)}(\alpha_0,\alpha_1,\alpha_2,2\alpha_4,-\alpha_4)$ 
such that 
$s_4(z)$ has a pole of order one at $t=\infty$ 
and 
all of $s_4(x,y,w)$ are holomorphic at $t=\infty.$ 
Moreover,  
$s_4(x,y,z,w)$ are all holomorphic at $t=0$ 
and for $s_4(x,y,z,w),$ $a_{0,0}=0, c_{0,0}=0.$ 
Thus, from the above discussion, 
we see that $\alpha_0-\alpha_1\in\Z, \,\,2\alpha_3+2\alpha_4\in\Z.$
\par
Let us suppose that $\alpha_0-\alpha_1=0$ and $\alpha_2\neq 0.$  
We can then assume that $\alpha_3+\alpha_4\neq 0.$
Thus, 
$\pi_2(x,y,z,w)$ is a solution of $B_4^{(1)}(2\alpha_4+\alpha_3,\alpha_3,\alpha_2,\alpha_1,0)$ 
such that 
$\pi_2(z)$ has a pole of order $n\,\,(n\geq 2)$ at $t=\infty$ 
and 
all of $\pi_2(x,y,w)$ are holomorphic at $t=\infty.$ 
Moreover, all of $\pi_2(x,y,z,w)$ are holomorphic at $t=0$ 
and 
for $\pi_2(x,y,z,w),$ $a_{0,0}=(2\alpha_4+\alpha_3)-\alpha_3\neq 0, \,\,b_{0,0}=-\alpha_3/\{2\alpha_4\}, \,\,
c_{0,0}=(\alpha_2+\alpha_3+\alpha_4)/(\alpha_3+\alpha_4),$ 
which implies that $\alpha_0-\alpha_1=0$ and $2\alpha_3+2\alpha_4\in\Z$ 
from Proposition \ref{prop:nec-z(n)-t=0holo-a_{0,0}nonzero-nonhalf}.
\par
Let us consider that  $\alpha_0-\alpha_1=0$ and $\alpha_2= 0.$ 
$s_1s_0(x,y,z,w)$ is then a solution of $B_4^{(1)}(-\alpha_0,-\alpha_1,2\alpha_0,\alpha_3,\alpha_4)$ 
such that 
only $s_1s_0(z)$ has a pole of order one at $t=\infty.$ 
Moreover, 
all of $s_1s_0(x,y,z,w)$ are holomorphic at $t=0$ and for  $s_1s_0(x,y,z,w),$ 
$a_{0,0}=0, b_{0,0}=1/2, \,\,c_{0,0}=0.$ 
Thus, it follows from the above discussion that $\alpha_0-\alpha_1=0$ and $2\alpha_3+2\alpha_4\in\Z.$

\end{proof}

{\bf Remark} 
Proposition \ref{prop:nec-z(n)-t=0holo-a_{0,0}nonzero-nonhalf} 
can be proved independent of the propositions in this section.
\newline
\par
Let us summarize Propositions \ref{prop:nec-order1-holozero-zero-nonzero} and \ref{prop:nec-order1-holozero-zero-zero}.

\begin{proposition}
\label{prop:nec-order1-t=0-holo-a_{0,0}zero}
{\it
Suppose that for $B_4^{(1)}(\alpha_j)_{0\leq j \leq 4},$  
there exists a rational solution 
such that 
$z$ has a pole of order one at $t=\infty$ 
and 
$x,y,w$ are all holomorphic at $t=\infty.$ 
Moreover, 
assume that 
$x,y,z,w$ are all holomorphic at $t=0$ 
and $a_{0,0}=0.$ 
One of the following then occurs:
\newline
{\rm (1)}\quad $\alpha_0=\alpha_1=0, $
\newline
{\rm (2)}\quad $\alpha_0-\alpha_1\in\Z, \,\,2\alpha_4\in\Z,$
\newline
{\rm (3)}\quad $\alpha_0-\alpha_1\in\Z, \,\,2\alpha_3+2\alpha_4\in\Z.$
}
\end{proposition}

\subsubsection{$a_{0,0}=\alpha_0-\alpha_1\neq0$}

We treat the case where $x,y,z,w$ are all holomorphic at $t=0$ and $a_{0,0}\neq0.$ 
Let us first deal with the case in which $b_{0,0}\neq 1/2.$ 

\begin{proposition}
\label{prop:z(1)-t=0-a_{0,0}nonzero-b_{0,0}-nonhalf}
{\it
Suppose that for $B_4^{(1)}(\alpha_j)_{0\leq j \leq 4},$  
there exists a rational solution 
such that 
$z$ has a pole of order one at $t=\infty$ 
and 
$x,y,w$ are all holomorphic at $t=\infty.$ 
Moreover, 
assume that 
$x,y,z,w$ are all holomorphic at $t=0$ 
and $a_{0,0}=\alpha_0-\alpha_1\neq 0, b_{0,0}=-\alpha_1/(\alpha_0-\alpha_1)\neq1/2.$ 
Either of the following then occurs:
\newline
{\rm (1)}\quad $\alpha_0+\alpha_1\in\Z, \,\,2\alpha_4\in\Z,$
\newline
{\rm (2)}\quad $\alpha_0+\alpha_1\in\Z, \,\,2\alpha_3+2\alpha_4\in\Z.$
}
\end{proposition}

\begin{proof}
Since $a_{0,0}=\alpha_0-\alpha_1\neq 0,$ it follows that $\alpha_0\neq 0$ or $\alpha_1\neq 0.$ 
If $\alpha_0\neq 0,$ then 
using $s_0$ we find that 
$s_0(x,y,z,w)$ is a solution of $B_4^{(1)}(-\alpha_0,\alpha_1,\alpha_2+\alpha_0,\alpha_3,\alpha_4)$ 
such that only $s_0(z)$ has a pole of order one at $t=0$ and all of $S_0(x,y,z,w)$ are 
holomorphic at $t=0.$ Moreover, for $s_0(x,y,z,w)$ $a_{0,0}=0.$ 
Thus, the proposition follows from Proposition \ref{prop:nec-order1-t=0-holo-a_{0,0}zero}.
\par
If $\alpha_1\neq 0,$ 
we use $s_1$ in the same way and can obtain the necessary conditions.

\end{proof}

Let us treat the case where $a_{0,0}\neq 0$ and $b_{0,0}=1/2.$

\begin{proposition}
\label{prop:z(1)-t=0-a_{0,0}nonzero-b_{0,0}-half}
{\it
Suppose that for $B_4^{(1)}(\alpha_j)_{0\leq j \leq 4},$  
there exists a rational solution 
such that 
$z$ has a pole of order one at $t=\infty$ 
and 
$x,y,w$ are all holomorphic at $t=\infty.$ 
Moreover, 
assume that 
$x,y,z,w$ are all holomorphic at $t=0$ 
and $a_{0,0}=\alpha_0-\alpha_1\neq 0, b_{0,0}=1/2.$ 
One of the following then occurs:
\newline
{\rm (1)}\quad $\alpha_0+\alpha_1\in\Z, \,\, 2\alpha_4\in\Z,$
\newline
{\rm (2)}\quad $\alpha_0+\alpha_1\in\Z, \,\,2\alpha_3+2\alpha_4\in\Z,$
\newline
{\rm (3)}\quad $\alpha_0+\alpha_1+2\alpha_2\in\Z, \,\,2\alpha_4\in\Z.$
}
\end{proposition}

\begin{proof}
By Proposition \ref{prop:t=0holo-a_{0,0}nonzero-half}, 
we can assume that 
$a_{0,0}=\alpha_0-\alpha_1, \,\,b_{0,0}=1/2, \,\,c_{0,0}=0,$ 
or 
that 
$a_{0,0}=\alpha_0-\alpha_1, \,\,b_{0,0}=1/2, \,\,c_{0,0}=(\alpha_0+\alpha_1)(\alpha_0-\alpha_1)/(1-2\alpha_3-2\alpha_4)\neq 0.$ 
\par
Let us first suppose that $\alpha_2\neq 0.$ 
$s_2(x,y,z,w)$ is then a solution of $B_4^{(1)}(\alpha_0+\alpha_2,\alpha_1+\alpha_2,-\alpha_2,\alpha_3+\alpha_2, \alpha_4)$ 
such that 
only $s_2(z)$ has a pole of order one at $t=\infty$ and all of $s_2(x,y,z,w)$ are holomorphic at $t=0.$ 
Moreover, for $s_2(x,y,z,w),$ $a_{0,0}=\alpha_0-\alpha_1\neq 0$ and $b_{0,0}\neq 1/2.$ 
Thus, from Proposition \ref{prop:z(1)-t=0-a_{0,0}nonzero-b_{0,0}-nonhalf}, 
we can obtain the necessary conditions. 
\par
Now, let us consider the case where $\alpha_2=0.$ 
Since $a_{0,0}=\alpha_0-\alpha_1\neq 0,$ it follows that $\alpha_0\neq 0$ or $\alpha_1\neq 0.$ 
We suppose that $\alpha_0\neq 0.$ 
$s_0(x,y,z,w)$ is then a solution of $B_4^{(1)}(-\alpha_0,\alpha_1,\alpha_0,\alpha_3,\alpha_4)$ 
such that 
only $s_2(z)$ has a pole of order one at $t=\infty$ and all of $s_2(x,y,z,w)$ are holomorphic at $t=0.$ 
Moreover, for $s_0(x,y,z,w),$ $a_{0,0}=-\alpha_0-\alpha_1$ and $b_{0,0}\neq 1/2.$ 
When $\alpha_0+\alpha_1=0,$ we can obtain the necessary conditions from Proposition \ref{prop:nec-order1-t=0-holo-a_{0,0}zero}. 
When $\alpha_0+\alpha_1neq0,$ we can obtain the necessary conditions from the above discussion. 
\par
If $\alpha_2=0$ and $\alpha_1\neq 0,$ 
using $s_1$ in the same way, 
we can obtain the necessary conditions.

\end{proof}

\subsection{The case where $z$ has a pole at $t=0$ }

\begin{proposition}
{\it
Suppose that for $B_4^{(1)}(\alpha_j)_{0\leq j \leq 4},$  
there exists a rational solution 
such that 
$z$ has a pole of order one at $t=\infty$ 
and 
$x,y,w$ are all holomorphic at $t=\infty.$ 
Moreover, 
assume that 
$z$ has a pole at $t=0$ 
and 
$x,y,w$ are all holomorphic at $t=0.$ 
One of the following then occurs:
\newline
{\rm (1)}\quad $\alpha_0=\alpha_1=0, $
\newline
{\rm (2)}\quad $\alpha_0-\alpha_1\in\Z, \,\,2\alpha_3+2\alpha_4\in\Z, $
\newline
{\rm (3)}\quad $\alpha_0-\alpha_1\in\Z, \,\,2\alpha_4\in\Z, $
\newline
{\rm (4)}\quad $\alpha_0+\alpha_1\in\Z, \,\,2\alpha_3+2\alpha_4\in\Z,  $
\newline
{\rm (5)}\quad $\alpha_0+\alpha_1\in\Z, \,\,2\alpha_4\in\Z,$
\newline
{\rm (6)}\quad $2\alpha_3\in\Z, \,\,2\alpha_4\in \Z. $
}
\end{proposition}

\begin{proof}
Let us first suppose that $\alpha_3(\alpha_3+\alpha_4)\neq 0.$ 
$s_3(x,y,z,w)$ is then a solution of $B_4^{(1)}(\alpha_0,\alpha_1,\alpha_2+\alpha_3,-\alpha_3,\alpha_4+\alpha_3)$ 
such that only $s_3(z)$ has a pole of order one at $t=\infty$ 
and all of $s_3(x,y,z,w)$ are holomorphic at $t=0.$ 
Moreover, for  $s_3(x,y,z,w),$ 
$a_{0,0}=\alpha_0-\alpha_1,\,\,b_{0,0}=1/2, \,\,d_{0,0}=0.$ 
If $\alpha_0-\alpha_1=0,$ we can obtain the necessary conditions from Proposition \ref{prop:nec-order1-t=0-holo-a_{0,0}zero}. 
If $\alpha_0-\alpha_1\neq0,$ we can obtain the necessary conditions 
from Proposition \ref{prop:z(1)-t=0-a_{0,0}nonzero-b_{0,0}-nonhalf} and \ref{prop:z(1)-t=0-a_{0,0}nonzero-b_{0,0}-half}.
\par
Now, let us consider the case where $\alpha_3\neq 0$ and $\alpha_3+\alpha_4=0.$ 
It then follows from Proposition \ref{prop:exam1} that $w \not\equiv 0.$ 
Thus, 
$s_3(x,y,z,w)$ is a solution of $B_4^{(1)}(\alpha_0,\alpha_1,\alpha_2+\alpha_3,-\alpha_3,\alpha_4+\alpha_3)$ 
such that only $s_3(z)$ has a pole of order $n \,\,(n\geq 2)$ at $t=\infty$ 
and all of $s_3(x,y,z,w)$ are holomorphic at $t=0.$ 
Moreover, for  $s_3(x,y,z,w),$ 
$a_{0,0}=\alpha_0-\alpha_1,\,\,b_{0,0}=1/2, \,\,d_{0,0}=0.$ 
Therefore, 
using Proposition \ref{prop:nec-z(n)-t=0holo-a_{0,0}nonzero-half}, 
we can obtain the necessary conditions.
\par
If $\alpha_3=0,$ then $\alpha_3+\alpha_4\neq 0,$ 
because $\alpha_4\neq 0.$ 
Thus, 
$s_4(x,y,z,w)$ is a solution of $B_4^{(1)}(\alpha_0,\alpha_1,\alpha_2,2\alpha_4,-\alpha_4)$ 
such that only $s_4(z)$ has a pole of order one at $t=\infty$ 
and has a pole at $t=0.$ 
Therefore, 
we can obtain the necessary condition based on the above discussion.   
\end{proof}

\subsection{The case where $y,w$ have a pole at $t=0$ }

\begin{proposition}
{\it
Suppose that for $B_4^{(1)}(\alpha_j)_{0\leq j \leq 4},$  
there exists a rational solution 
such that 
$z$ has a pole of order one at $t=\infty$ 
and 
$x,y,w$ are all holomorphic at $t=\infty.$ 
Moreover, 
assume that 
$y,w$ both have a pole at $t=0$ 
and 
$x,z$ are both holomorphic at $t=0.$ 
One of the following then occurs:
\newline
{\rm (1)}\quad $\alpha_0+\alpha_2=\alpha_1+\alpha_2=0,$
\newline
{\rm (2)}\quad $\alpha_0-\alpha_1\in\Z, \,\,2\alpha_4\in\Z,$
\newline
{\rm (3)}\quad $\alpha_0-\alpha_1\in\Z, \,\,\alpha_0+\alpha_1\in\Z,$
\newline
{\rm (4)}\quad $\alpha_0+\alpha_1\in\Z, \,\,2\alpha_3+2\alpha_4\in\Z,$
\newline
{\rm (5)}\quad $2\alpha_3\in\Z, \,\,2\alpha_4\in\Z.$
}
\end{proposition}

\begin{proof}
By Proposition \ref{prop:t=0-(y,w)}, 
let us note that $\alpha_4(\alpha_3+\alpha_4)\neq 0.$ 
We first assume that $\alpha_2\neq 0.$ 
$s_2(x,y,z,w)$ 
is then a solution of 
$B_4^{(1)}(\alpha_0+\alpha_2,\alpha_1+\alpha_2,-\alpha_2,\alpha_3+\alpha_2,\alpha_4)$ 
such that only $s_2(z)$ has a pole of order one at $t=\infty$ 
and 
all of $s_2(x,y,z,w)$ are holomorphic at $t=0.$ 
Moreover, 
for $s_2(x,y,z,w),$ $a_{0,0}=c_{0,0}=0.$ 
Thus, by Proposition \ref{prop:nec-order1-holozero-zero-zero}, 
we obtain the necessary conditions.  
\par
Now, let us suppose that $\alpha_2=0$ and $\alpha_0\neq 0.$ 
$s_0(x,y,z,w)$ is then a solution of $B_4^{(1)}(-\alpha_0,\alpha_1,\alpha_0,\alpha_3,\alpha_4)$ 
such that only $s_2(z)$ has a pole of order one at $t=\infty$ 
and 
only $s_2(y,w)$ have a pole at $t=0.$ 
Based on the above discussion, 
we then obtain the necessary condition. 
\par
If $\alpha_2=0$ and $\alpha_1\neq 0,$ 
using $s_1$ in the same way, 
we obtain the necessary conditions. 
\par
If $\alpha_0=\alpha_1=\alpha_2=0,$ 
the parameters then satisfy one of the conditions in the proposition.
\end{proof}

\section{Necessary conditions$\cdots$the case where $z$ has a pole of order $n \,\,(n\geq 2)$ at $t=\infty$ }
In this section, for $B_4^{(1)}(\alpha_j)_{0\leq j \leq 4},$
we treat a rational solution such that  $z$ has a pole of order $n \,\,(n\geq 2)$ at $t=\infty$ 
and assume that $\alpha_4=0.$ 
This section consists of three subsections. 
In the first, second and third subsections, 
we treat the case where $x,y,z,w$ are all holomorphic at $t=0,$ 
the case where $z$ has a pole at $t=0,$ 
and 
the case where $y,w$ both have a pole at $t=0,$ 
respectively.

\subsection{The case where $x,y,z,w$ are all holomorphic at $t=0$ }

\subsubsection{The case where $a_{0,0}=0 $}

\begin{proposition}
\label{prop:nec-z(n)-t=0holo-a_{0,0}=0}
{\it
Suppose that for $B_4^{(1)}(\alpha_j)_{0\leq j \leq 4},$ 
$\alpha_4=0$ and 
there exists a rational solution 
such that 
$z$ has a pole of order $(n\geq 2)$ at $t=\infty$ 
and 
$x,y,w$ are all holomorphic at $t=\infty.$ 
Moreover, 
assume that 
$x,y,z,w$ are all holomorphic at $t=0$ 
and 
$a_{0,0}=0.$ 
Then, $\alpha_4=0, \,\,\alpha_0-\alpha_1\in\Z.$
}
\end{proposition}

\begin{proof}
This proposition follows from Proposition \ref{rational-necessary-1}.

\end{proof}

\subsubsection{$a_{0,0}=\alpha_0-\alpha_1\neq 0$}

Let us first treat the case where $x,y,z,w$ are all holomorphic at $t=0$ 
and $a_{0,0}\neq0, b_{0,0}\neq 1/2.$

\begin{proposition}
\label{prop:nec-z(n)-t=0holo-a_{0,0}nonzero-nonhalf}
{\it
Suppose that for $B_4^{(1)}(\alpha_j)_{0\leq j \leq 4},$ 
$\alpha_4=0$ and 
there exists a rational solution 
such that 
$z$ has a pole of order $(n\geq 2)$ at $t=\infty$ 
and 
$x,y,w$ are all holomorphic at $t=\infty.$ 
Moreover, 
assume that 
$x,y,z,w$ are all holomorphic at $t=0$ 
and 
$a_{0,0}=\alpha_0-\alpha_1\neq 0$ 
and $b_{0,0}=-\alpha_1/(\alpha_0-\alpha_1)\neq 1/2.$ 
Then, 
$\alpha_4=0, \,\,\alpha_0+\alpha_1\in\Z.$
}
\end{proposition}

\begin{proof}
Since $a_{0,0}=\alpha_0-\alpha_1\neq 0,$ 
it follows that $\alpha_0\neq 0$ or $\alpha_1\neq 0.$ 
If $\alpha_0\neq 0$ 
by $s_0$ and Proposition \ref{prop:nec-z(n)-t=0holo-a_{0,0}=0}, 
we can prove the proposition. 
\par
If $\alpha_1\neq 0,$ 
using $s_0$ in the same way, 
we can show the proposition. 

\end{proof}

Let us treat the case where $x,y,z,w$ are all holomorphic at $t=0$ 
and $a_{0,0}\neq0, b_{0,0}=1/2.$

\begin{proposition}
\label{prop:nec-z(n)-t=0holo-a_{0,0}nonzero-half}
{\it
Suppose that for $B_4^{(1)}(\alpha_j)_{0\leq j \leq 4},$ 
$\alpha_4=0$ and 
there exists a rational solution 
such that 
$z$ has a pole of order $(n\geq 2)$ at $t=\infty$ 
and 
$x,y,w$ are all holomorphic at $t=\infty.$ 
Moreover, 
assume that 
$x,y,z,w$ are all holomorphic at $t=0$ 
and 
$a_{0,0}=\alpha_0-\alpha_1\neq 0$ 
and $b_{0,0}=1/2.$ 
Either of the following then occurs:
\newline
{\rm (1)}\quad $\alpha_4=0, \,\,\alpha_0+\alpha_1=0,$
\newline
{\rm (2)}\quad $\alpha_4=0, \,\,\alpha_0+\alpha_1+2\alpha_2\in\Z.$
}
\end{proposition}

\begin{proof}
By Proposition \ref{prop:t=0holo-a_{0,0}nonzero-half}, 
we can assume that 
$$
a_{0,0}=\alpha_0-\alpha_1, \,\,b_{0,0}=\frac12, \,\,
c_{0,0}=\frac{(\alpha_0+\alpha_1)(\alpha_0-\alpha_1)}{1-2\alpha_3-2\alpha_4}\neq 0, \,\,
d_{0,0}=-\frac{(1-2\alpha_4)(1-2\alpha_3-2\alpha_4)}{2(\alpha_0+\alpha_1)(\alpha_0-\alpha_1)},
$$
where 
$\alpha_3+\alpha_4\neq 1/2.$ 
\par
Suppose that $\alpha_2\neq 0.$ 
$s_2(x,y,z,w)$ is then a rational solution of 
$B_4^{(1)}(\alpha_0+\alpha_2,\alpha_1+\alpha_2,-\alpha_2,\alpha_3+\alpha_2,\alpha_4)$ 
such that $s_2(z)$ has a pole of order $n \,(n\geq 2)$ at $t=\infty$ and 
all of $s_2(x,y,z,w)$ are holomorphic at $t=0.$ 
Moreover, assume that 
$a_{0,0}=\alpha_0-\alpha_1, \,\,b_{0,0}\neq 1/2.$ 
Thus, by Proposition \ref{prop:nec-z(n)-t=0holo-a_{0,0}nonzero-nonhalf}, 
we can obtain the necessary conditions. 
\par
Suppose that $\alpha_2=0.$ 
Since $\alpha_0-\alpha_1\neq 0,$ 
it follows that $\alpha_0\neq 0$ or $\alpha_1\neq 0.$ 
Let us first assume that $\alpha_0\neq 0.$ 
$s_0(x,y,z,w)$ 
is then a rational solution of 
$B_4^{(1)}(-\alpha_0,\alpha_1,\alpha_0,\alpha_3,\alpha_4)$ 
such that $s_0(z)$ has a pole of order $n \,(n\geq 2)$ at $t=\infty$ and 
all of $s_0(x,y,z,w)$ are holomorphic at $t=0.$ 
Moreover, assume that 
$a_{0,0}=\alpha_0-\alpha_1, \,\,b_{0,0}= 1/2.$ 
Based on the above discussion, 
we can obtain the necessary conditions. 
\par
If $\alpha_2=0$ and $\alpha_1\neq 0,$ 
we use $s_1$ in the same way and 
can obtain the necessary conditions. 
\end{proof}

\subsection{The case where $z$ has a pole of order $m \geq 1$ at $t=0$}

\begin{proposition}
{\it
Suppose that for $B_4^{(1)}(\alpha_j)_{0\leq j \leq 4},$ 
$\alpha_4=0$ and 
there exists a rational solution 
such that 
$z$ has a pole of order $(n\geq 2)$ at $t=\infty$ 
and 
$x,y,w$ are all holomorphic at $t=\infty.$ 
Moreover, 
assume that 
$z$ has a pole of order $m\geq 1$ at $t=0.$ 
One of the following then occurs:
\newline
{\rm (1)}\quad $\alpha_0-\alpha_1=0, \,\,\alpha_4=0,$
\newline
{\rm (2)}\quad $2\alpha_3\in\Z, \,\,\alpha_4=0,$
\newline
{\rm (3)}\quad $\alpha_0+\alpha_1\in\Z, \,\,\alpha_4=0,$
\newline
{\rm (4)}\quad $2\alpha_3\in\Z, \,\,\alpha_4=0.  $
}
\end{proposition}

\begin{proof}
We may assume that $\alpha_3\neq 0.$ 
It then follows from Corollaries \ref{coro:z(n)-s_3} and \ref{coro:t=0-z(n)-s_3} 
that 
$s_3(x,y,z,w)$ is a rational solution of 
$B_4^{(1)}(\alpha_0,\alpha_1,\alpha_2+\alpha_3,-\alpha_3,\alpha_4+\alpha_3)$ 
such that only $s_3(z)$ has a pole of order one at $t=\infty$ 
and 
all of $s_3(x,y,z,w)$ are holomorphic at $t=0.$ 
Moreover, 
for $s_3(x,y,z,w),$ $a_{0,0}=\alpha_0-\alpha_1, \,b_{0,0}=1/2, \,d_{0,0}=0.$ 
We can now assume that $\alpha_0-\alpha_1\neq 0.$ 
Thus, by Proposition \ref{prop:z(1)-t=0-a_{0,0}nonzero-b_{0,0}-half}, 
we can obtain the necessary conditions. 
\end{proof}

\subsection{The case where $y,w$ have a pole at $t=0$}

\begin{proposition}
{\it
If $\alpha_4=0,$ then 
for $B_4^{(1)}(\alpha_j)_{0\leq j \leq 4},$  
there exists no rational solution 
such that 
$z$ has a pole of order $n \,\,(n\geq 2)$ at $t=\infty$ 
and 
$x,y,w$ are all holomorphic at $t=\infty$ 
and 
$y,w$ have a pole at $t=0$ 
and 
$x,z$ are both holomorphic at $t=0.$ 
}
\end{proposition}

\begin{proof}
This proposition follows from Proposition \ref{prop:t=0-(y,w)}.
\end{proof}

\section{The standard forms of the parameters for rational solutions}

Let us summarize the discussion in Section 7, 8.

\begin{proposition}
{\it
Suppose that 
for $B_4^{(1)}(\alpha_j)_{0\leq j \leq 4},$ 
there exists a rational solution. 
One of the following then occurs:
\newline
{\rm (1)}\quad $\alpha_0-\alpha_1\in\Z, \,\,2\alpha_3+2\alpha_4\in \Z,$
\newline
{\rm (2)}\quad $\alpha_0-\alpha_1\in\Z, \,\,2\alpha_4\in \Z,$
\newline
{\rm (3)}\quad $\alpha_0+\alpha_1\in\Z, \,\,2\alpha_3+2\alpha_4\in \Z,$
\newline
{\rm (4)}\quad $\alpha_0+\alpha_1\in\Z, \,\,2\alpha_4\in \Z,$
\newline
{\rm (5)}\quad $\alpha_0-\alpha_1\in\Z, \,\,\alpha_0+\alpha_1 \in \Z,$
\newline
{\rm (6)}\quad $2\alpha_3\in\Z, \,\,2\alpha_4\in \Z.$
}
\end{proposition}

\begin{corollary}
\label{coro-reduction}
{\it
Suppose that 
for $B_4^{(1)}(\alpha_j)_{0\leq j \leq 4},$ 
there exists a rational solution. 
By some B\"acklund transformations, 
the parameters can then be transformed so that 
$\alpha_0-\alpha_1\in\Z, \,\,2\alpha_3+2\alpha_4\in\Z.$
}
\end{corollary}

\subsection{Shift operators}

Following Sasano \cite{Sasano-6}, 
we obtain shift operators of the parameters.

\begin{proposition}
{\it
$T_1, T_2, T_3$ and $T_4$ are defined by 
$$
T_1=s_4\pi_1s_1s_2s_4s_3s_4s_3s_2s_1, \,\,
T_2=s_0T_1s_0, \,\,
T_3=s_2T_2s_2, \,\,
T_4=s_3T_3s_3,
$$
respectively. 
Then, 
\begin{align*}
T_1(\alpha_0, \alpha_1, \alpha_2, \alpha_3, \alpha_4)
&= (\alpha_0, \alpha_1, \alpha_2, \alpha_3, \alpha_4)+(1,-1,0,0,0), \\
T_2(\alpha_0, \alpha_1, \alpha_2, \alpha_3, \alpha_4)
&= (\alpha_0, \alpha_1, \alpha_2, \alpha_3, \alpha_4)+(-1,-1,1,0,0), \\
T_3(\alpha_0, \alpha_1, \alpha_2, \alpha_3, \alpha_4)
&= (\alpha_0, \alpha_1, \alpha_2, \alpha_3, \alpha_4)+(0,0,-1,1,0), \\
T_4(\alpha_0, \alpha_1, \alpha_2, \alpha_3, \alpha_4)
&= (\alpha_0, \alpha_1, \alpha_2, \alpha_3, \alpha_4)+(0,0,0,-1,1).
\end{align*}

}
\end{proposition}

\subsection{Standard forms of the parameters for rational solutions}

\begin{proposition}
\label{prop:standard forms}
{\it
Suppose that 
for $B_4^{(1)}(\alpha_j)_{0\leq j \leq 4},$ 
there exists a rational solution. 
By some B\"acklund transformations, 
the parameters can then be transformed so that 
one of the following occurs: 
{\rm (1)}\quad $\alpha_0-\alpha_1=0, \,\,\alpha_3+\alpha_4=0, \alpha_4\neq 0,$ 
{\rm (2)}\quad $\alpha_0-\alpha_1=0, \,\,\alpha_3+\alpha_4=1/2. $ 
The cases {\rm (1)} and {\rm (2)} denote the standard forms I and II, respectively. 
\par
Especially, 
the parameters can be transformed into the standard form I, 
if they satisfy one of the conditions in our main theorem.  
}
\end{proposition}

\begin{proof}
By Corollary \ref{coro-reduction}, 
we may assume that $\alpha_0-\alpha_1\in\mathbb{Z}$ and $2\alpha_3+2\alpha_4\in\mathbb{Z}.$ 
One of the following cases then occurs:
\begin{align*}
&{\rm (1)}\quad \alpha_0-\alpha_1\equiv2\alpha_3+2\alpha_4\equiv 0 & &{\rm mod} \, 2, & 
&{\rm (2)}\quad \alpha_0-\alpha_1\equiv0, \,\,2\alpha_3+2\alpha_4\equiv 1 & &{\rm mod} \, 2, \\
&{\rm (3)}\quad \alpha_0-\alpha_1\equiv1, \,\,2\alpha_3+2\alpha_4\equiv 0 & &{\rm mod} \, 2, & 
&{\rm (4)}\quad \alpha_0-\alpha_1\equiv2\alpha_3+2\alpha_4\equiv 1 & &{\rm mod} \, 2.
\end{align*}
\par
If case (1) occurs,  
using $T_2, T_3,$ 
we have $\alpha_0-\alpha_1=\alpha_3+\alpha_4=0.$ 
\par
If case (2) occurs,  
using $T_2, T_3,$ 
we obtain $\alpha_0-\alpha_1=0, \,\alpha_3+\alpha_4=1/2.$  
\par
If case (3) occurs,  
using $T_2, T_3,$ 
we have $\alpha_0-\alpha_1=1, \,\alpha_3+\alpha_4=0.$ 
Furthermore, 
by $\pi_2s_3s_1,$ 
we obtain $\alpha_0-\alpha_1=0, \,\alpha_3+\alpha_4=1/2.$
\par
If case (4) occurs,  
using $T_2, T_3,$ 
we have $\alpha_0-\alpha_1=\alpha_3+\alpha_4=1,$
which implies that 
$$
(\alpha_0,\alpha_1,\alpha_2, \alpha_3, \alpha_4)=
(-\alpha_2+1/2,-\alpha_2-1/2,\alpha_2,\alpha_3,-\alpha_3+1/2).
$$ 
By $s_2s_1s_2T_3,$ 
we obtain 
$$
(-\alpha_2+1/2,-\alpha_2-1/2,\alpha_2,\alpha_3,-\alpha_3+1/2)
\longrightarrow
(-\alpha_2,-\alpha_2,\alpha_2-1/2,\alpha_3+1/2,-\alpha_3+1/2).
$$
\par
If $\alpha_0-\alpha_1=\alpha_3+\alpha_4=0$, 
by $T_4,$ 
we may assume that $\alpha_0-\alpha_1=\alpha_3+\alpha_4=0$ and $\alpha_4\neq 0.$ 

\end{proof}

If $\alpha_0-\alpha_1=0$ and $\alpha_3+\alpha_4=0,$ 
it follows 
from Corollary 
\ref{coro:solution-1} 
that 
for $B_4^{(1)}(\alpha_j)_{0\leq j \leq 4},$ 
there exists a unique rational solution which is expressed by 
\begin{equation*}
x\equiv 0, \,y\equiv 1/2, \, z=1/\{2\alpha_4\}\cdot t, \,w\equiv 0.
\end{equation*}
We have then to only 
treat the standard form II  
in order to 
classify the rational solutions of $B_4^{(1)}(\alpha_j)_{0\leq j \leq 4}.$

\section{The standard form II $\cdots$ (1)}

\subsection{The case where $z$ has a pole of order $n \,\,(n\geq 2)$ at $t=\infty$}

By Corollary \ref{coro:order-n}, 
we can prove the following proposition:

\begin{proposition}
{\it
Suppose that 
$\alpha_0-\alpha_1=0, \,\,\alpha_3+\alpha_4=1/2$ 
and 
for $B_4^{(1)}(\alpha_j)_{0\leq j \leq 4},$ 
there exists a rational solution. 
Moreover, 
assume that $z$ has a pole of order $n \,\,(n\geq 2)$ at $t=\infty$ 
and $x,y,w $ are all holomorphic at $t=\infty.$ 
One of the following then occurs:
\newline
{\rm (1)}\quad $\alpha_0-\alpha_1=0, \,\,\alpha_3=1/2, \,\,\alpha_4=0,$
\newline
{\rm (2)}\quad $\alpha_0=\alpha_1=\alpha_2=\alpha_3=0, \,\,\alpha_4=1/2.$
}
\end{proposition}

\subsection{The case where $z$ has a pole of order one at $t=\infty$}

\begin{proposition}
{\it
Suppose that 
$\alpha_0-\alpha_1=0, \,\,\alpha_3+\alpha_4=1/2$ 
and 
for $B_4^{(1)}(\alpha_j)_{0\leq j \leq 4},$ 
there exists a rational solution 
such that 
$z$ has a pole of order one at $t=\infty$ 
and $x,y,w $ are all holomorphic at $t=\infty.$ 
One of the following then occurs:
\newline
{\rm (1)}\quad $x,y,z,w$ are all holomorphic at $t=0$ and $a_{0,0}=0, \,\,c_{0,0}\neq 0,$
\newline
{\rm (2)}\quad $z$ has a pole at $t=0$ and $x,y,w$ are all holomorphic at $t=0,$
\newline
{\rm (3)}\quad $y,w$ both have a pole of order one at $t=0$ and $x,z$ are holomorphic at $t=0.$
}
\end{proposition}

\begin{proof}
If $x,y,z,w$ are all holomorphic at $t=0,$ 
it follows from Proposition \ref{prop:a0-determine} 
that 
$a_{0,0}=0,$ because $\alpha_0-\alpha_1=0.$ 
Moreover, 
\begin{equation*}
h_{0,0}=
\begin{cases}
0, \,\, &{\rm if} \,\,c_{0,0}=0,   \\
\alpha_3(\alpha_3+2\alpha_4-1), \,\, &{\rm if} \,\,c_{0,0}\neq 0.
\end{cases}
\end{equation*}
\par
On the other hand, $h_{\infty,0}=-1/4.$ Thus, 
$c_{0,0}\neq 0$ because $h_{\infty,0}-h_{0,0}\in\Z.$

\end{proof}

\subsubsection{The case where $x,y,z,w$ are all holomorphic at $t=0$}

By Proposition \ref{prop:nec-order1-holozero-zero-nonzero}, 
we have the following proposition:

\begin{proposition}
{\it
Suppose that 
$\alpha_0-\alpha_1=0, \,\,\alpha_3+\alpha_4=1/2$ 
and 
for $B_4^{(1)}(\alpha_j)_{0\leq j \leq 4},$ 
there exists a rational solution 
such that 
$z$ has a pole of order one at $t=\infty$ 
and $x,y,w $ are all holomorphic at $t=\infty.$ 
Moreover, 
assume that 
$x,y,z,w$ are all holomorphic at $t=0.$ 
Then, 
$\alpha_0-\alpha_1=0, \,\,\alpha_3+\alpha_4=1/2, \,\,2\alpha_4\in\Z.$
}
\end{proposition}

\subsubsection{The case where $z$ has a pole at $t=0$}

\begin{proposition}
{\it
Suppose that 
$\alpha_0-\alpha_1=0, \,\,\alpha_3+\alpha_4=1/2$ 
and 
for $B_4^{(1)}(\alpha_j)_{0\leq j \leq 4},$ 
there exists a rational solution 
such that 
$z$ has a pole of order one at $t=\infty$ 
and $x,y,w $ are all holomorphic at $t=\infty.$ 
Moreover, 
assume that 
$z$ has a pole at $t=0$ and $x,y,w$ are all holomorphic at $t=0.$
One of the following then occurs:
\newline
{\rm (1)}\quad 
$\alpha_0=\alpha_1=\alpha_2=0, \,\,\alpha_3+\alpha_4=1/2, $
\newline
{\rm (2)}\quad 
$\alpha_0-\alpha_1=0, \,\,\alpha_3+\alpha_4=1/2, \,\,\alpha_4\in\Z.$
}
\end{proposition}

\begin{proof}
We may assume that $\alpha_0=\alpha_1\neq 0$ and $\alpha_3\neq 1/2.$ 
Moreover, 
from Propositions \ref{prop:hinf} and \ref{prop:hzero-2}, 
it follows that 
$h_{\infty,0}-h_{0,0}=\alpha_4^2-\alpha_4\in\Z,$ 
which implies that $\alpha_3\not\in \Z.$ 
\par
$s_3(x,y,z,w)$ is then a rational solution of $B_4^{(1)}(\alpha_0,\alpha_1,\alpha_2+\alpha_3,-\alpha_3,1/2)$ 
such that $s_3(z)$ has a pole of order one at $t=\infty$ and all of $s_3(x,y,z,w)$ are 
holomorphic at $t=0$ and $a_{0,0}=0.$ 
If $c_{0,0}\neq 0$ for $s_3(x,y,z,w),$ 
it follows from Propositions \ref{prop:hinf} and \ref{prop:hzero-1} 
that $h_{\infty,0}-h_{0,0}=-1/4 \in\Z,$ 
which is impossible. 
Thus, we have $a_{0,0}=c_{0,0}=0$ for $s_3(x,y,z,w).$
\par
Now, let us assume that $\alpha_2+\alpha_3\neq 0.$ 
It then follows from Lemmas \ref{lem:holo-1} and \ref{lem:holo-2} 
that 
$s_3\pi_2(x,y,z,w)$ is a rational solution of 
$B_4^{(1)}(1-\alpha_3,-\alpha_3,\alpha_2+\alpha_3,\alpha_1,0)$ 
such that $s_3\pi_2(z)$ has a pole of order $n \,(n\geq 2)$ at $t=\infty$ 
and 
all of $s_3\pi_2(x,y,z,w)$ are holomorphic at $t=0.$ 
Moreover, 
for $s_3\pi_2(x,y,z,w),$ $a_{0,0}=1, \,b_{0,0}=\alpha_3, \,c_{0,0}\neq 0. $ 
Thus, by Proposition \ref{prop:nec-z(n)-t=0holo-a_{0,0}nonzero-nonhalf}, 
we find that $1-2\alpha_3\in\Z,$ which implies that $\alpha_4\in\Z$ 
because $\alpha_3+\alpha_4=1/2$ and $\alpha_3\not\in \Z.$ 
\par
Let us treat the case where $\alpha_2+\alpha_3=0.$ 
$s_0s_1s_3(x,y,z,w)$ is then a rational solution of $B_4^{(1)}(-\alpha_0,-\alpha_1,2\alpha_0,-\alpha_3,1/2)$ 
such that 
$s_0s_1s_3(z)$ has a pole of order one at $t=\infty$ and all of $s_1s_0s_3(x,y,z,w)$ 
are holomorphic at $t=0.$ 
Furthermore, for $s_0s_1s_3(x,y,z,w),$ $a_{0,0}=c_{0,0}=0.$ 
Thus, by $\pi_2$, we have $\alpha_4\in\Z.$ 
\end{proof}

\subsubsection{The case where $y,w$ both have a pole of order one at $t=0$   }

\begin{proposition}
{\it
Suppose that 
$\alpha_0-\alpha_1=0, \,\,\alpha_3+\alpha_4=1/2$ 
and 
for $B_4^{(1)}(\alpha_j)_{0\leq j \leq 4},$ 
there exists a rational solution 
such that 
$z$ has a pole of order one at $t=\infty$ 
and $x,y,w $ are all holomorphic at $t=\infty.$ 
Moreover, 
assume that 
$y,w$ both have a pole of order one at $t=0$ and $x,z$ are both holomorphic at $t=0.$ 
Then, 
$\alpha_0-\alpha_1=0, \,\,\alpha_3+\alpha_4=1/2, \,\,\alpha_0+\alpha_1\in\Z.$
}
\end{proposition}

\begin{proof}
We may assume that $\alpha_2+\alpha_3+\alpha_4\neq 0.$ 
Moreover, 
by Proposition \ref{prop:exam2}, we can suppose that $x\not\equiv 0.$ 
\par
$\pi_2(x,y,z,w)$ is then a rational solution of 
$B_4^{(1)}(2\alpha_4+\alpha_3,\alpha_3,\alpha_2,\alpha_1,0)$ 
such that 
$\pi_2(z)$ has a pole of order $n \,(n\geq 2)$ at $t=\infty$ 
and 
all of $\pi_2(x,y,z,w)$ are holomorphic at $t=0.$ 
Moreover, 
for $\pi_2(x,y,z,w),$ 
$a_{0,0}=2\alpha_4, b_{0,0}=1/2.$ 
By Proposition \ref{prop:nec-z(n)-t=0holo-a_{0,0}nonzero-half}, 
we then obtain the necessary condition.
\end{proof}

\subsection{Summary}

\begin{proposition}
\label{prop:standard form II (1)}
{\it
Suppose that 
$\alpha_0-\alpha_1=0, \,\,\alpha_3+\alpha_4=1/2$ 
and 
for $B_4^{(1)}(\alpha_j)_{0\leq j \leq 4},$ 
there exists a rational solution. 
One of the following then occurs:
\newline
{\rm (1)}\quad $\alpha_0-\alpha_1=0, \,\,\alpha_3+\alpha_4=1/2, \,\,2\alpha_4\in\Z,$
\newline
{\rm (2)}\quad $\alpha_0-\alpha_1=0, \,\,\alpha_3+\alpha_4=1/2, \,\,\alpha_0+\alpha_1\in\Z.$ 
}
\end{proposition}

\begin{corollary}
\label{coro:standard form II (1)}
{\it
Suppose that 
$\alpha_0-\alpha_1=0, \,\,\alpha_3+\alpha_4=1/2$ 
and 
for $B_4^{(1)}(\alpha_j)_{0\leq j \leq 4},$ 
there exists a rational solution. 
By some B\"acklund transformations, 
the parameters can then be transformed 
so that one of the following occurs:
\newline
{\rm (1)}\quad $\alpha_0-\alpha_1=0, \,\,\alpha_3+\alpha_4=0,\,\alpha_4\neq 0,$
\newline
{\rm (2)}\quad $\alpha_0=\alpha_1=\alpha_2=0, \,\,\alpha_3+\alpha_4=1/2. $ 
\par
Especially, 
the parameters can be transformed into those of the standard form I 
if they satisfy one of the following:
\newline
(i)\quad $\alpha_0-\alpha_1=0, \,\,\alpha_3+\alpha_4=1/2, \,\,\alpha_4\in\Z,$ 
\newline
(ii)\quad $\alpha_0-\alpha_1=0, \,\,\alpha_3+\alpha_4=1/2, \,\,\alpha_0+\alpha_1\in\Z, \,\alpha_0+\alpha_1 \equiv 1 \,\mathrm{mod} \, 2.$
\par
If the above cases do not occur, 
the parameters can be transformed into those of {\rm (2)}. 
}
\end{corollary}

\section{Standard forms II$\cdots$ (2)}

In this section, 
we treat the case where 
$\alpha_0=\alpha_1=\alpha_2=0, \,\,\alpha_3+\alpha_4=1/2. $

\subsection{The case where $z$ has a pole of order $n \,(n\geq 2)$ at $t=\infty$}
By Proposition \ref{coro:order-n}, 
we can first prove the following proposition:

\begin{proposition}
{\it
Suppose that 
$\alpha_0=\alpha_1=\alpha_2=0, \,\,\alpha_3+\alpha_4=1/2$ 
and 
for $B_4^{(1)}(\alpha_j)_{0\leq j \leq 4},$ 
there exists a rational solution 
such that $z$ has a pole of order $n \,(n\geq 2)$ at $t=\infty$ 
and all of $x,y,w$ are holomorphic at $t=\infty.$ 
One of the following then occurs:
\newline
{\rm (1)}\quad $\alpha_0=\alpha_1=\alpha_2=\alpha_3=0, \,\,\alpha_4=1/2,$
\newline
{\rm (2)}\quad $\alpha_0=\alpha_1=\alpha_2=0, \,\alpha_3=1/2, \,\alpha_4=0.$
}
\end{proposition}

\subsection{The case where $z$ has a pole of order one at $t=\infty$}

In order to treat the case where $z$ has a pole of order one at $t=\infty,$ 
we prove the following lemma:

\begin{lemma}
{\it
Suppose that 
$\alpha_0=\alpha_1=\alpha_2=0, \,\,\alpha_3+\alpha_4=1/2$ 
and 
for $B_4^{(1)}(\alpha_j)_{0\leq j \leq 4},$ 
there exists a rational solution 
such that 
$z$ has a pole of order one at $t=\infty$ 
and 
$x,y,w$ are all holomorphic at $t=\infty.$ 
Moreover, 
assume that 
all of $x,y,w$ are holomorphic at $t=\infty.$ 
One of the following then occurs:
\newline
{\rm (1)}\quad $x,y,z,w$ are all holomorphic at $t=0$ and $c_{0,0}\neq0,$
\newline
{\rm (2)}\quad only $z$ has a pole at $t=0.$
}
\end{lemma}

\begin{proof}
From Proposition \ref{prop:hinf}, 
we find that $h_{\infty,0}=-1/4.$ 
Thus, the lemma follows from Propositions \ref{prop:hzero-1}, \ref{prop:hzero-2}, \ref{prop:hzero-3} and \ref{rational-necessary-2}.

\end{proof}

\subsubsection{The case where $x,y,z,w$ are all holomorphic at $t=0$}

By Proposition \ref{prop:nec-order1-holozero-zero-nonzero}, 
we obtain the following proposition:

\begin{proposition}
{\it
Suppose that 
$\alpha_0=\alpha_1=\alpha_2=0, \,\,\alpha_3+\alpha_4=1/2$ 
and 
for $B_4^{(1)}(\alpha_j)_{0\leq j \leq 4},$ 
there exists a rational solution 
such that 
$z$ has a pole of order one at $t=\infty$ 
and $x,y,w$ are all holomorphic at $t=\infty.$ 
Moreover, 
assume that 
$x,y,z,w$ are all holomorphic at $t=0$ and $c_{0,0}\neq0.$ 
Then, $\alpha_0=\alpha_1=\alpha_2=0, \,\,\alpha_3+\alpha_4=1/2, \,\,2\alpha_4\in\Z.$
}
\end{proposition}

\subsubsection{The case where $z$ has a pole at $t=0$}

\begin{proposition}
{\it
Suppose that 
$\alpha_0=\alpha_1=\alpha_2=0, \,\,\alpha_3+\alpha_4=1/2$ 
and 
for $B_4^{(1)}(\alpha_j)_{0\leq j \leq 4},$ 
there exists a rational solution 
such that 
$z$ has a pole of order one at $t=\infty$ 
and $x,y,w$ are all holomorphic at $t=\infty.$ 
Moreover, 
assume that 
only $z$ has a pole at $t=0.$ 
Then, 
$\alpha_0=\alpha_1=\alpha_2=0, \,\,\alpha_3+\alpha_4=1/2, \,\alpha_4\in\Z.$
}
\end{proposition}

\begin{proof}
We may assume that $\alpha_3\neq 1/2.$ 
Moreover, 
from Propositions \ref{prop:hinf} and \ref{prop:hzero-2}, it follows that 
$h_{\infty,0}-h_{0,0}=\alpha_4^2-\alpha_4\in\Z,$ 
which implies that $\alpha_3\not\in\Z$ because $\alpha_3+\alpha_4=1/2.$
\par
$s_3(x,y,z,w)$ is then a rational solution of $B_4^{(1)}(0,0,\alpha_3,-\alpha_3,1/2)$ 
such that 
$s_3(z)$ has a pole of order one at $t=\infty$ 
and all of $s_3(x,y,z,w)$ are holomorphic at $t=0.$ 
Furthermore, 
for $s_3(x,y,z,w),$ $a_{0,0}=0.$ 
If $c_{0,0}\neq0$ for $s_3(x,y,z,w),$ 
it follows  
from Propositions \ref{prop:hinf} and \ref{prop:hzero-1} 
that $h_{\infty,0}-h_{0,0}=-1/4\in\Z,$ which is impossible. 
Thus, for $s_3(x,y,z,w),$ $a_{0,0}=c_{0,0}=0.$ 
Therefore, 
from Lemmas \ref{lem:holo-1} and \ref{lem:holo-2}, 
we find that 
$s_3\pi_2(x,y,z,w)$ is a rational solution of $B_4^{(1)}(1-\alpha_3,-\alpha_3,\alpha_3,0,0)$ 
such that 
$s_3\pi_2(z)$ has a pole of order $n \,(n\geq 2)$ at $t=\infty$ and 
all of $s_3\pi_2(x,y,z,w)$ are holomorphic at $t=0.$ 
Moreover, for $\pi_2s_3(x,y,z,w),$ $a_{0,0}=1, b_{0,0}=\alpha_3,$ 
which implies that $1-2\alpha_3\in\Z,$ that is, $\alpha_4\in\Z$ from Proposition \ref{prop:nec-z(n)-t=0holo-a_{0,0}nonzero-nonhalf}. 
\end{proof}

\subsection{Summary}

Let us summarize the results in this section.

\begin{proposition}
\label{prop:standard form II (2)}
{\it
Suppose that 
$\alpha_0=\alpha_1=\alpha_2=0, \,\,\alpha_3+\alpha_4=1/2$ 
and 
for $B_4^{(1)}(\alpha_j)_{0\leq j \leq 4},$ 
there exists a rational solution. 
Then, $\alpha_0=\alpha_1=\alpha_2=0, \,\,\alpha_3+\alpha_4=1/2, \,2\alpha_4\in\Z.$
}
\end{proposition}

\begin{corollary}
\label{coro:standard form II (2)}
{\it
Suppose that 
$\alpha_0=\alpha_1=\alpha_2=0, \,\,\alpha_3+\alpha_4=1/2$ 
and 
for $B_4^{(1)}(\alpha_j)_{0\leq j \leq 4},$ 
there exists a rational solution. 
By some B\"acklund transformations, 
the parameters can then be transformed so that 
one of the following occurs:
\newline
{\rm (1)}\quad $\alpha_0-\alpha_1=0, \,\,\alpha_3+\alpha_4=0,\,\alpha_4\neq 0,$
\newline
{\rm (2)}\quad $\alpha_0=\alpha_1=\alpha_2=\alpha_3=0, \,\,\alpha_4=1/2. $ 
\par
Especially, 
the parameters can be transformed into the standard form I, 
if $\alpha_0=\alpha_1=\alpha_2=0, \,\,\alpha_3+\alpha_4=1/2, \,\alpha_4\in\Z.$ 
Otherwise, 
the parameters can be transformed into those of case (2). 
}
\end{corollary}

\section{Standard form II$\cdots$ (3)}
In this section, 
we treat the case where 
$\alpha_0=\alpha_1=\alpha_2=\alpha_3=0, \,\alpha_4=1/2. $ 
For this purpose, 
let us first prove the following lemma:

\begin{lemma}
\label{lem:point}
{\it
Suppose that 
for $B_4^{(1)} (0,0,0,0,1/2),$ 
there exists a rational solution. 
$z$ then has a pole of order $n \,\,(n\geq 2)$ at $t=\infty.$ 
}
\end{lemma}

\begin{proof}
Suppose that $z$ has a pole of order one at $t=\infty.$ 
By Proposition \ref{prop:hinf}, 
we then see that $h_{\infty,0}=-1/4.$ 
\par
On the other hand, 
by Propositions \ref{prop:hzero-1}, \ref{prop:hzero-2}, and \ref{prop:hzero-3}, 
we find that $h_{0,0}=0.$ 
Thus, 
it follows that $h_{\infty,0}-h_{0,0}=-1/4 \not\in\Z,$ 
which contradicts Proposition \ref{rational-necessary-2}.

\end{proof}

We can then prove the following proposition:

\begin{proposition}
\label{prop:standard form II (3)}
{\it
For $B_4^{(1)}(0,0,0,0,1/2),$ 
there exists no rational solution. 
}
\end{proposition}

\begin{proof}
Suppose that 
for $B_4^{(1)} (0,0,0,0,1/2),$ 
there exists a rational solution. 
By Lemma \ref{lem:point}, 
we can then assume that $z$ has a pole of order $n \,\,(n\geq 2)$ at $t=\infty.$ 
\par
$s_4s_3(x,y,z,w)$ is a rational solution of 
$B_4^{(1)}(0,0,1,-1,1/2)$ such that $z$ has a pole of order one at $t=\infty.$
It then follows that for $s_4s_3(x,y,z,w),$ 
$h_{\infty,0}=3/4.$ 
On the other hand, for $s_4s_3(x,y,z,w),$ 
$h_{0,0}\in\Z.$ 
Thus, it follows that $h_{\infty,0}-h_{0,0}=3/4\not\in\Z,$ 
which contradicts Proposition \ref {rational-necessary-2}.

\end{proof}

\section{Proof of main theorem}
In this section, 
we prove the main theorem.

\begin{proof}
Suppose that 
for $B_4^{(1)}(\alpha_j)_{0\leq j \leq 4}$ 
there exists a rational solution. 
From Proposition \ref{prop:standard forms},  and Corollaries \ref{coro:standard form II (1)} and \ref{coro:standard form II (2)},
it follows that 
the parameters can be transformed so that 
one of the following occurs: 
\newline
\hspace{15mm}
{\rm (1)} \quad $\alpha_0-\alpha_1=0, \,\alpha_3+\alpha_4=0, \,\alpha_4\neq 0,$ 
{\rm (2)} \quad $\alpha_0=\alpha_1=\alpha_2=\alpha_3=0, \,\alpha_4=1/2.$
\newline
Especially, 
from Proposition \ref{prop:standard forms}, and Corollaries \ref{coro:standard form II (1)} and \ref{coro:standard form II (2)}, 
we see that 
the parameters can be transformed into the standard form I 
if they satisfy one of the conditions in this theorem. 
Otherwise, 
the parameters can be transformed into those of (2). 
\par
If $\alpha_0-\alpha_1=0, \,\alpha_3+\alpha_4=0, \,\alpha_4\neq 0,$ 
then 
from Corollary \ref{coro:solution-1}, 
we observe that 
for $B_4^{(1)}(\alpha_j)_{0\leq j \leq 4}$ 
there exists a rational solution and $
x\equiv 0, \,y\equiv 1/2, \,
z=1/\{2\alpha_4\}\cdot t, \,
w\equiv 0
$ 
and it is unique. 
\par
From Proposition \ref{prop:standard form II (3)}, 
we see that 
for $B_4^{(1)}(0,0,0,0,1/2),$ 
there exists no rational solution.
\end{proof}

\appendix
\section{Rational solutions of the Sasano system of type $D_4^{(1)}$}
Following Sasano \cite{Sasano-2}, 
we introduce the Sasano system of type $D_4^{(1)},$ 
which is defined by 
\begin{equation*}
D_4^{(1)}(\alpha_j)_{0\leq j \leq 4}
\begin{cases}
tx^{\prime}=2x^2y-x^2+(\alpha_0+\alpha_1)x-2w+t, \\
ty^{\prime}=-2xy^2+2xy-(\alpha_0+\alpha_1)y+\alpha_1, \\
tz^{\prime}=2z^2w-tz^2-(1-\alpha_3-\alpha_4)z+1-2y, \\
tw^{\prime}=-2zw^2+2tzw+(1-\alpha_3-\alpha_4)w+\alpha_3t,  \\
\alpha_0+\alpha_1+2\alpha_2+\alpha_3+\alpha_4=1.
\end{cases}
\end{equation*}
\par
$D_4^{(1)}(\alpha_j)_{0\leq j \leq 4}$ has the B\"acklund transformations, 
$s_0, s_1, s_2, s_3, s_4, \pi_1, \pi_2, \pi_3, \pi_4,$ 
which are given by 
\begin{align*}
&s_0: \,\,
(x,y,z,w,t;\alpha_0,\alpha_1,\alpha_2,\alpha_3,\alpha_4)
\rightarrow
\left(
x+\frac{\alpha_0}{y-1},y,z,w,t;
-\alpha_0, \alpha_1, \alpha_2+\alpha_0, \alpha_3, \alpha_4
\right),              \\
&s_1: \,\,
(x,y,z,w,t;\alpha_0,\alpha_1,\alpha_2,\alpha_3,\alpha_4)
\rightarrow
\left(
x+\frac{\alpha_1}{y},y,z,w,t;
\alpha_0,-\alpha_1,\alpha_2+\alpha_1,\alpha_3,\alpha_4
\right),              \\
&s_2: \,\,
(x,y,z,w,t;\alpha_0,\alpha_1,\alpha_2,\alpha_3,\alpha_4)
\rightarrow \\
&\hspace{30mm}
\left(
x, y-\frac{\alpha_2z}{xz-1}, z,w-\frac{\alpha_2x}{xz-1}, t; 
\alpha_0+\alpha_2, \alpha_1+\alpha_2, -\alpha_2, \alpha_3+\alpha_2, \alpha_4+\alpha_2
\right), \\
&s_3: \,\,
(x,y,z,w,t;\alpha_0,\alpha_1,\alpha_2,\alpha_3,\alpha_4)
\rightarrow
\left(
x,y,z+\frac{\alpha_3}{w},w,t;
\alpha_0,\alpha_1,\alpha_2+\alpha_3,-\alpha_3,\alpha_4
\right),              \\
&s_4: \,\,
(x,y,z,w,t;\alpha_0,\alpha_1,\alpha_2,\alpha_3,\alpha_4)
\rightarrow
\left(
x,y,z+\frac{\alpha_4}{w-t},w,t;
\alpha_0,\alpha_1,\alpha_2+\alpha_4,\alpha_3,-\alpha_4
\right),              \\
&\pi_1: \,\,
(x,y,z,w,t;\alpha_0,\alpha_1,\alpha_2,\alpha_3,\alpha_4)
\rightarrow
(-x,1-y,-z,-w,-t; 
\alpha_1, \alpha_0, \alpha_2, \alpha_3, \alpha_4
), \\
&\pi_2: \,\,
(x,y,z,w,t;\alpha_0,\alpha_1,\alpha_2,\alpha_3,\alpha_4)
\rightarrow
(x,y,z,w-t,-t;\alpha_0,\alpha_1,\alpha_2,\alpha_4,\alpha_3),  \\
&\pi_3: \,\,
(x,y,z,w,t;\alpha_0,\alpha_1,\alpha_2,\alpha_3,\alpha_4)
\rightarrow
\left(
tz, \frac{w}{t}, \frac{x}{t}, ty,t; 
\alpha_4, \alpha_3, \alpha_2, \alpha_1, \alpha_0
\right), \\
&\pi_4: \,\,
(x,y,z,w,t;\alpha_0,\alpha_1,\alpha_2,\alpha_3,\alpha_4)
\rightarrow
\left(
-tz, \frac{t-w}{t}, -\frac{x}{t}, t-ty, t; 
\alpha_3, \alpha_4, \alpha_2, \alpha_0, \alpha_1
\right). 
\end{align*}

The B\"acklund transformation group 
$\langle s_0, s_1, s_2, s_3, s_4, \pi_1, \pi_2, \pi_3, \pi_4 \rangle$ is isomorphic to 
$\tilde{W}(D_4^{(1)}).$

\subsection{The properties of the B\"acklund transformations}

By direct calculation, 
we obtain the following proposition:

\begin{proposition}
\label{prop:D4-transformation}
{\it
{\rm (0)}\quad If $y\equiv 1$ for $D_4^{(1)}(\alpha_j)_{0\leq j \leq 4},$ then $\alpha_0=0.$ 
\newline
{\rm (1)}\quad If $y\equiv 0$ for $D_4^{(1)}(\alpha_j)_{0\leq j \leq 4},$ then $\alpha_1=0.$
\newline
{\rm (2)}\quad If $xz\equiv 1$ for $D_4^{(1)}(\alpha_j)_{0\leq j \leq 4},$ then $\alpha_2=0.$
\newline
{\rm (3)}\quad If $w\equiv 0$ for $D_4^{(1)}(\alpha_j)_{0\leq j \leq 4},$ then $\alpha_3=0.$
\newline
{\rm (4)}\quad If $w=t$ for $D_4^{(1)}(\alpha_j)_{0\leq j \leq 4},$ then $\alpha_4=0.$
}
\end{proposition}
By Proposition \ref{prop:D4-transformation}, 
we do not have to consider the infinite solutions of $D_4^{(1)}(\alpha_j)_{0\leq j \leq 4}.$

\subsection{Main theorem for $D_4^{(1)}(\alpha_j)_{0\leq j \leq 4}$}

Sasano \cite{Sasano-2} proved that $D_4^{(1)}(\alpha_j)_{0\leq j \leq 4}$ is equivalent to $B_4^{(1)}(\alpha_j)_{0\leq j \leq 4}.$ 
%%%%%%%%%%%%%%%%
\begin{proposition}
\label{prop:equivalece-D_4-B_4}
{\it
Suppose that $(x,y,z,w)$ is a solution of $D_4^{(1)}(\alpha_j)_{0\leq j \leq 4}$ and 
\begin{align*}
&X=x, Y=y, Z=\frac{1}{z}, W=-(zw+\alpha_3)z, \\
&A_0=\alpha_0, A_1=\alpha_1, A_2=\alpha_2, A_3=\alpha_3, A_4=\frac{\alpha_4-\alpha_3}{2}.
\end{align*}
$(X,Y,Z,W)$ is then a solution of $B_4^{(1)}(A_j)_{0\leq j \leq 4}.$
}
\end{proposition}

\begin{theorem}
\label{thm:D4}
{\it
Suppose that $D_4^{(1)}(\alpha_j)_{0\leq j \leq 4}$ has a rational solution. 
By some B\"acklund transformations, 
the parameters and solution can then be transformed so that 
$$
\alpha_0-\alpha_1=\alpha_3+\alpha_4=0, \,\,{\it and} \,\,
(x,y,z,w)=(0,1/2,2\alpha_4/t, t/2).
$$
Moreover, $D_4^{(1)}(\alpha_j)_{0\leq j \leq 4}$ has a rational solution 
if and only if one of the following occurs: 
\begin{align*}
&{\rm (1)} &  &\alpha_0-\alpha_1\in\Z, & &\alpha_3+\alpha_4\in \Z, & &\alpha_0-\alpha_1\equiv \alpha_3+\alpha_4  & &\mathrm{ mod} \,\,2,   \\
%\newline
&{\rm (2)} & &\alpha_0-\alpha_1\in\Z, & &\alpha_3-\alpha_4\in \Z,               &  &\alpha_0-\alpha_1\equiv \alpha_3-\alpha_4 & &\mathrm{ mod} \,\,2,  \\
%\newline
&{\rm (3)} & &\alpha_0+\alpha_1\in\Z, & &\alpha_3+\alpha_4\in \Z,  &  &\alpha_0+\alpha_1\equiv \alpha_3+\alpha_4 & &\mathrm{ mod} \,\,2,  \\
%\newline
&{\rm (4)} & &\alpha_0+\alpha_1\in\Z, & &\alpha_3-\alpha_4\in \Z,               &   &\alpha_0+\alpha_1\equiv \alpha_3-\alpha_4 & &\mathrm{ mod} \,\,2,  \\
&{\rm (5)} & &\alpha_0-\alpha_1\in\Z, & &\alpha_0+\alpha_1 \in \Z,   &    &\alpha_0-\alpha_1\not\equiv \alpha_0+\alpha_1 & &\mathrm{ mod} \,\,2,  \\
%\newline
&{\rm (6)}  &  &\alpha_3-\alpha_4\in\Z, & &\alpha_3+\alpha_4\in \Z,                      &     &\alpha_3-\alpha_4 \not\equiv \alpha_3+\alpha_4 & &\mathrm{ mod} \,\,2.
\end{align*}
}
\end{theorem}

{\bf Remark}\quad 
In Theorem \ref{thm:D4}, we can assume that $\alpha_1, \alpha_4\neq 0.$ 
For this purpose, following Sasano \cite{Sasano-2}, we define 
\begin{align*}
&T_1:=s_3s_0s_2s_4s_1s_2\pi_4, & &T_2:=s_4s_1s_2s_3s_0s_2\pi_4, \\
&T_3:=s_3s_2s_0s_1s_2s_3\pi_1\pi_2, & &T_4:=s_4s_3s_2s_1s_0s_2\pi_1\pi_2,
\end{align*}
which implies that
\begin{align*}
&T_1(\alpha_0, \alpha_1, \alpha_2, \alpha_3, \alpha_4)=(\alpha_0, \alpha_1, \alpha_2, \alpha_3, \alpha_4)+(1,0,-1,1,0),  \\
&T_2(\alpha_0, \alpha_1, \alpha_2, \alpha_3, \alpha_4)=(\alpha_0, \alpha_1, \alpha_2, \alpha_3, \alpha_4)+(0,1,-1,0,1), \\
&T_3(\alpha_0, \alpha_1, \alpha_2, \alpha_3, \alpha_4)=(\alpha_0, \alpha_1, \alpha_2, \alpha_3, \alpha_4)+(0,0,0,1,-1), \\
&T_4(\alpha_0, \alpha_1, \alpha_2, \alpha_3, \alpha_4)=(\alpha_0, \alpha_1, \alpha_2, \alpha_3, \alpha_4)+(0,0,-1,1,1).
\end{align*}
By $T_3,$ we may suppose that $\alpha_4\neq 0.$ 
By $T_1T_2T_4^{-1},$ we can assume that $\alpha_1\neq 0.$

\section{Rational solutions of the Sasano system of type $D_5^{(2)}$}

\begin{equation*}
D_5^{(2)}(\alpha_j)_{0\leq j \leq 4}
\begin{cases}
tx^{\prime}=2x^2y-tx^2-2\alpha_0x+1-2x^2z(zw+\alpha_3), \\
ty^{\prime}=-2xy^2+2txy+2\alpha_0y+\alpha_1t+2z(zw+\alpha_3)(2xy+\alpha_1),  \\
tz^{\prime}=2z^2w-z^2+(1-2\alpha_4)z+t-2xz^2(xy+\alpha_1),  \\
tw^{\prime}=-2zw^2+2zw-(1-2\alpha_4)w+\alpha_3+2x(xy+\alpha_1)(2zw+\alpha_3), \\
\alpha_0+\alpha_1+\alpha_2+\alpha_3+\alpha_4=1/2.
\end{cases}
\end{equation*}

\begin{align*}
&s_0: (x,y,z,w,t;\alpha_0, \alpha_1, \alpha_2, \alpha_3, \alpha_4) 
\longrightarrow 
\left(
-x,-y+\frac{2\alpha_0}{x}-\frac{1}{x^2},-z,-w,-t; 
-\alpha_0,\alpha_1+2\alpha_0, \alpha_2, \alpha_3, \alpha_4
\right), \\
&s_1: (x,y,z,w,t;\alpha_0, \alpha_1, \alpha_2, \alpha_3, \alpha_4)  
\longrightarrow  
\left(
x+\frac{\alpha_1}{y}, y, z, w,t;\alpha_0+\alpha_1, -\alpha_1, \alpha_2+\alpha_1, \alpha_3, \alpha_4
\right), \\
&s_2: (x,y,z,w,t;\alpha_0, \alpha_1, \alpha_2, \alpha_3, \alpha_4)  
\longrightarrow 
\left(
x, y-\frac{\alpha_2z}{xz-1}, z, w-\frac{\alpha_2x}{xz-1}, t; 
\alpha_0, \alpha_1+\alpha_2, -\alpha_2, \alpha_3+\alpha_2, \alpha_4
\right), \\
&s_3: (x,y,z,w,t;\alpha_0, \alpha_1, \alpha_2, \alpha_3, \alpha_4)  
\longrightarrow 
\left(
x,y,z+\frac{\alpha_3}{w}, w,t; 
\alpha_0, \alpha_1, \alpha_2+\alpha_3, -\alpha_3, \alpha_4+\alpha_3
\right), \\
&s_4: (x,y,z,w,t;\alpha_0, \alpha_1, \alpha_2, \alpha_3, \alpha_4) 
\longrightarrow 
\left(
x,y,z,w-\frac{2\alpha_4}{z}+\frac{t}{z^2},-t;
\alpha_0, \alpha_1, \alpha_2, \alpha_3+2\alpha_4, -\alpha_4
\right),\\
&\psi: (x,y,z,w,t;\alpha_0, \alpha_1, \alpha_2, \alpha_3, \alpha_4) 
\longrightarrow 
\left(
\frac{z}{t}, tw, tx, \frac{y}{t},t; \alpha_4, \alpha_3, \alpha_2, \alpha_1, \alpha_0
\right). 
\end{align*}
{\bf Remark}\quad 
We correct the definition of $s_4$ by Sasano \cite{Sasano-2}, 
which is given by 
$$
s_4: \quad (x,y,z,w,t)\longrightarrow 
\left(
x,y,z,w-\frac{2\alpha_4}{w}+\frac{t}{z^2}, -t
\right). 
$$

\subsection{The properties of the B\"acklund transformations}

\begin{proposition}
\label{prop:D5^{2}-transformation}
{\it
{\rm (0)}\quad For $D_5^{(2)}(\alpha_j)_{0\leq j \leq 4},$ there exists no solution such that $x\equiv 0.$  
\newline
{\rm (1)}\quad If $y\equiv 0$ for $D_5^{(2)}(\alpha_j)_{0\leq j \leq 4},$ then $\alpha_1=0$ or $t+2z(zw+\alpha_3)=0.$ 
\newline
{\rm (2)}\quad If $xz\equiv 1$ for $D_5^{(2)}(\alpha_j)_{0\leq j \leq 4},$ then $\alpha_2=0.$ 
\newline
{\rm (3)}\quad If $w\equiv 0$ for $D_5^{(2)}(\alpha_j)_{0\leq j \leq 4},$ then $\alpha_3=0$ or $1+2x(xy+\alpha_1)=0.$ 
\newline
{\rm (4)}\quad For $D_5^{(2)}(\alpha_j)_{0\leq j \leq 4},$ there exists no solution such that $z\equiv 0.$  
}
\end{proposition}
By Proposition \ref{prop:D5^{2}-transformation}, 
we have to consider the infinite solution such that $x\equiv \infty,$ or $z\equiv \infty.$ 
For this purpose, 
we first treat the case where $y\equiv 0$ and $\alpha_1\neq 0,$ and the case where $w\equiv 0$ and $\alpha_3\neq 0.$

\begin{proposition}
\label{prop:D_5^2-y=0-alpha_1neq0}
{\it
Suppose that $y\equiv 0$ and $\alpha_1\neq 0$ for $D_5^{(2)}(\alpha_j)_{0\leq j \leq 4}.$ 
It then follows that $\alpha_4\neq 0$ and either of the following occurs:
\newline
{\rm (1)}\quad $\alpha_0+\alpha_1=\alpha_3+\alpha_4=0$ and 
$$
x=-\frac{1}{2\alpha_1}, \,y=0, \,z=\frac{t}{2\alpha_4}, \,w=0, 
$$
{\rm (2)}\quad $\alpha_0=1/2, \,\alpha_1=-1/2$ and 
$$
x=1+\frac{4\alpha_4(\alpha_3+\alpha_4)}{t}, \,y=0, \,z=\frac{t}{2\alpha_4}, \,w=-\frac{2\alpha_4(\alpha_3+\alpha_4)}{t}.
$$
}
\end{proposition}

Proposition \ref{prop:D_5^2-y=0-alpha_1neq0} implies that 
we have to consider the infinite solution such that $x\equiv \infty.$

\begin{proposition}
\label{prop:D_5^2-w=0-alpha_3neq0}
{\it
Suppose that $w\equiv 0$ and $\alpha_3\neq 0$ for $D_5^{(2)}(\alpha_j)_{0\leq j \leq 4}.$ 
It then follows that $\alpha_0\neq 0,$ and either of the following occurs: 
\newline
{\rm (1)}\quad $\alpha_0+\alpha_1=\alpha_3+\alpha_4=0$ and 
$$
x=\frac{1}{2\alpha_0}, \, y=0, z=-\frac{t}{2\alpha_3}, w=0,
$$
{\rm (2)}\quad $\alpha_3=-1/2, \,\alpha_4=1/2$ and 
$$
x=\frac{1}{2\alpha_0}, y=-2\alpha_0(\alpha_0+\alpha_1), z=t+4\alpha_0(\alpha_0+\alpha_1), w=0.
$$
}
\end{proposition}
%%%%%%%%%%%%%%
Proposition \ref{prop:D_5^2-w=0-alpha_3neq0} implies that 
we have to consider the infinite solution such that $z\equiv \infty.$

\subsection{The infinite solution}
\subsubsection{The case where $x\equiv \infty$}

In order to determine the solution such that $x\equiv\infty,$ 
following Sasano \cite{Sasano-2}, 
we introduce the coordinate transformation $r_1,$ 
which is given by 
$$
r_1:\quad x_1=1/x, \,y_1=-(xy+\alpha_1)x, \,z_1=z, \,w_1=w.
$$

\begin{proposition}
\label{prop:D5-x=inf}
{\it
Suppose that $x\equiv \infty$ for $D_5^{(2)}(\alpha_j)_{0\leq j \leq 4}.$ 
It then follows that 
$\alpha_4\neq 0$ and either of the following occurs:
\newline
{\rm (1)}\quad $\alpha_0=\alpha_3+\alpha_4=0$
$$
x_1=0, \,y_1=\frac12, \,z_1=\frac{t}{2\alpha_4}, \, w_1=0,
$$
that is, $x=\infty, \,y=0, \,z=t/\{2\alpha_4\}, \,w=0,$
\newline
{\rm (2)}\quad $\alpha_0=0, \,\alpha_1=1/2$ and 
$$
x_1=0, y_1=\frac12+\frac{2\alpha_4(\alpha_3+\alpha_4)}{t}, \,z_1=\frac{t}{2\alpha_4}, \,w_1=-\frac{2\alpha_4(\alpha_3+\alpha_4)}{t},
$$
that is, $x=\infty, \,y=0, \,z=t/\{2\alpha_4\}, \,w=-2\alpha_4(\alpha_3+\alpha_4)t^{-1}.$
}
\end{proposition}

\begin{proof}
For $D_5^{(2)}(\alpha_j)_{0\leq j \leq 4},$ 
$r_1$ transforms the system of $(x,y,z,w)$ into the system of $(x_1, y_1, z_1, w_1),$ 
which is given by
\begin{equation*}
(*)
\begin{cases}
tx_1^{\prime}=2x_1^2y_1+2\alpha_1x_1+t+2\alpha_0x_1-x_1^2+2z_1(z_1w_1+\alpha_3),  \\
ty_1^{\prime}=-2x_1y_1^2-2\alpha_0y_1+2x_1y_1-2\alpha_1y_1+\alpha_1, \\
tz_1^{\prime}=2z_1^2w_1-z_1^2+(1-2\alpha_4)z_1+t+2y_1z_1^2, \\
tw_1^{\prime}=-2z_1w_1^2+2z_1w_1-(1-2\alpha_4)w_1+\alpha_3-2y_1(2z_1w_1+\alpha_3).
\end{cases}
\end{equation*}
Setting $x_1\equiv 0,$ we have 
\begin{equation}
\label{eqn:x=inf-(1)}
t+2z_1^2w_1+2\alpha_3z_1=0,
\end{equation}
which implies that 
\begin{equation}
\label{eqn:x=inf-(2)}
1+4z_1w_1z_1^{\prime}+2z_1^2w_1^{\prime}+2\alpha_3z_1^{\prime}=0.
\end{equation}
From $(*),$ (\ref{eqn:x=inf-(1)}) and (\ref{eqn:x=inf-(2)}), 
it then follows that $z_1w_1=-\alpha_3-\alpha_4.$ Thus, considering (\ref{eqn:x=inf-(1)}), 
we find that $\alpha_4\neq 0$ and  
\begin{equation}
\label{eqn:x=inf-(3)}
z_1=\frac{t}{2\alpha_4}, \,w_1=-\frac{2\alpha_4(\alpha_3+\alpha_4)}{t}.
\end{equation}
\par
Substituting (\ref{eqn:x=inf-(3)}) into 
$$
tz_1^{\prime}=2z_1^2w_1-z_1^2+(1-2\alpha_4)z_1+t+2y_1z_1^2, 
$$
we obtain 
\begin{equation}
\label{eqn:x=inf-(4)}
y_1=\frac12+\frac{2\alpha_4(\alpha_3+\alpha_4)}{t}.
\end{equation}
\par
Moreover, substituting (\ref{eqn:x=inf-(3)}), (\ref{eqn:x=inf-(4)}) and $x_1\equiv 0$ into 
$$
ty_1^{\prime}=-2x_1y_1^2-2\alpha_0y_1+2x_1y_1-2\alpha_1y_1+\alpha_1, 
$$
we have 
$$
\{1-2(\alpha_0+\alpha_1)\}\frac{2\alpha_4(\alpha_3+\alpha_4)}{t}-\alpha_0=0,  \,\,(\alpha_4\neq 0),
$$
which proves the proposition.
\end{proof}

{\bf Remark} 
In both cases of Proposition \ref{prop:D5-x=inf}, 
we can express the solution by 
$$
x_1\equiv 0, \, 
y_1=\frac12+\frac{2\alpha_4(\alpha_3+\alpha_4)}{t}, \,
z_1=\frac{t}{2\alpha_4}, \,
w_1=-\frac{2\alpha_4(\alpha_3+\alpha_4)}{t},
$$
and 
$$
x=\infty, \, y=0, \,z=\frac{t}{2\alpha_4}, \,w=-\frac{2\alpha_4(\alpha_3+\alpha_4)}{t}.
$$

\subsubsection{The case where $z\equiv \infty$}

In order to determine the solution such that $z\equiv\infty,$ 
following Sasano \cite{Sasano-2}, 
we introduce the coordinate transformation $r_3,$ 
which is given by 
$$
r_3: \quad x_3=x, \,y_3=y, \, z_3=1/z, \,w_3=-z(zw+\alpha_3).
$$

\begin{proposition}
\label{prop:D5-z=inf}
{\it
Suppose that $z\equiv \infty$ for $D_5^{(2)}(\alpha_j)_{0\leq j \leq 4}.$ 
It then follows that $\alpha_0\neq 0$ and either of the following occurs:
\newline
{\rm (1)}\quad $\alpha_0+\alpha_1=\alpha_4=0$ and 
$$
x_3=\frac{1}{2\alpha_0}, \,y_3=0, \, z_3=0, \,w_3=\frac{t}{2},
$$
that is, $x=1/\{2\alpha_0\}, \,y=0, \,z=\infty, \, w=0,$ 
\newline
{\rm (2)}\quad $\alpha_3=1/2, \,\alpha_4=0$ and 
$$
x_3=\frac{1}{2\alpha_0}, \, y_3=-2\alpha_0(\alpha_0+\alpha_1), \, z_3=0, \,w_3=\frac{t}{2}+2\alpha_0(\alpha_0+\alpha_1),
$$
that is, $x=1/\{2\alpha_0\}, \,y=-2\alpha_0(\alpha_0+\alpha_1), \, z=\infty, \,w=0.$
}
\end{proposition}

\begin{proof}
For $D_5^{(2)}(\alpha_j)_{0\leq j \leq 4},$ 
$r_3$ transforms the system of $(x,y,z,w)$ into the system of $(x_1, y_1, z_1, w_1),$ 
which is given by
\begin{equation*}
(***)
\begin{cases}
tx_3^{\prime}=2x_3^2y_3-tx_3^2-2\alpha_0x_3+1+2x_3^2w_3, \\
ty_3^{\prime}=-2x_3y_3^2+2tx_3y_3+2\alpha_0y_3+\alpha_1t-2w_3(2x_3y_3+\alpha_1), \\
tz_3^{\prime}=2z_3^2w_3+2\alpha_3z_3+1-(1-2\alpha_4)z_3-tz_3^2+2x_3(x_3y_3+\alpha_1), \\
tw_3^{\prime}=-2z_3w_3^2-2tz_3w_3+(1-2\alpha_3-2\alpha_4)w_3+\alpha_3t. 
\end{cases}
\end{equation*}
Setting $z_3\equiv 0,$ we have 
\begin{equation}
\label{eqn:z=inf-(1)}
1+2x_3^2y_3+2\alpha_1x_3=0,
\end{equation}
which implies that 
\begin{equation}
\label{eqn:z=inf-(2)}
4x_3y_3x_3^{\prime}+2x_3^2y_3^{\prime}+2\alpha_1x_3^{\prime}=0
\end{equation}
From $(***),$ (\ref{eqn:z=inf-(1)}) and (\ref{eqn:z=inf-(2)}), 
it then follows that $x_3y_3=-\alpha_0-\alpha_1.$
Thus, considering (\ref{eqn:z=inf-(1)}), 
we find that $\alpha_0\neq 0$ and  
\begin{equation}
\label{eqn:z=inf-(3)}
x_3=\frac{1}{2\alpha_0}, \,y_3=-2\alpha_0(\alpha_0+\alpha_1).
\end{equation}
\par
Substituting (\ref{eqn:z=inf-(3)}) into 
$$
tx_3^{\prime}=2x_3^2y_3-tx_3^2-2\alpha_0x_3+1+2x_3^2w_3, 
$$
we obtain 
\begin{equation}
\label{eqn:z=inf-(4)}
w_3=\frac{t}{2}+2\alpha_0(\alpha_0+\alpha_1).
\end{equation}
\par
Moreover, substituting (\ref{eqn:z=inf-(3)}), (\ref{eqn:z=inf-(4)}) and $z_3\equiv 0$ into
$$
tw_3^{\prime}=-2z_3w_3^2-2tz_3w_3+(1-2\alpha_3-2\alpha_4)w_3+\alpha_3t,
$$
we have 
$$
-\alpha_4t+2\alpha_0(\alpha_0+\alpha_1)(1-2\alpha_3-2\alpha_4)=0,\,\,(\alpha_0\neq 0),
$$
which proves the proposition. 
\end{proof}

{\bf Remark} In both cases of Proposition \ref{prop:D5-z=inf}, 
we can express the solution by 
$$
x_3=\frac{1}{2\alpha_0}, \,y_3=-2\alpha_0(\alpha_0+\alpha_1), \,z_3=0, \,w_3=\frac{t}{2}+2\alpha_0(\alpha_0+\alpha_1),
$$
and 
$$
x=\frac{1}{2\alpha_0}, \,y=-2\alpha_0(\alpha_0+\alpha_1), \,z=\infty, \,y=0.
$$

\subsubsection{The case where $x=z\equiv \infty$}

In order to determine the solution such that $x=z=\infty,$ 
we define $r_5=r_1r_3=r_3r_1$ by 
$$
r_5: \quad x_5=1/x, \,y_5=-(xy+\alpha_1)x, \,z_5=1/z, \, w_5=-(zw+\alpha_3)z.
$$

\begin{proposition}
\label{prop:D5-x,z=inf}
{\it
Suppose that $x=z\equiv \infty$ for $D_5^{(2)}(\alpha_j)_{0\leq j \leq 4}.$ 
Then, $\alpha_0=\alpha_4=0$ and 
$$
x_5=0, \, y_5=1/2, \,z_5=0, w_5=t/2, 
$$
that is, $x=\infty, \,y=0, \,z=\infty, \, w=0.$ 
}
\end{proposition}

\begin{proof}
For $D_5^{(2)}(\alpha_j)_{0\leq j \leq 4},$ 
$r_5$ then transforms the system of $(x,y, z, w)$ 
into the system of $(x_5, y_5, z_5, w_5),$ which is given by 
\begin{equation*}
\begin{cases}
tx_5^{\prime}=2x_5^2y_5+-x_5^2+(2\alpha_0+2\alpha_1)x_5+t-2w_5,  \\
ty_5^{\prime}=-2x_5y_5^2-(2\alpha_0+2\alpha_1)y_5+2x_5y_5+\alpha_1,  \\
tz_5^{\prime}=2z_5^2w_5-(1-2\alpha_3-2\alpha_4)z_5+1-tz_5^2-2y_5,  \\
tw_5^{\prime}=-2z_5w_5^2+(1-2\alpha_3-2\alpha_4)w_5+2tz_5w_5+\alpha_3t. 
\end{cases}
\end{equation*}
Setting $x_5=z_5\equiv 0$ in 
\begin{equation*}
\begin{cases}
tx_5^{\prime}=2x_5^2y_5+-x_5^2+(2\alpha_0+2\alpha_1)x_5+t-2w_5,  \\
tz_5^{\prime}=2z_5^2w_5-(1-2\alpha_3-2\alpha_4)z_5+1-tz_5^2-2y_5, 
\end{cases}
\end{equation*}
we obtain $w_5=t/2, \,y_5=1/2.$ 
\par
Substituting $y_5=1/2, \,w_5=t/2$ into 
\begin{equation*}
\begin{cases}
ty_5^{\prime}=-2x_5y_5^2-(2\alpha_0+2\alpha_1)y_5+2x_5y_5+\alpha_1,\\
tw_5^{\prime}=-2z_5w_5^2+(1-2\alpha_3-2\alpha_4)w_5+2tz_5w_5+\alpha_3t,
\end{cases}
\end{equation*}
we have $\alpha_0=\alpha_4=0.$ 

\end{proof}

\subsection{The B\"acklund transformations and the infinite solutions}

\begin{proposition}
{\it
Suppose that 
$D_5^{(2)}(\alpha_j)_{0\leq j \leq 4}$ has an infinite solution such that $x\equiv \infty.$ 
The actions of the B\"acklund transformations are then as follows:
\newline
{\rm (0)}\quad $s_0(\infty, 0, t/\{2\alpha_4\}, -2\alpha_4(\alpha_3+\alpha_4)t^{-1})=(\infty, 0, t/\{2\alpha_4\}, -2\alpha_4(\alpha_3+\alpha_4)t^{-1}), $ 
\newline
{\rm (1)-(i)}\quad if $\alpha_1\neq 0,$ 
\begin{equation*}
s_1(\infty, 0, t/\{ 2\alpha_4 \}, -2\alpha_4 (\alpha_3+\alpha_4) t^{-1})= 
(1/\{ 2\alpha_1\}+2\alpha_4(\alpha_3+\alpha_4)/\{ \alpha_1t \}, 0, t/\{ 2\alpha_4 \}, -2\alpha_4(\alpha_3+\alpha_4)t^{-1}), 
\end{equation*}
{\rm (1)-(ii)}\quad if $\alpha_1=0,$ 
\begin{equation*}
s_1(\infty, 0, t/\{2\alpha_4\}, -2\alpha_4(\alpha_3+\alpha_4)t^{-1})=(\infty, 0, t/\{2\alpha_4\}, -2\alpha_4(\alpha_3+\alpha_4)t^{-1}),
\end{equation*}
{\rm (2)}\quad $s_2(\infty, 0, t/\{2\alpha_4\}, -2\alpha_4(\alpha_3+\alpha_4)t^{-1})=(\infty, 0, t/\{2\alpha_4\}, -2\alpha_4(\alpha_2+\alpha_3+\alpha_4)t^{-1}), $ 
\newline
{\rm (3)-(i)}\quad 
if $\alpha_3+\alpha_4\neq 0,$ 
\begin{equation*}
s_3(\infty, 0, t/\{2\alpha_4\}, -2\alpha_4(\alpha_3+\alpha_4)t^{-1})= (\infty, 0, t/\{2(\alpha_3+\alpha_4)\}, -2\alpha_4(\alpha_3+\alpha_4)t^{-1}),
\end{equation*}
{\rm (3)-(ii)}\quad 
if $\alpha_3+\alpha_4= 0,$ 
\begin{equation*}
s_3(\infty, 0, t/\{2\alpha_4\}, -2\alpha_4(\alpha_3+\alpha_4)t^{-1})=
(\infty, 0, \infty, 0),
\end{equation*}
{\rm (4)}\quad $s_4(\infty, 0, t/\{2\alpha_4\}, -2\alpha_4(\alpha_3+\alpha_4)t^{-1})=(\infty, 0, t/\{2(-\alpha_4)\}, -2(-\alpha_4)(\alpha_3+\alpha_4)t^{-1}),$
\newline
{\rm (5)}\quad $\psi(\infty, 0, t/\{2\alpha_4\}, -2\alpha_4(\alpha_3+\alpha_4)t^{-1})=(1/\{2\alpha_4\}, -2\alpha_4(\alpha_3+\alpha_4), \infty, 0).$ 
}
\end{proposition}

\begin{proposition}
{\it
Suppose that 
$D_5^{(2)}(\alpha_j)_{0\leq j \leq 4}$ has an infinite solution such that $z\equiv \infty.$ 
The actions of the B\"acklund transformations are then as follows:
\newline
{\rm (0)}\quad $s_0(1/\{2\alpha_0\}, -2\alpha_0(\alpha_0+\alpha_1), \infty, 0)=(-1/\{2\alpha_0\}, 2\alpha_0(\alpha_0+\alpha_1), \infty, 0),$ 
\newline
{\rm (1)-(i)}\quad if $\alpha_0+\alpha_1\neq 0,$ 
$$
s_1(1/\{2\alpha_0\}, -2\alpha_0(\alpha_0+\alpha_1), \infty, 0)=(1/\{2(\alpha_0+\alpha_1)\}, -2\alpha_0(\alpha_0+\alpha_1), \infty, 0),
$$ 
{\rm(1)-(ii)}\quad if $\alpha_0+\alpha_1= 0,$
$$
s_1(1/\{2\alpha_0\}, -2\alpha_0(\alpha_0+\alpha_1), \infty, 0)=(\infty, 0, \infty, 0),
$$
{\rm (2)}\quad $s_2(1/\{2\alpha_0\}, -2\alpha_0(\alpha_0+\alpha_1), \infty, 0)=(1/\{2\alpha_0\}, -2\alpha_0(\alpha_0+\alpha_1+\alpha_2), \infty, 0),$
\newline
{\rm (3)-(i)}\quad if $\alpha_3\neq 0,$ 
$$
s_3(1/\{2\alpha_0\}, -2\alpha_0(\alpha_0+\alpha_1), \infty, 0)=(1/\{2\alpha_0\}, -2\alpha_0(\alpha_0+\alpha_1), t/\{2\alpha_3\}+2\alpha_0(\alpha_0+\alpha_1)/\alpha_3, 0), 
$$
{\rm (3)-(ii)}\quad if $\alpha_3= 0,$ 
$$
s_3(1/\{2\alpha_0\}, -2\alpha_0(\alpha_0+\alpha_1), \infty, 0)=(1/\{2\alpha_0\}, -2\alpha_0(\alpha_0+\alpha_1), \infty, 0),
$$
{\rm (4)}\quad $s_4(1/\{2\alpha_0\}, -2\alpha_0(\alpha_0+\alpha_1), \infty, 0)=(1/\{2\alpha_0\}, -2\alpha_0(\alpha_0+\alpha_1), \infty, 0),$
\newline
{\rm (5)}\quad $\psi(1/\{2\alpha_0\}, -2\alpha_0(\alpha_0+\alpha_1), \infty, 0)=(\infty, 0, t/\{2\alpha_0\}, -2\alpha_0(\alpha_0+\alpha_1)t^{-1}). $
}
\end{proposition}

\begin{proposition}
{\it
Suppose that 
$D_5^{(2)}(\alpha_j)_{0\leq j \leq 4}$ has an infinite solution such that $x=z\equiv \infty.$ 
The actions of the B\"acklund transformations are then as follows:
\newline
{\rm (0)}\quad $s_0(\infty, 0, \infty, 0)=(\infty, 0, \infty, 0),$ 
\newline
{\rm (1)-(i)}\quad if $\alpha_1\neq 0,$
$$
s_1(\infty, 0, \infty, 0)=(1/\{2\alpha_1\}, 0, \infty, 0),
$$
{\rm (1)-(ii)}\quad if $\alpha_1=0,$
$$
s_1(\infty, 0, \infty, 0)=(\infty, 0, \infty, 0), 
$$
{\rm (2)}\quad $s_2(\infty, 0, \infty, 0)=(\infty, 0, \infty, 0),$ 
\newline
{\rm (3)-(i)}\quad if $\alpha_3\neq 0,$
$$
s_3(\infty, 0, \infty, 0)=(\infty, 0,t/\{2\alpha_3\}, 0),
$$
{\rm (3)-(ii)}\quad if $\alpha_3= 0,$
$$
s_3(\infty, 0, \infty, 0)=(\infty, 0, \infty, 0), 
$$
{\rm (4)}\quad $s_4(\infty, 0, \infty, 0)=(\infty, 0, \infty, 0), $
\newline
{\rm (5)}\quad $\psi(\infty, 0, \infty, 0)=(\infty, 0, \infty, 0). $

}
\end{proposition}

\subsection{Main theorem for $D_5^{(2)}(\alpha_j)_{0\leq j \leq 4}$}

Sasano \cite{Sasano-2} proved that $D_4^{(1)}(\alpha_j)_{0\leq j \leq 4}$ is equivalent to $D_5^{(2)}(\alpha_j)_{0\leq j \leq 4}.$

\begin{proposition}
\label{prop:D_4=D_5}
{\it 
Suppose that $(x,y,z,w)$ is a solution of $D_4^{(1)}(\alpha_j)_{0\leq j \leq 4},$ 
and 
\begin{align*}
&X=\frac{1}{x}, \,Y=-(xy+\alpha_1)x, \,Z=\frac{1}{z}, \,W=-(zw+\alpha_3)z, \\
&A_0=\frac{\alpha_0-\alpha_1}{2}, \,A_1=\alpha_1, \,A_2=\alpha_2, \,A_3=\alpha_3, \, A_4=\frac{\alpha_4-\alpha_3}{2}.
\end{align*}
$(X,Y,Z,W)$ is then a solution of $D_5^{(2)}(A_j)_{0\leq j \leq 4}.$
}
\end{proposition}

By Theorem \ref{thm:D4} and Proposition \ref{prop:D_4=D_5}, 
we obtain the following theorem:

\begin{theorem}
\label{thm:D_5}
{\it
Suppose that for $D_5^{(2)}(\alpha_j)_{0\leq j \leq 4},$ 
there exists a rational solution. 
By some B\"acklund transformations, 
the parameters and solution can then be transformed so that 
$$
\alpha_0=\alpha_3+\alpha_4=0, \,\alpha_1, \alpha_4\neq 0 \,\,{\it and } \,\,(x,y,z,w)=(\infty, 0, t/\{2\alpha_4\}, 0). 
$$
Moreover, $D_5^{(2)}(\alpha_j)_{0\leq j \leq 4}$ has a rational solution 
if and only if one of the following occurs: 
\begin{align*}
&{\rm (1)} &  &2\alpha_0\in\Z, & &2\alpha_3+2\alpha_4\in \Z, & &2\alpha_0\equiv 2\alpha_3+2\alpha_4  & &\mathrm{ mod} \,\,2,   \\
%\newline
&{\rm (2)} & &2\alpha_0\in\Z, & &2\alpha_4\in \Z,               &  &2\alpha_0\equiv 2\alpha_4 & &\mathrm{ mod} \,\,2,  \\
%\newline
&{\rm (3)} & &2\alpha_0+2\alpha_1\in\Z, & &2\alpha_3+2\alpha_4\in \Z,  &  &2\alpha_0+2\alpha_1\equiv 2\alpha_3+2\alpha_4 & &\mathrm{ mod} \,\,2,  \\
%\newline
&{\rm (4)} & &2\alpha_0+2\alpha_1\in\Z, & &2\alpha_4\in \Z,               &   &2\alpha_0+2\alpha_1\equiv 2\alpha_4 & &\mathrm{ mod} \,\,2,  \\
&{\rm (5)} & &2\alpha_0\in\Z, & &2\alpha_1 \in \Z,   &    &2\alpha_1 \equiv 1 & &\mathrm{ mod} \,\,2,  \\
%\newline
&{\rm (6)}  &  &2\alpha_3\in\Z, & &2\alpha_4\in \Z,                      &     &2\alpha_3 \equiv 1 & &\mathrm{ mod} \,\,2.
\end{align*}
}
\end{theorem}

By $s_1,$ we can obtain the following corollary:

\begin{corollary}
{\it
Suppose that for $D_5^{(2)}(\alpha_j)_{0\leq j \leq 4},$ 
there exists a rational solution. 
By some B\"acklund transformations, 
the parameters and solution can then be transformed so that 
$$
\alpha_0+\alpha_1=\alpha_3+\alpha_4=0, \,\alpha_0,\alpha_4\neq 0 \,\,{\it and } \,\,(x,y,z,w)=(1/\{2\alpha_0\}, 0, t/\{2\alpha_4\}, 0). 
$$
}
\end{corollary}

\end{document}